\documentclass[english]{article}
\usepackage[T1]{fontenc}
\usepackage[latin9]{inputenc}
\usepackage{color}
\usepackage{babel}
\usepackage{mathrsfs}
\usepackage{amsmath}
\usepackage{amsthm}
\usepackage{amssymb}
\usepackage[unicode=true,pdfusetitle,
 bookmarks=true,bookmarksnumbered=false,bookmarksopen=false,
 breaklinks=true,pdfborder={0 0 0},pdfborderstyle={},backref=false,colorlinks=true]
 {hyperref}

\makeatletter
\theoremstyle{plain}
\newtheorem{thm}{\protect\theoremname}
\theoremstyle{definition}
\newtheorem{defn}[thm]{\protect\definitionname}
\theoremstyle{definition}
\newtheorem{example}[thm]{\protect\examplename}
\theoremstyle{remark}
\newtheorem{rem}[thm]{\protect\remarkname}
\theoremstyle{plain}
\newtheorem{prop}[thm]{\protect\propositionname}
\theoremstyle{plain}
\newtheorem{lem}[thm]{\protect\lemmaname}
\theoremstyle{plain}
\newtheorem{cor}[thm]{\protect\corollaryname}
\theoremstyle{plain}
\newtheorem{assumption}[thm]{\protect\assumptionname}

\DeclareMathSymbol{\shortminussymb}{\mathbin}{AMSa}{"39}

\makeatother

\providecommand{\assumptionname}{Assumption}
\providecommand{\corollaryname}{Corollary}
\providecommand{\definitionname}{Definition}
\providecommand{\examplename}{Example}
\providecommand{\lemmaname}{Lemma}
\providecommand{\propositionname}{Proposition}
\providecommand{\remarkname}{Remark}
\providecommand{\theoremname}{Theorem}

\begin{document}
\title{Poincaré inequalities for Markov chains: a meeting with Cheeger, Lyapunov
and Metropolis}
\author{Christophe Andrieu, Anthony Lee, Sam Power, Andi Q. Wang\\
\\
School of Mathematics, University of Bristol}

\maketitle

\begin{abstract}
We develop a theory of weak Poincaré inequalities to characterize
convergence rates of ergodic Markov chains. Motivated by the application
of Markov chains in the context of algorithms, we develop a relevant
set of tools which enable the practical study of convergence rates
in the setting of Markov chain Monte Carlo methods, but also well
beyond.

\newpage{}

\tableofcontents{}
\end{abstract}
\global\long\def\dif{\mathrm{d}}%
\global\long\def\Var{\mathrm{Var}}%
\global\long\def\R{\mathbb{R}}%
\global\long\def\X{\mathcal{X}}%
\global\long\def\calE{\mathcal{E}}%
\global\long\def\E{\mathsf{E}}%
\global\long\def\Ebb{\mathbb{E}}%

\global\long\def\ELL{\mathrm{L}^{2}}%
\global\long\def\osc{\mathrm{osc}}%
\global\long\def\Id{\mathrm{Id}}%
\global\long\def\essup{\mathrm{ess\,sup}}%
 
\global\long\def\shortminus{\mathrm{\shortminussymb}}%
\global\long\def\muess{\mathrm{ess_{\mu}}}%

\newpage{}

\section{Introduction}

This report is the result of a research programme initiated in \cite{ALPW2021}
that aims to understand and develop functional-analytic tools to characterize
the rate of convergence to equilibrium of discrete-time Markov chains.
While analysis of the right-spectral gap of time-reversible Markov
chains is fairly standard and has played an important rôle in the
analysis of Markov chain Monte Carlo (MCMC) algorithms, functional-analytic
results for nonreversible or subgeometrically convergent Markov chains
are scarce. Notable exceptions are \cite{fill1991eigenvalue} and
\cite{diaconis1996nash}, the latter being the closest in spirit to
our work. On the other hand, the characterization of the convergence
to equilibrium of \textit{continuous-time} processes, both reversible
and nonreversible, geometric and subgeometric, is considerably more
developed. Study of subgeometric rates of convergence can be traced
back to \cite{liggett1991l_2}, which was later generalized and developed
in \cite{Rockner2001}, with a general framework relying on \textit{weak
Poincaré inequalities} (WPIs). Further significant contributions to
the analysis of diffusion processes were made by the French school
in the late 2000s -- early 2010s in a series of contributions, for
instance \cite{bakry2008simple,Bakry2010,cattiaux2010functional,Cattiaux2012}.

Beyond the scattered nature of this literature, the continuous-time
scenario possesses a plethora of specific technical difficulties,
which often render it difficult to penetrate for the uninitiated.
On the other hand, while the discrete-time Markov chain setup is indeed
technically simpler, it has its own subtleties and challenges, which
have not thus far been covered in a comprehensive way in the literature.
As such, many of our present results are not merely transpositions
of existing continuous-time results into the discrete-time setting. 

Importantly, the main motivation behind our work being our interest
in MCMC methods -- and more generally algorithms which utilize ergodic
Markov chains -- we address numerous questions not addressed in the
existing literature, concerning for example optimality and comparison
of Markov chains. Our own recent experience shows that these functional-analytic
tools we develop are complementary to the classical drift and minorization
approach, à la Meyn and Tweedie \cite{meyn:tweedie:1993}, which has
proved particularly useful and fruitful in the context of MCMC algorithms.
We provide several concrete examples and applications of our techniques
which are relevant for the analysis of MCMC methods; in particular
we have been able to answer some questions (see, for instance, \cite{ALPW2021}
or Subsection~\ref{subsec:Spectral-gap-of-RWM}) which had eluded
us and others previously.

\subsection{A roadmap}

Beyond an attempt to develop a coherent and self-contained document
on WPIs for Markov chains, we also make a number of novel contributions.

This manuscript can be summarized as follows:
\begin{itemize}
\item Section~\ref{sec:Fundamentals} focuses on definitions of weak Poincaré
inequalities (WPIs) in the discrete-time setting and their immediate
implications. In Subsection~\ref{subsec:fundamentals-Definitions-and-basic}
three equivalent parametrizations of WPIs are discussed in detail
and we summarise their implications for rates of convergence to equilibrium.
In Subsection~\ref{subsec:bounded_to_p}, we connect convergence
for bounded functions in $\ELL$ with convergence of $\mathrm{L}^{p}$
functions. We establish in Subsection~\ref{subsec:fundamentals-Deducing-WPIs-from}
reverse implications: showing that a given rate of convergence implies
the existence of a WPI. In Subsection~\ref{subsec:fundamentals-Bounds-on-the}
we show how WPIs can be used to bound directly the asymptotic variance
of ergodic averages. In Subsection~\ref{subsec:fundamentals-Towards-spectral-interpretations}
we draw links between WPIs and subgeometric rates of convergence with
spectral properties of the operators involved.
\item Section~\ref{sec:Optimal-choices} is dedicated to the notion of
\textit{optimal} WPIs (Subsection~\ref{subsec:optimal-choices-alpha-beta}),
lower bounds on rates of convergence (Subsection~\ref{subsec:optimal-choice-Lower-bounds-on}),
comparison results of the Peskun--Tierney type (Subsection \ref{subsec:optimal-choice-Ordering-of-rates}),
optimal sieve functionals (Subsection~\ref{subsec:optimal-choice-phi})
and a form of duality (Subsection~\ref{subsec:optimal-choice-Duality}).
\item Section~\ref{sec:Establishing-WPIs} develops practical tools for
establishing WPIs in practice. In Subsection~\ref{subsec:Cheeger-meets-Poincar=0000E9}
we generalize Cheeger inequalities for Markov chains to establish
WPIs. In Subsection~\ref{subsec:establish-WPI-WPIs-from-RUPI} we
establish links between $\mu-$irreducibility and the existence of
WPIs via the abstract \textit{RUPI} condition. Subsection~\ref{subsec:establish-WPI-Lyapunov-meets-Poincar=0000E9}
discusses connections between drift and minorization techniques with
Poincaré inequalities. We discuss an alternative strategy to establish
WPIs: a \textit{local} Poincaré inequality for a \textit{restricted}
version of the Markov chain is combined with a drift condition. Finally
in Subsection~\ref{subsec:establish-WPI-Restricted-Markov-chains},
we study how the knowledge of SPIs for restricted versions of a given
Markov chain can be used to deduce WPIs for the unrestricted chain.
\item In Section~\ref{sec:Examples-and-applications} we present applications
of the theory in particular scenarios. In Subsections~\ref{subsec:example-Lower-bounds-pseudo}--\ref{subsec:example-Lower-bounds-RWM-heavy-tail}
we establish lower bounds on the rate of convergence of a type of
pseudo-marginal algorithm and the random walk Metropolis (RWM) algorithm
targeting heavy-tailed distributions. In Subsection~\ref{subsec:Spectral-gap-of-RWM}
we establish dimension dependence of $d^{-1}$ of the spectral gap
of the RWM algorithm for a class of light-tailed target distributions,
effectively providing the first direct proof of this result. This
result is specialized to the Gaussian scenario in Subsection~\ref{subsec:Spectral-gap-for-gaussian}.
In Subsection~\ref{subsec:example-Central-limit-theorems} we show
how our results can be used to establish the existence of a central
limit theorem for ergodic averages.
\item Finally the Appendix contains some deferred proofs and miscellaneous
results omitted from the main body of the text.
\end{itemize}
The highlights of this report will ultimately be turned into standard,
more succinct and focussed manuscripts for specialists.

\subsection{Notation}

We will write $\mathbb{N}=\left\{ 1,2,\dots\right\} $ for the set
of natural numbers, $\mathbb{N}_{0}:=\mathbb{N}\cup\left\{ 0\right\} $,
and $\R_{+}=\left(0,\infty\right)$ for positive real numbers. 

Outside of specific examples, we will be working throughout on a general
measurable space $\left(\E,\mathscr{E}\right)$. 
\begin{itemize}
\item For a set $A\in\mathscr{E}$, its complement in $\E$ is denoted by
$A^{\complement}$. We denote the corresponding indicator function
by $\mathbf{1}_{A}:\E\to\left\{ 0,1\right\} $.
\item We assume that $\left(\E,\mathscr{E}\right)$ is equipped with a probability
measure $\mu$, and write $\ELL\left(\mu\right)$ for the Hilbert
space of (equivalence classes of) real-valued $\mu$--square-integrable
measurable functions with inner product 
\[
\langle f,g\rangle=\int_{\E}f\left(x\right)g\left(x\right)\,\dif\mu\left(x\right)\,,
\]
and corresponding norm $\|\cdot\|_{2,\mu}$, and if there is no ambiguity,
we may just write $\|\cdot\|_{2}$. We write $\ELL_{0}\left(\mu\right)$
for the set of functions $f\in\ELL\left(\mu\right)$ which also satisfy
$\mu(f)=0$.
\item More generally, for $p\in[1,\infty)$, we write $\mathrm{L}^{p}\left(\mu\right)$
for the Banach space of real-valued measurable functions with finite
$p$-norm, $\|f\|_{p}:=\left(\int_{\E}|f|^{p}\,\dif\mu\right)^{1/p}$,
and $\mathrm{L}_{0}^{p}\left(\mu\right)$ for $f\in\mathrm{L}^{p}\left(\mu\right)$
with $\mu\left(f\right)=0$.
\item We assume that the diagonal is measurable in $\mathsf{E}\times\mathsf{E}$,
i.e. $\{(x,x):x\in\mathsf{E}\}\in\mathscr{E}\otimes\mathscr{E}$.
This assumption holds, for instance, on a Polish space endowed with
its Borel $\sigma$-algebra.
\item We write $\mathscr{E}_{+}:=\{A\in\mathscr{E}:\mu(A)>0\}$.
\item For $\mu$ and $\nu$ probability measures on $\left(\E,\mathscr{E}\right)$,
we let $\left\Vert \mu-\nu\right\Vert _{{\rm TV}}:=\sup_{A\in\mathscr{E}}\left|\mu\left(A\right)-\nu\left(A\right)\right|$.
\item For a measurable function $f:\mathsf{E}\to\R$, let $\|f\|_{\osc}:=\mathrm{ess_{\mu}}\sup f-\mathrm{ess}_{\mu}\inf f$.
\item For two probability measures $\mu$ and $\nu$ on $(\E,\mathscr{E})$
we let $\mu\otimes\nu\left(A\times B\right)=\mu\left(A\right)\nu\left(B\right)$
for $A,B\in\mathscr{E}$. For a Markov kernel $P\left(x,\dif y\right)$
on $\E\times\mathscr{E}$, we write for $\bar{A}\in\mathscr{E}\otimes\mathscr{E}$,
the minimal product $\sigma$-algebra, $\mu\otimes P\left(\bar{A}\right)=\int_{\bar{A}}\mu\left(\dif x\right)P\left(x,\dif y\right)$. 
\item A point mass distribution at $x$ will be denoted by $\delta_{x}\left(\dif y\right)$.
\item $\Id:\ELL\left(\mu\right)\to\ELL\left(\mu\right)$ denotes the identity
mapping, $f\mapsto f$. We also use this symbol for the identity ${\rm Id}\colon\mathsf{X}\rightarrow\mathsf{X}$.
\item Given a bounded linear operator $T:\ELL\left(\mu\right)\to\ELL\left(\mu\right)$,
we let $\calE\left(T,f\right)$ be the Dirichlet form defined by $\left\langle \left(\Id-T\right)f,f\right\rangle $
for any $f\in\ELL\left(\mu\right)$.
\item For such an operator $T$, we write $T^{*}$ for its adjoint operator
$T^{*}:\ELL\left(\mu\right)\to\ELL\left(\mu\right)$, which satisfies
$\left\langle f,Tg\right\rangle =\left\langle T^{*}f,g\right\rangle $
for any $f,g\in\ELL\left(\mu\right)$.
\item For such an operator $T$, we denote its spectrum by $\sigma\left(T\right)$.
We denote the spectrum of the restriction of $T$ to $\ELL_{0}\left(\mu\right)$
by $\sigma_{0}(T)$.
\item For a $\mu\shortminus$invariant Markov kernel $T$ we let the right-spectral
gap be
\[
{\rm Gap}_{\mathrm{R}}\left(T\right):=\inf_{g\in\ELL_{0}\left(\mu\right),g\neq0}\frac{\calE\left(T,g\right)}{\left\Vert g\right\Vert _{2}^{2}}\,.
\]
\item For a given $f\in\ELL\left(\mu\right)$, the asymptotic variance is
defined as ${\rm {\rm var}}\left(T,f\right):=\lim_{n\to\infty}n{\rm {\rm var}}\left(\frac{1}{n}\sum_{i=1}^{n}T^{n}f\right)$.
\item We will write $a\wedge b$ to mean the (pointwise) minimum of real-valued
functions $a,b$ and $a\vee b$ for the corresponding maximum. For
$s\in\R$, we will write $\left(s\right)_{+}:=s\vee0$ for the positive
part.
\item $\inf A$ denotes the infimum of set $A\subset\R$ and $\inf\emptyset=\infty$.
\item For a norm $\left|\cdot\right|$, which will always be clear from
the context, we define the closed ball of radius $r$ around $x$
to be
\[
\mathcal{B}(x,r):=\left\{ y\in\mathsf{E}:\left|y-x\right|\leq r\right\} .
\]
\item We adopt the following $\mathcal{O}$ (resp. $\Omega$) notation to
indicate when functions grow no faster than (resp. no slower than)
other functions. For $a\in\mathbb{R}\cup\{\infty\}$
\begin{itemize}
\item If $f(x)\in\mathcal{O}(g(x))$ as $x\to a$, this means $\underset{x\to a}{\lim\sup}\left|\frac{f(x)}{g(x)}\right|<\infty$.
When $a=+\infty$ then we may drop explicit mention of $a$. 
\item If $f(x)\in\Omega(g(x))$ as $x\to a$, this means $\underset{x\to a}{\lim\inf}\left|\frac{f(x)}{g(x)}\right|>0$.
In particular $f\in\mathcal{O}(g)\iff g\in\Omega(f)$. 
\end{itemize}
\end{itemize}

\section{Fundamentals\label{sec:Fundamentals}}

\subsection{Definitions and basic properties\label{subsec:fundamentals-Definitions-and-basic}}

We first give the basic definitions needed in order to define a weak
Poincaré inequality.
\begin{defn}
\label{def:Phi_fn}
\begin{enumerate}
\item \label{enu:def_Phi1}We call a functional $\Phi:\ELL\left(\mu\right)\to\left[0,\infty\right]$
a \textit{sieve functional}, or \textit{sieve}, if for any $f\in\ELL\left(\mu\right)$,
$c>0$, it holds that
\[
\Phi\left(cf\right)=c^{2}\Phi\left(f\right),\quad\left\Vert f-\mu\left(f\right)\right\Vert _{2}^{2}\leqslant\mathfrak{a}\Phi\left(f-\mu\left(f\right)\right),
\]
for a finite constant $\mathfrak{a}:=\sup_{f\in\ELL_{0}\left(\mu\right)\backslash\left\{ 0\right\} }\left\Vert f\right\Vert _{2}^{2}/\Phi\left(f\right)$. 
\item Let $P$ be a $\mu\shortminus$invariant Markov kernel. We say that
a sieve is $P\shortminus$\textit{non-expansive} if $\Phi\left(Pf\right)\leqslant\Phi\left(f\right)$
for $f\in\ELL_{0}\left(\mu\right)$.
\end{enumerate}
\end{defn}

For simplicity and when no ambiguity is possible, we may refer to
a $P\shortminus$non-expansive sieve simply as a sieve.
\begin{example}
\label{exa:osc2}Our main example of a $P$--non-expansive sieve,
for any $P$, is $\Phi=\left\Vert \cdot\right\Vert _{{\rm osc}}^{2}$,
with $\mathfrak{a}\leq1$.
\end{example}

There are two ways to parameterize weak Poincaré inequalities for
$P$, which are equivalent under a mild assumption.
\begin{defn}
\label{def:WPI}We say that a $\mu\shortminus$reversible kernel $T$
satisfies a $\left(\Phi,\alpha\right)\shortminus$\textit{weak Poincaré
inequality}, abbreviated $\left(\Phi,\alpha\right)\shortminus$WPI,
if for a sieve $\Phi$ and a decreasing function $\alpha:\left(0,\infty\right)\to[0,\infty)$,
\begin{equation}
\left\Vert f\right\Vert _{2}^{2}\le\alpha\left(r\right)\calE\left(T,f\right)+r\Phi\left(f\right),\quad\forall r>0,f\in\ELL_{0}(\mu).\label{eq:WPI}
\end{equation}
Secondly, using the same notation, we can parameterize in terms of
$\beta$: we say that a $\left(\Phi,\beta\right)\shortminus$WPI holds
if:
\begin{equation}
\left\Vert f\right\Vert _{2}^{2}\le s\calE\left(T,f\right)+\beta\left(s\right)\Phi\left(f\right),\quad\forall s>0,f\in\ELL_{0}\left(\mu\right),\label{eq:beta-WPI}
\end{equation}
where $\beta:\left(0,\infty\right)\to[0,\infty)$ is a decreasing
function with $\beta\left(s\right)\to0$ as $s\to\infty$.
\end{defn}

If $T$ satisfies a $(\Phi,\alpha)$-WPI or a $(\Phi,\beta)$-holds
but the specific $\alpha$ or $\beta$ are not relevant, we may say
that a $\Phi$-WPI holds.

In practice, we are interested in bounding the convergence to equilibrium
of a given $\mu$-invariant Markov kernel, $P$. To obtain such bounds,
in the framework of Definition~\ref{def:WPI}, we will take $T=P^{*}P$,
or if $P$ is $\mu$-reversible, we may take directly $T=P$. 
\begin{rem}
\label{rem:nonrev_WPI}Given a general $\mu$-invariant Markov kernel
$T$ (which is not necessarily reversible), one can still define a
WPI for $T$, namely the requirement that (\ref{eq:WPI}) holds for
our general kernel $T$. However, it is enough to define (\ref{eq:WPI})
only for reversible kernels, since 
\begin{align*}
\calE\left(T,f\right) & =\left\langle \left(\Id-T\right)f,f\right\rangle \\
 & =\left\langle \left(\Id-\left(T+T^{*}\right)/2\right)f,f\right\rangle \\
 & =\calE\left(\left(T+T^{*}\right)/2,f\right),
\end{align*}
due to the fact that $\left(T-T^{*}\right)/2$ is antisymmetric, and
we are considering real-valued $f$. Since the kernel $\left(T+T^{*}\right)/2$
is reversible, it is thus sufficient to consider WPIs for reversible
kernels.
\end{rem}

For any decreasing function $F\colon\mathbb{R}_{+}\rightarrow\mathbb{R}$
we let $F^{\shortminus}:\R\to\left[0,\infty\right]$ given by $F^{\shortminus}\left(x\right):=\inf\left\{ y>0\colon F\left(y\right)\leqslant x\right\} $,
for $x\in\R$, be its generalized inverse. The following proposition
shows that one can straightforwardly move between the two formulations
of WPIs.
\begin{prop}
\label{prop:a-b-WPI-corres}Let $P$ be a Markov kernel on $\big(\mathsf{E},\mathscr{E}\big)$,
$\Phi$ be a sieve, and $\mathfrak{a}:=\sup_{f\in\ELL_{0}\left(\mu\right)\backslash\left\{ 0\right\} }\left\Vert f\right\Vert _{2}^{2}/\Phi\left(f\right)$. 
\begin{enumerate}
\item \label{enu:alphavsbeta-alpha}If a $\left(\Phi,\alpha\right)\shortminus$\textup{WPI}
holds with $\alpha\left(r\right)=0$ for $r\geqslant\mathfrak{a}$,
then a $\left(\Phi,\beta\right)\shortminus$\textup{WPI, with} $\beta:=\alpha^{\shortminus}$
on $\left(0,\infty\right)$, holds and for any $r,s\geqslant0$,
\begin{enumerate}
\item $\alpha^{\shortminus}\circ\alpha\left(r\right)\leqslant r$ with equality
when $\alpha$ is strictly decreasing; 
\item $s\leqslant\alpha\circ\alpha^{\shortminus}\left(s\right)$ if $\alpha$
is right continuous;
\item $\beta\leqslant\mathfrak{a}$.
\end{enumerate}
\item \label{enu:alphavsbeta-beta}If a $\left(\Phi,\beta\right)\shortminus$\textup{WPI}
holds with $\beta\leqslant\mathfrak{a}$, then a $\left(\Phi,\alpha\right)\shortminus$\textup{WPI}
holds, with $\alpha:=\beta^{\shortminus}$ on $\left(0,\infty\right)$,
and for any $r,s>0$,
\begin{enumerate}
\item $\beta^{\shortminus}\circ\beta\left(s\right)\leqslant s$; 
\item $r\leqslant\beta\circ\beta^{\shortminus}\left(r\right)$ if $\beta$
is right continuous;
\item $\alpha\left(r\right)=0$ for $r\geqslant\mathfrak{a}$.
\end{enumerate}
\item \label{enu:alpha-alpha-minus-identity}If $\alpha$ in \ref{enu:alphavsbeta-alpha}
(resp. $\beta$ in \ref{enu:alphavsbeta-beta}) is right continuous
then $\left(\alpha^{\shortminus}\right)^{\shortminus}=\alpha$ (resp.
$\left(\beta^{\shortminus}\right)^{\shortminus}=\beta$); that is,
the two parametrizations are equivalent.
\end{enumerate}
\end{prop}

\begin{proof}
Statement \ref{enu:alphavsbeta-alpha}. Assume that a $\left(\Phi,\alpha\right)\shortminus$WPI
holds, let $s>0$ and $\mathfrak{R}\left(s\right):=\left\{ r>0:\alpha\left(r\right)\leqslant s\right\} \neq\emptyset$,
where the nonemptiness follows from the assumption on $\alpha$. Then
for any $r\in\mathfrak{R}\left(s\right)$, it holds that
\[
\left\Vert f\right\Vert _{2}^{2}\leqslant s\calE\left(T,f\right)+r\Phi\left(f\right),
\]
and therefore
\begin{align*}
\left\Vert f\right\Vert _{2}^{2} & \leqslant\inf\left\{ s{\cal E}\left(T,f\right)+r\Phi\left(f\right):r\in\mathfrak{R}\left(s\right)\right\} \\
 & =s\calE\left(T,f\right)+\alpha^{\shortminus}\left(s\right)\Phi\left(f\right).
\end{align*}
Note that $\alpha\left(r\right)=0$ for $r\geqslant\mathfrak{a}$
implies that for any $s>0$, 
\[
\alpha^{\shortminus}\left(s\right):=\inf\left\{ r>0:\alpha\left(r\right)\leqslant s\right\} =\inf\{r\in\left(0,\frak{\mathfrak{a}}\right]:\alpha\left(r\right)\leqslant s\}\leqslant\mathfrak{a}\,.
\]
We use the results of \cite{embrechts2013note}, stated for an increasing
function $\mathsf{T}$, but directly applicable here by setting, using
their notation, $\mathsf{T}=-\alpha$ and noting that $\alpha^{\shortminus}\left(s\right)=\mathsf{T}^{\shortminus}\left(-s\right)$.
From \cite[Proposition 1, (2)]{embrechts2013note}, $\alpha^{\shortminus}$
is decreasing.

For any $\varepsilon\geqslant0$, let $s\left(\varepsilon\right):=\sup_{r>\varepsilon}\alpha\left(r\right)$.
If $s\left(0\right)<\infty$, then $\alpha^{\shortminus}\left(s\right)=0$
for $s>s\left(0\right)$. Otherwise, $\lim_{\varepsilon\downarrow0}s\left(\varepsilon\right)=\infty$,
since $\alpha$ is decreasing. Therefore for any $\varepsilon>0$
and any $s\geqslant s\left(\varepsilon\right)$, we have $\inf\left\{ r>0:\alpha\left(r\right)\leqslant s\right\} \leqslant\inf\left\{ r>0:\alpha\left(r\right)\leqslant s\left(\varepsilon\right)\right\} \leqslant\varepsilon$
and $\alpha^{\shortminus}\left(s\right)\leq\alpha^{\shortminus}\left(s\left(\varepsilon\right)\right)\leqslant\varepsilon$.
Hence, $\lim_{s\rightarrow\infty}\alpha^{\shortminus}\left(s\right)=0$,
and thus a $\left(\Phi,\beta\right)\shortminus$WPI with $\beta:=\alpha^{\shortminus}$
holds.

The other listed properties are standard for generalized inverse (monotone)
functions \cite[Proposition 1, (3) and (4)]{embrechts2013note}, using
that $\alpha^{\shortminus}\circ\alpha\left(r\right)=\alpha^{\shortminus}\left(-\mathsf{T}\left(r\right)\right)=\mathsf{T}^{\shortminus}\circ\mathsf{T}\left(r\right)$
and noting that here $\alpha^{\shortminus}\leqslant\mathfrak{a}<\infty$. 

The second statement \ref{enu:alphavsbeta-beta} follows along the
same lines. 

For statement \ref{enu:alpha-alpha-minus-identity} we use that from
\cite[Proposition 1, (5)]{embrechts2013note}, $\alpha\left(r\right)\geqslant s\iff r\leqslant\alpha^{-}\left(s\right)$,
therefore
\begin{align*}
\left(\alpha^{\shortminus}\right)^{\shortminus}\left(r\right) & =\inf\left\{ s>0:\alpha^{\shortminus}\left(s\right)\leqslant r\right\} \\
 & =\inf\left\{ s>0:\alpha^{\shortminus}\left(s\right)<r\right\} \\
 & =\inf\left\{ s>0:\alpha\left(r\right)\leqslant s\right\} \\
 & =\alpha\left(s\right).
\end{align*}
The proof for $\beta$ is identical.
\end{proof}
\begin{defn}
In the situation where a $\left(\Phi,\alpha\right)\shortminus$WPI
(resp. $\left(\Phi,\beta\right)-$WPI) holds for $\alpha$ (resp.
$\beta$) right continuous, we refer to it as a $\left(\Phi,\alpha,\beta\right)\shortminus$WPI
where $\beta=\alpha^{\shortminus}$ (resp. $\alpha=\beta^{\shortminus}$).
\end{defn}

The main interest of WPIs is summarized below:
\begin{thm}[Theorem 8 \cite{ALPW2021}]
\label{thm:WPI_F_bd}Let $P$ be a $\mu-$invariant Markov kernel
on $\big(\mathsf{E},\mathscr{E}\big)$ and assume that $T:=P^{*}P$
satisfies a $\left(\Phi,\beta\right)-$WPI for a sieve $\Phi$. Then
for $f\in\ELL_{0}\left(\mu\right)$ such that $0<\Phi\left(f\right)<\infty$
and any $n\in\mathbb{N}$, it holds that
\begin{equation}
\left\Vert P^{n}f\right\Vert _{2}^{2}\leq\gamma\left(n\right)\Phi\left(f\right),\label{eq:P-Phi-gamma-convergent}
\end{equation}
where $\gamma\left(n\right):=F_{\mathfrak{a}}^{-1}\left(n\right)$,
where $F_{\mathfrak{a}}\colon(0,\mathfrak{a}]\rightarrow\mathbb{R}$
is the decreasing convex and invertible function
\[
F_{\mathfrak{a}}\left(x\right):=\int_{x}^{\mathfrak{a}}\frac{{\rm d}v}{K^{*}\left(v\right)},
\]
with $K^{*}\colon[0,\infty)\rightarrow\left[0,\infty\right]$ defined
as $K^{*}\left(v\right):=\sup_{u\ge0}\left\{ uv-K\left(u\right)\right\} $,
the convex conjugate of $K\colon[0,\infty)\rightarrow[0,\infty)$
given by $K\left(u\right):=u\,\beta\left(1/u\right)$ for $u>0$ and
$K\left(0\right):=0$.

The function $\gamma$ satisfies $\gamma\left(n\right)\to0$ as $n\to\infty$.
\end{thm}

\begin{rem}
In practice, the precise value of $\mathfrak{a}$ as given in Definition~\ref{def:Phi_fn}
may not be known, however an upper bound $a\ge\mathfrak{a}$ is typically
known, as in Example~\ref{exa:osc2}. The conclusions of Theorem~\ref{thm:WPI_F_bd}
remain true when we consider $F_{a}:=\int_{\cdot}^{a}\dif v/K^{*}(v)=F_{\mathfrak{a}}+c$
for $c=\int_{\mathfrak{a}}^{a}\dif v/K^{*}(v)\ge0$, and we obtain
the convergence bound in (\ref{eq:P-Phi-gamma-convergent}) with $\gamma=\gamma(\cdot;a):=F_{a}^{-1}=F_{\mathfrak{a}}^{-1}(\cdot-c)\ge F_{\mathfrak{a}}^{-1}$.
\end{rem}

\begin{rem}
Our proof of this theorem actually supplies a collection of bounds
on $\left\Vert P^{n}f\right\Vert _{2}^{2}$ which trade off tightness
for tractability. In particular, writing $v_{n}=\left\Vert P^{n}f\right\Vert _{2}^{2}/\Phi\left(f\right)$,
one can deduce (in decreasing order of tightness) the bounds
\begin{align*}
\text{for all }n\geqslant1,\quad v_{n} & \leqslant v_{n-1}-K^{*}\left(v_{n-1}\right)\\
\implies\quad v_{n} & \leqslant\left(\Id-K^{*}\right)^{\circ n}\left(v_{0}\right)\\
\implies\quad\left\Vert P^{n}f\right\Vert {}_{2}^{2} & \leqslant\Phi\left(f\right)\cdot\left(\Id-K^{*}\right)^{\circ n}\left(\frac{\left\Vert f\right\Vert _{2}^{2}}{\Phi\left(f\right)}\right)
\end{align*}
and
\begin{align*}
\text{for all }n\geqslant1,\quad F_{\mathfrak{a}}\left(v_{n}\right)-F_{\mathfrak{a}}\left(v_{n-1}\right) & \geqslant1\\
\implies\quad F_{\mathfrak{a}}\left(v_{n}\right) & \geqslant n+F_{\mathfrak{a}}\left(v_{0}\right)\\
\implies\quad\left\Vert P^{n}f\right\Vert {}_{2}^{2} & \leqslant\Phi\left(f\right)\cdot F_{\mathfrak{a}}^{-1}\left(n+F_{\mathfrak{a}}\left(\frac{\left\Vert f\right\Vert _{2}^{2}}{\Phi\left(f\right)}\right)\right).
\end{align*}
Each of these forms will be useful in deducing converse results, i.e.
converting rates of convergence into WPIs.
\end{rem}

\begin{rem}
Given only a WPI for $P$, one can deduce variance dissipation for
the continuous-time semigroup obtained by Poissonizing $P$, i.e.
let $P_{t}=\exp\left(t\mathcal{L}\right)$ with $\mathcal{L}=P-I$,
then

\[
\left\Vert P_{t}f\right\Vert ^{2}\leqslant\Phi\left(f\right)\cdot\gamma\left(2t\right).
\]
\end{rem}

\begin{defn}
\label{def:Phi-gamma-convergence}A $\mu-$invariant Markov kernel
$P$ satisfying (\ref{eq:P-Phi-gamma-convergent}) with $\gamma\downarrow0$
as $n\to\infty$ is said to be $\left(\Phi,\gamma\right)\shortminus$convergent.
If the specific rate $\gamma$ is not important, we may say that $P$
is $\Phi$-convergent.
\end{defn}

The Dirichlet form $\mathcal{E}\left(P^{*}P,f\right)$ may not be
tractable or straightforward to work with. In the reversible scenario,
it is possible to deduce a $\left(\Phi,\beta\right)-$WPI for $\mathcal{E}\left(P^{2},f\right)$
from simpler Dirichlet forms or properties of $P$.
\begin{thm}[\cite{ALPW2021}, Theorem~21 and Theorem~42]
 Let $P$ be a $\mu-$invariant Markov kernel on $\left(\mathsf{E},\mathscr{E}\right)$
and assume that $P$ satisfies a $\left(\tilde{\Phi},\beta_{+}\right)-$WPI
for a sieve $\tilde{\Phi}$. Then, 
\begin{enumerate}
\item if, in addition, \textcolor{red}{${\color{black}P}$ }is $\mu$-reversible
and $\left(-P\right)$ satisfies a $\left(\tilde{\Phi},\beta_{-}\right)-$\textup{WPI,
we have that} $P^{2}$ satisfies a $\left(\Phi,\beta\right)-$\textup{WPI
with, for $s>0$ and $f\in\mathrm{L}_{0}^{2}\left(\mu\right)$,
\begin{align*}
\beta\left(s\right) & :=\inf\left\{ s_{1}\beta_{+}\left(s_{2}\right)+\beta_{-}\left(s_{1}\right)|s_{1}>0,s_{2}>0,s_{1}s_{2}=s\right\} \,,\\
\Phi\left(f\right) & :=\tilde{\Phi}\left(f\right)\vee\tilde{\Phi}\left(\left(\Id+P\right)^{1/2}f\right)\,;
\end{align*}
}
\item if for any $\left(x,A\right)\in\mathsf{E}\times\mathcal{\mathscr{E}}$
we have $P\left(x,A\right)\geqslant\varepsilon\left(x\right)\cdot\int_{A}\delta_{x}\left({\rm d}y\right)$
for some $\varepsilon\colon\mathsf{E}\rightarrow\left[0,1\right]$,
we have that $P^{2}$ satisfies a $\left(\Phi,\beta\right)-$\textup{WPI
with}, \textup{for $s>0$ and $f\in\mathrm{L}_{0}^{2}\left(\mu\right)$,
\begin{align*}
\beta\left(s\right) & :=\inf\left\{ s_{1}\beta_{-}\left(s_{2}\right)+\beta_{+}\left(s_{1}\right)|s_{1}>0,s_{2}>0,s_{1}s_{2}=s\right\} ,\\
\Phi\left(f\right) & :=\tilde{\Phi}\left(f\right){\color{black}{\color{red}{\color{black}\vee\|f\|_{\osc}^{2}}}\,},
\end{align*}
}where here $\beta_{-}\left(s\right):=\frac{1}{2}\mu\big(\varepsilon\left(X\right)^{-1}\geqslant s\big)$. 
\end{enumerate}
For practical purposes it may be useful to note that $K^{*}=K_{+}^{*}\circ K_{-}^{*}$
and $K^{*}=K_{-}^{*}\circ K_{+}^{*}$ in the respective cases above,
with $K_{\pm}^{*}$ defined as in Theorem~\ref{thm:WPI_F_bd}, but
for $\beta_{\pm}$.
\end{thm}

\subsection{$(\|\cdot\|_{p}^{2},\gamma_{p})-$convergence from $(\|\cdot\|_{{\rm osc}}^{2},\gamma)-$convergence\label{subsec:bounded_to_p}}

In practice it can sometime be difficult to establish that a candidate
sieve $\Phi$, found through calculations, is indeed a sieve. In contrast
the cases $\Phi=\|\cdot\|_{\infty}^{2}$ or $\Phi=\|\cdot\|_{\osc}^{2}$
can simplify calculations greatly. This appears at first sight to
be at the expense of generality in terms of the class of functions
for which convergence can be established. The following, which follows
directly from \cite[Lemma 5.1]{Cattiaux2012}, shows that $(\|\cdot\|_{{\rm osc}}^{2},\gamma)\shortminus$convergence
automatically implies $(\|\cdot\|_{p}^{2},\gamma_{p})\shortminus$convergence.
(We note that the result of \cite[Lemma 5.1]{Cattiaux2012} is even
more general, but this full generality is not needed here.) We will
make use of this result throughout this manuscript in order to simplify
presentation. An alternative strategy to handle broader classes of
functions is suggested in \cite[Proposition~37, Theorems~38,~42]{ALPW2021},
where $\left(\left\Vert \cdot\right\Vert _{p}^{2},\beta_{p}:=\beta^{1-1/p}\right)\shortminus$WPIs
for $p\in\left[2,\infty\right]$ are considered directly. We do not
know whether either of these two approaches is suboptimal in general
but have observed that one recovers similar rates in the polynomial
scenario. We note however that we have found the approach given in
\cite{ALPW2021} more difficult to use in practice. We provide a proof
of the result of \cite[Lemma 5.1]{Cattiaux2012} in Appendix \ref{app:first-appendix}
for the reader's convenience.
\begin{prop}
\label{prop:cattiaux-et-al-gamma-p}Let $P$ be a $\mu\shortminus$invariant
Markov kernel, assumed to be\linebreak{}
 $\left(\left\Vert \cdot\right\Vert _{{\rm osc}}^{2},\gamma\right)\shortminus$convergent.
Then $P$ is also $\left(\left\Vert \cdot\right\Vert _{p}^{2},\gamma_{p}\right)\shortminus$convergent
for $p>2$, with 
\[
\gamma_{p}\left(n\right)\leqslant2^{4+4/p}\left[\gamma\left(n\right)\right]^{1-\frac{2}{p}},\quad n\in\mathbb{N}.
\]

\end{prop}

Since the bound for $\gamma_{2}$ is not decreasing, the above result
does not provide an $\ELL$ convergence rate for all $\ELL$ functions.
However, as mentioned in \cite{Rockner2001}, we can deduce uniform
${\rm L}^{1}$ convergence for all $\ELL$ functions from uniform
$\ELL$ convergence for all bounded functions.
\begin{prop}
The following are equivalent:

\begin{equation}
\lim_{n\to\infty}\sup_{f:\mu(f^{2})\le1}\left\Vert P^{n}f-\mu(f)\right\Vert _{1}=0,\label{eq:l1_conv-disc}
\end{equation}
 and
\begin{equation}
\lim_{n\to\infty}\sup_{f:\|f\|_{\infty}\le1}\left\Vert P^{n}f-\mu(f)\right\Vert _{2}=0.\label{eq:l2_conv-disc}
\end{equation}
\end{prop}

\begin{proof}
We start with (\ref{eq:l1_conv-disc})$\Rightarrow$(\ref{eq:l2_conv-disc}).
So consider $f$ with $\|f\|_{\infty}\le1$.
\begin{align*}
\|P^{n}f-\mu(f)\|_{2}^{2} & =\int|P^{n}f-\mu(f)|\cdot|P^{n}f-\mu(f)|\,\dif\mu\\
 & \le2\int|P^{n}f-\mu(f)|\,\dif\mu,
\end{align*}
and this final expression converges uniformly over $f$ to 0 by (\ref{eq:l1_conv-disc}),
since $\{f:\|f\|_{\infty}\le1\}\subset\{f:\mu(f^{2})\le1\}$. We now
consider the converse, (\ref{eq:l2_conv-disc})$\Rightarrow$(\ref{eq:l1_conv-disc}).
Without loss of generality we may consider $f\in\mathcal{F}=\{g\in{\rm L}_{0}^{2}(\mu):\left\Vert g\right\Vert _{2}\leq1\}$.
Let $\epsilon>0$ be arbitrary; we will show that for $n$ large enough,
$\sup_{f:\left\Vert f\right\Vert _{2}\le1}\int\left|P^{n}f\right|\,\dif\mu\leq\epsilon$.
Take $K=4/\epsilon$ and $N$ large enough such that
\[
\sup_{g:\left\Vert g\right\Vert _{\infty}\le K}\|P^{N}(g)-\mu(g)\|_{2}\leq\frac{\epsilon}{2},
\]
which is valid due to (\ref{eq:l2_conv-disc}). Decomposing an arbitrary
$f\in\mathcal{F}$ as $f=f\cdot{\bf 1}_{A}+f\cdot{\bf 1}_{A^{\complement}}$
for $A\in\mathscr{E}$, we have
\[
\int\left|P^{N}f\right|\,\dif\mu\le\int\left|P^{N}\left(f\cdot{\bf 1}_{A}\right)\right|\,\dif\mu+\int\left|P^{N}\left(f\cdot{\bf 1}_{A^{\complement}}\right)\right|\,\dif\mu,
\]
by Minkowski's inequality. Now by Jensen's inequality, $\mu$-invariance
of $P^{N}$, and Cauchy--Schwarz,
\[
\int\left|P^{N}\left(f\cdot{\bf 1}_{A^{\complement}}\right)\right|\,\dif\mu\leq\int\left|f\cdot{\bf 1}_{A^{\complement}}\right|\,\dif\mu\leq\left\Vert f\right\Vert _{2}\mu(A^{\complement})^{1/2}\leq\mu(A^{\complement})^{1/2}.
\]
Take $A=\{x\in\mathsf{E}:\left|f(x)\right|\leq K\}$, and we obtain
by Markov's inequality
\[
\mu(A^{\complement})=\mu({\bf 1}_{\left|f\right|^{2}>K^{2}})\leq\frac{1}{K^{2}}.
\]
From 
\[
\left|\mu(f\cdot{\bf 1}_{A^{\complement}})\right|\leq\mu\left(\left|f\cdot{\bf 1}_{A^{\complement}}\right|\right)\leq1/K,
\]
and $\mu(f)=0$ we also obtain $\left|\mu(f\cdot{\bf 1}_{A})\right|\leq1/K$.
Finally, we deduce that
\begin{align*}
\int\left|P^{N}f\right|\,\dif\mu & \leq\int\left|P^{N}\left(f\cdot{\bf 1}_{A}\right)\right|\,\dif\mu+\int\left|P^{N}\left(f\cdot{\bf 1}_{A^{\complement}}\right)\right|\,\dif\mu.\\
 & \leq\int\left|P^{N}\left(f\cdot{\bf 1}_{A}\right)-\mu(f\cdot{\bf 1}_{A})\right|\,\dif\mu+\left|\mu(f\cdot{\bf 1}_{A})\right|+\frac{1}{K}\\
 & \leq\frac{\epsilon}{2}+\frac{2}{K}\\
 & \leq\epsilon.
\end{align*}
Since $f\in\mathcal{F}$ was arbitrary, the result follows. 
\end{proof}

\subsection{Deducing WPIs from subgeometric rates of convergence \label{subsec:fundamentals-Deducing-WPIs-from}}

Given a quantitative estimate of the convergence of $\left\Vert P^{n}f\right\Vert _{2}^{2}$,
it is possible to deduce a quantitative WPI for $\mathcal{E}\left(P^{*}P,f\right)$.
\begin{prop}[{\cite[Proposition 24; see also Remark 25]{ALPW2021}}]
 Let $P$ be a $\mu-$invariant Markov kernel on $\big(\mathsf{E},\mathscr{E}\big)$,
and let $\Phi$ be a sieve. 
\end{prop}

\begin{enumerate}
\item Suppose that for some $K^{*}$ nonnegative, increasing, convex, and
satisfying $K^{*}\left(0\right)=0$, there holds for all $f\in\ELL_{0}\left(\mu\right)$
such that $0<\Phi\left(f\right)<\infty$ and for all $n\geqslant0$
an estimate of the form $\left\Vert P^{n}f\right\Vert {}_{2}^{2}\leqslant\Phi\left(f\right)\cdot\left(\mathrm{Id}-K^{*}\right)^{\circ n}\left(\frac{\left\Vert f\right\Vert _{2}^{2}}{\Phi\left(f\right)}\right)$.
It then follows that $\mathcal{E}\left(P^{*}P,f\right)\geqslant\Phi\left(f\right)\cdot K^{*}\left(\frac{\left\Vert f\right\Vert _{2}^{2}}{\Phi\left(f\right)}\right)$ 
\item Suppose that for a function $F:\R_{+}\to\left(0,\infty\right)$ which
is decreasing, continuous, divergent at $0$, with an inverse function
$F^{-1}$ which is decreasing, continuous, and convex, and such that
$\log\left(-\mathrm{D}F^{-1}\right)$ is convex, there holds for all
$f\in\ELL_{0}\left(\mu\right)$ such that $0<\Phi\left(f\right)<\infty$
and for all $n\geqslant0$ an estimate of the form $\left\Vert P^{n}f\right\Vert {}_{2}^{2}\leqslant\Phi\left(f\right)\cdot F^{-1}\left(n+F\left(\frac{\left\Vert f\right\Vert _{2}^{2}}{\Phi\left(f\right)}\right)\right).$
It then follows that $\mathcal{E}\left(P^{*}P,f\right)\geqslant\Phi\left(f\right)\cdot K^{*}\left(\frac{\left\Vert f\right\Vert _{2}^{2}}{\Phi\left(f\right)}\right)$,
where $K^{*}=\mathrm{Id}-F^{-1}\left(1+F\left(\cdot\right)\right)$
is nonnegative, increasing, convex, and satisfies $K^{*}\left(0\right)=0$.
\item Suppose that for a function $\gamma:\R_{+}\to\left(0,\infty\right)$
which is decreasing and has limit $0$ at $\infty$, there holds for
all $f\in\ELL_{0}\left(\mu\right)$ such that $0<\Phi\left(f\right)<\infty$
and for all $n\geqslant0$ an estimate of the form $\left\Vert P^{n}f\right\Vert {}_{2}^{2}\leqslant\Phi\left(f\right)\cdot\gamma\left(n\right).$
Suppose also that $P$ is $\mu$-reversible. It then follows that
$\mathcal{E}\left(P^{*}P,f\right)\geqslant\Phi\left(f\right)\cdot K^{*}\left(\frac{\left\Vert f\right\Vert _{2}^{2}}{\Phi\left(f\right)}\right)$,
for some $K^{*}$ which is nonnegative, increasing, convex, and satisfies
$K^{*}\left(0\right)=0$.
\end{enumerate}
\begin{rem}
Note that for reversible kernels $P$, it holds for all $f\in\ELL_{0}\left(\mu\right)$
that the sequence $\gamma_{f}:n\mapsto\left\Vert P^{n}f\right\Vert {}_{2}^{2}$
is decreasing, continuous, convex, and that $\log\left(-{\rm D}\gamma_{f}\right)$
is convex, and hence that the assumption in Part 2 of the above Proposition
holds.
\end{rem}

\subsection{Bounds on the Asymptotic Variance \label{subsec:fundamentals-Bounds-on-the}}

A by-product of the WPI analysis is that the asymptotic variance of
ergodic averages of the Markov chain in question can be upper-bounded
for suitable functions.
\begin{thm}
Let $P$ be a $\mu-$reversible Markov kernel on $\big(\mathsf{E},\mathscr{E}\big)$
and let $\Phi$ be a sieve such that for all $f\in\ELL_{0}\left(\mu\right)$
such that $0<\Phi\left(f\right)<\infty$, the optimized WPI holds:

\[
\frac{\mathcal{E}\left(P^{*}P,f\right)}{\Phi\left(f\right)}\geqslant K^{*}\left(\frac{\left\Vert f\right\Vert _{2}^{2}}{\Phi\left(f\right)}\right).
\]
Assume also that the map $v\mapsto v-K^{*}\left(v\right)$ is increasing
on $(0,\mathfrak{a}]$. Define $B\left(v\right)=\int_{0}^{v}\frac{w}{K^{*}\left(w\right)}\:\mathrm{d}w$,
which is assumed to be finite for $v\in\left[0,\mathfrak{a}\right]$.
Then the asymptotic variance of $f$ can be bounded as

\[
\mathrm{var}\left(P,f\right)\leqslant4\cdot\Phi\left(f\right)\cdot B\left(\frac{\left\Vert f\right\Vert _{2}^{2}}{\Phi\left(f\right)}\right).
\]
 
\end{thm}

\begin{proof}
Using reversibility of the kernel, we write the asymptotic variance
of $f$ as

\[
\mathrm{var}\left(P,f\right)=\int_{-1}^{1}\nu_{f}\left(\mathrm{d}\lambda\right)\cdot\frac{1+\lambda}{1-\lambda}.
\]
Bounding $\frac{1+\lambda}{1-\lambda}=\frac{\left(1+\lambda\right)^{2}}{1-\lambda^{2}}\leqslant4\cdot\frac{1}{1-\lambda^{2}}$,
we can thus bound

\begin{align*}
\mathrm{var}\left(P,f\right) & \leqslant4\cdot\int_{-1}^{1}\nu_{f}\left(\mathrm{d}\lambda\right)\cdot\frac{1}{1-\lambda^{2}}\\
 & =4\cdot\sum_{n\geqslant0}\left\Vert P^{n}f\right\Vert _{2}^{2}.
\end{align*}
Recall now our tightest discrete-time bound on the variance of the
semigroup, with $S:=\Id-K^{*}$,

\[
\left\Vert P^{n}f\right\Vert {}_{2}^{2}\leqslant\Phi\left(f\right)\cdot S^{\circ n}\left(\frac{\left\Vert f\right\Vert _{2}^{2}}{\Phi\left(f\right)}\right),
\]
we write $v=\frac{\left\Vert f\right\Vert _{2}^{2}}{\Phi\left(f\right)}$
and bound the asymptotic variance as

\begin{align*}
\mathrm{var}\left(P,f\right) & \leqslant4\cdot\Phi\left(f\right)\cdot\sum_{n\geqslant0}S^{\circ n}\left(v\right)\\
 & =:4\cdot\Phi\left(f\right)\cdot\tilde{B}\left(v\right).
\end{align*}

We now control the growth of $\tilde{B}$. Noting that $S$ is nonnegative,
increasing, and concave, a simple induction argument proves that $S^{\circ n}$
also has these properties, and since $\tilde{B}$ is a nonnegative
combination of these functions, it too has these properties.

Now, isolating the first term in the sum which defines $\tilde{B}$,
we have the recursion $\tilde{B}\left(v\right)=v+\tilde{B}\left(S\left(v\right)\right)$,
which allows us to write

\begin{align*}
v & =\tilde{B}\left(v\right)-\tilde{B}\left(S\left(v\right)\right)\\
 & =\int_{S\left(v\right)}^{v}\tilde{B}'\left(w\right)\,\mathrm{d}w.
\end{align*}
By concavity, it holds that for $w\in\left[S\left(v\right),v\right]$,
$\tilde{B}'\left(w\right)\geqslant\tilde{B}'\left(v\right)$, whence

\begin{align*}
v & \geqslant\left(v-S\left(v\right)\right)\cdot\tilde{B}'\left(v\right)\\
 & =K^{*}\left(v\right)\cdot\tilde{B}'\left(v\right)\\
\implies\quad\tilde{B}'\left(v\right) & \leqslant\frac{v}{K^{*}\left(v\right)}.
\end{align*}
Now, arguing that $\tilde{B}\left(0\right)=0$ and integrating, we
obtain the expression
\[
\tilde{B}\left(v\right)\le B\left(v\right):=\int_{0}^{v}\frac{w}{K^{*}\left(w\right)}\:\mathrm{d}w
\]
from which the result follows.
\end{proof}
\begin{rem}
An analogous result can be shown for a continuous-time Markov process
$\{P_{t}:t\geqslant0\}$, by defining the Dirichlet form in terms
of the infinitesimal generator.
\end{rem}

\begin{rem}
It is plausible that the assumption that $v\mapsto v-K^{*}\left(v\right)$
is increasing might follow from the defining properties of $K$ and/or
$\beta$, but we have been unable to establish this directly. In all
of our explicit examples, this condition holds.
\end{rem}

\subsection{Towards spectral interpretations \label{subsec:fundamentals-Towards-spectral-interpretations}}

In the reversible scenario, spectral representations of the operator
$P$ can provide useful insights. Subgeometric convergence naturally
implies that the spectral radius of $P$ is one and therefore that
the spectrum accumulates at $-1$ or $1$. The following are attempts
to make these ideas more concrete.

\subsubsection{Concentration of the spectrum}

When $P$ is reversible, we can utilize the spectral projection-valued
measure representation of $P$. Thus for a given $f\in\ELL_{0}(\mu)$,
let $\nu_{f}(\dif\lambda)$ be the positive measure on $\sigma(P)$
which satisfies
\[
\langle P^{n}f,f\rangle=\int_{\sigma(P)}\lambda^{n}\,\nu_{f}(\dif\lambda).
\]
Note that $\nu_{f}$ is a probability measure precisely when $\left\Vert f\right\Vert _{2}=1$.
From our $(\Phi,\beta)-$WPI, we can conclude $(\Phi,\gamma)$--
convergence of $\|P^{n}f\|_{2}^{2}$ for some $\gamma:\mathbb{N}_{0}\to\R$
with $\gamma\left(n\right)\downarrow0$ as $n\to\infty$. This gives
some control on the moments of $\nu_{f}$: for any $f\in\ELL_{0}(\mu)$
with $\left\Vert f\right\Vert _{2}=1$,
\begin{equation}
\|P^{n}f\|_{2}^{2}=\int_{\sigma(P)}\lambda^{2n}\,\nu_{f}\left(\dif\lambda\right)\le\Phi\left(f\right)\gamma\left(n\right).\label{eq:nu_f-moment-bd}
\end{equation}
In particular, we have
\[
\sup_{f:\left\Vert f\right\Vert _{2}=1}\left\{ \frac{\int_{\sigma(P)}\lambda^{2n}\,\nu_{f}\left(\dif\lambda\right)}{\Phi\left(f\right)}\right\} \leq\gamma\left(n\right),
\]
from which we may deduce by Markov's inequality
\[
\sup_{f:\left\Vert f\right\Vert _{2}=1}\left\{ \frac{\mathbb{P}_{\nu_{f}}\left(\lambda^{2}>\exp\left(-\delta\right)\right)}{\Phi\left(f\right)}\right\} \leq\inf_{n\geq1}\left\{ \frac{\gamma\left(n\right)}{\exp\left(-\delta n\right)}\right\} .
\]
For example, if $\gamma(n)\leq cn^{-k}$, then there exists $C$ such
that 
\[
\sup_{f:\left\Vert f\right\Vert _{2}=1}\left\{ \frac{\mathbb{P}_{\nu_{f}}\left(\lambda^{2}>\exp\left(-\delta\right)\right)}{\Phi\left(f\right)}\right\} \leq C\delta^{k}.
\]
This may be viewed as the subgeometric counterpart to the fact that
if $\gamma(n)=\rho^{n}$ then this implies by the same reasoning that
$\mathbb{P}_{\nu_{f}}(\lambda^{2}>\rho)=0$ for all $f$ with $\Phi(f)<\infty$
and $\left\Vert f\right\Vert _{2}=1$.

\subsubsection{Spectrum of the Independent Metropolis--Hastings algorithm}

Consider the Independent Metropolis--Hastings (IMH), also known as
an \textit{independence sampler}, on a countable state space $\E=\mathbb{N}_{0}$.
For a fixed target distribution $\pi$ and proposal distribution $q$
on $\E$, at position $X_{n}=x$, the chain proposes a move to $Y\sim q$,
and conditional on $Y=y$, accepts this move with probability $1\wedge\frac{\pi(y)q(x)}{\pi(x)q(y)}$
and sets $X_{n+1}=y$, otherwise the move is rejected and $X_{n+1}=x$.
For brevity, we define
\[
w(x):=\frac{\pi(x)}{q(x)},\quad x\in\E.
\]
For the IMH, the spectrum of the transition kernel $P$ has been characterized
in \cite{gaasemyr2003spectrum}:
\[
\sigma(P)=\{\mathsf{r}_{w}:w\in\mathcal{W}\}\cup\{1\},
\]
where $\mathcal{W}=\{w(x):x\in\E\}$, $\mathsf{r}_{w}:=\mathbb{P}(X_{1}=x\,|X_{0}=x,w(x)=w)$
are the rejection probabilities. 

In order to be concrete, we consider a specific choice of $\pi,q$:
we take geometric $\pi(x)=(1-a)\cdot a^{x}$ and $q(x)=(1-b)\cdot b^{x}$
for $x\in\E=\mathbb{N}_{0}$, where $0<b<a<1$. In this case, the
Markov chain will converge subgeometrically, with rate $n^{-\frac{b}{a-b}}$
for bounded functions (this can be seen by a straightforward adaptation
of the example in \cite[Section 2.3.1]{ALPW2021}). In this countable
state space setting, it is furthermore possible to explicitly characterize
the spectrum \cite{gaasemyr2003spectrum}. By computing explicitly
the rejection probabilities $\mathsf{r}_{w}$, we find that
\begin{equation}
\sigma(P)=\left\{ \Lambda_{m}:=1-\frac{1-b}{1-a}\cdot\left(\frac{b}{a}\right)^{m}+\frac{a-b}{1-a}\cdot b^{m}:m\in\mathbb{N}_{0}\right\} \cup\{1\}.\label{eq:spectrum_IMH}
\end{equation}
Since $\Lambda_{m}\uparrow1$ as $m\to\infty$, we see there is no
spectral gap, and indeed choosing a smaller value of $b$ -- which
leads to a slower rate of convergence for bounded functions -- causes
the spectrum to concentrate even more tightly around $1$.

Given a test function $f\in\ELL_{0}(\pi)$ with $\|f\|_{2}=1$, we
can consider its spectral measure $\nu_{f}(\cdot)$ on $\sigma(P)$,
which has the property that $\langle P^{n}f,f\rangle=\int_{\sigma(P)}\lambda^{n}\,\nu_{f}(\dif\lambda)$
for all $n\in\mathbb{N}_{0}$. Since $f$ has unit norm, $\nu_{f}$
is a probability mass function supported on $\{\Lambda_{m}:m\in\mathbb{N}_{0}\}$.
The function $f$ is thus entirely characterized by the measure $\nu_{f}$,
and many of its properties can be read off from this.

For example, if 
\begin{equation}
\int_{\sigma(P)}(1-\lambda)^{-1}\,\nu_{f}(\dif\lambda)=\sum_{m\in\mathbb{N}_{0}}(1-\Lambda_{m})^{-1}\,\nu_{f}(\Lambda_{m})<\infty,\label{eq:imh_spec_sum}
\end{equation}
then $f$ will have a finite asymptotic variance. Given our expression
for the $\Lambda_{m}$ (\ref{eq:spectrum_IMH}), we see this will
be the case when the masses $\nu_{f}(\Lambda_{m})$ decay strictly
faster than $(a/b)^{m}$, to ensure the sum in (\ref{eq:imh_spec_sum})
is finite.

\section{Optimal choices of $\alpha,\beta,\Phi$ and ordering \label{sec:Optimal-choices}}

Given our formulation of a WPI in Definition~\ref{def:WPI}, it is
natural to ask how one might optimize the constituent components:
that is, how to make formal the notion of a ``best'' possible $\alpha,\beta$
or $\Phi$. 

\subsection{Optimal $\alpha$ and $\beta$ \label{subsec:optimal-choices-alpha-beta}}

We start by fixing a given sieve $\Phi$, and seeking an optimal $\alpha$
and $\beta$. We assume that $\Phi$ is such that there exist functions
$f$ such that $0<\Phi\left(f\right)<\infty$. Since ${\rm var}_{\mu}\left(f\right)\leq\mathfrak{a}\Phi\left(f\right)$,
$\Phi(f)=0\Rightarrow{\rm var}_{\mu}(f)=0$ and so this assumption
means only that we avoid the scenario where the only functions such
that $\Phi\left(f\right)<\infty$ are constant functions.

We define minimal $\alpha$ and $\beta$ functions, for a given sieve
$\Phi$, as the (pointwise) minimal functions satisfying Definition~\ref{def:WPI}.
\begin{defn}
\label{def:alpha-beta-star}For a $\mu-$invariant Markov kernel $T$
and sieve $\Phi$ define,
\begin{enumerate}
\item for any $r>0$,
\[
\alpha^{\star}\left(r;\Phi\right):=\sup\left\{ \frac{\left\Vert g\right\Vert _{2}^{2}}{\mathcal{E}\left(T,g\right)}\left(1-\frac{r}{\left\Vert g\right\Vert _{2}^{2}}\right)\colon g\in\ELL_{0}\left(\mu\right),\Phi\left(g\right)=1\right\} \vee0,
\]
noting that if $r\ge\mathfrak{a}$, $\alpha^{\star}\left(r;\Phi\right)=0$;
\item for any $s>0$, 
\[
\beta^{\star}\left(s;\Phi\right):=\sup\left\{ \left\Vert g\right\Vert _{2}^{2}-s\mathcal{E}\left(T,g\right)\colon g\in\ELL_{0}\left(\mu\right),\Phi\left(g\right)=1\right\} \vee0\quad.
\]
\end{enumerate}
When $\Phi=\|\cdot\|_{{\rm osc}}^{2}$ we shall plainly write $\alpha^{\star}\left(\cdot\right):=\alpha^{\star}\left(\cdot;\Phi\right)$
and $\beta^{\star}\left(\cdot\right):=\beta^{\star}\left(\cdot;\Phi\right)$. 
\end{defn}

Despite their definitions it is not clear that the functions $\alpha^{\star}$
and $\beta^{\star}$ satisfy all the conditions required for a WPI
to hold. The following theorem clarifies this point and also establishes
that $\alpha^{\star}$ and $\beta^{\star}$ are inverses of each other
when restricted to appropriate domains. The statement requires the
existence of some $(\Phi,\alpha)$- or $(\Phi,\beta)$-WPI, which
we note can be established with the results of Subsection~\ref{subsec:establish-WPI-WPIs-from-RUPI}
for $\Phi=\|\cdot\|_{\osc}^{2}$. In particular Corollary~\ref{cor:WPI_from_irred}
establishes that $\mu-$irreducibility is a sufficient condition for
the existence of a WPI.
\begin{thm}
\label{thm:alpha-beta-star-WPI-continuous-convex} Suppose that the
$\mu-$invariant kernel $T$ possesses some $(\Phi,\alpha)$- or $(\Phi,\beta)$-WPI.
Then $\alpha^{\star}(\cdot;\Phi)$ defines a $(\Phi,\alpha^{\star})\shortminus$WPI
and $\beta^{\star}(\cdot;\Phi)$ defines a $(\Phi,\beta^{\star})\shortminus$WPI.
Furthermore, the functions $\alpha^{\star}(\cdot;\Phi):(0,\mathfrak{a}]\to[0,\infty)$
and $\beta^{\star}(\cdot;\Phi):[0,\infty)\to[0,\mathfrak{a}]$ are
convex and continuous. In addition, $\beta^{\star}$ is strictly decreasing
to $0$ and $\alpha^{\star}=\left(\beta^{\star}\right)^{-1}$ is the
inverse function, which is well-defined on $(0,\mathfrak{a}]$ and
strictly decreasing.
\end{thm}

\begin{proof}
We consider the $\beta$ formulation, and drop explicit reference
to the fixed $\Phi$ under consideration; the $\alpha$ formulation
is analogous. By assumption, we know that $T$ possesses a $(\Phi,\beta)$-WPI,
for some function $\beta$ as in Definition~\ref{def:WPI} (\textit{c.f.}
Proposition~\ref{prop:a-b-WPI-corres}). By definition of $\beta^{\star}$,
we have that $0\leq\beta^{\star}\le\beta$ pointwise and so $\beta^{\star}(s)\to0$
as $s\to\infty$. Since the pointwise supremum of affine functions
(of $s$) is convex, we obtain convexity and continuity of $\beta^{\star}$,
from the fact that it is the composition of a nondecreasing convex
continuous function, $s\mapsto\max\{0,s\}$, with a convex function.
We observe that $\beta^{\star}(0)=\mathfrak{a}$. Now, let $s_{0}:=\inf\{s>0:\beta^{\star}(s)=0\}$,
which may be infinite. Since $\beta^{\star}$ is convex and continuous,
it is strictly decreasing on $(0,s_{0})$. It follows that $\beta^{\star}$
is invertible on $(0,s_{0})$ with inverse $(\beta^{\star})^{-1}:(0,\mathfrak{a}]\to[0,\infty)$
that is also convex and strictly decreasing.

Now we show that $\alpha^{\star}=(\beta^{\star})^{-1}$. For $r\in(0,\mathfrak{a}]$,
let $s:=(\beta^{\star})^{-1}(r)$. For any $f\in\ELL_{0}(\mu)$ with
$\Phi(f)=1$ we have
\[
\|f\|_{2}^{2}-s\mathcal{E}(T,f)\leq\beta^{\star}(s)=r,
\]
and this implies 
\[
\alpha^{\star}(r)=\sup_{f:\Phi(f)=1}\frac{\|f\|_{2}^{2}}{\mathcal{E}(T,f)}-\frac{r}{\mathcal{E}(T,f)}\leq s.
\]
Assume for the sake of contradiction that $\alpha^{\star}(r)=t<s$.
For any $f\in\ELL_{0}(\mu)$ with $\Phi(f)=1$ we have
\[
\frac{\|f\|_{2}^{2}}{\mathcal{E}(T,f)}-\frac{r}{\mathcal{E}(T,f)}\leq t,
\]
and so
\[
\beta^{\star}(t)=\sup_{f:\Phi(f)=1}\|f\|_{2}^{2}-t\mathcal{E}(T,f)\leq r=\beta^{\star}(s),
\]
which is a contradiction since $\beta^{\star}$ is decreasing, and
we conclude.
\end{proof}
\begin{rem}
\label{rem:psi-alpha}The function $\alpha^{\star}$ may be upper
and lower bounded using the function $\psi:\R_{+}\to[0,\infty)$,
\[
\psi(t;\Phi):=\inf_{f:\Phi(f)=1,\|f\|_{2}^{2}>t}\frac{\mathcal{E}(T,f)}{\|f\|_{2}^{2}},
\]
which is nondecreasing. The behaviour of $\psi(\cdot;\Phi)$ as $t$
decreases to $0$ gives bounds on $\alpha^{\star}(\cdot;\Phi)$. Indeed,
we find that for any $t>r$,
\[
\frac{1}{\psi(t;\Phi)}\left(1-\frac{r}{t}\right)\leq\alpha^{\star}(r;\Phi)\leq\frac{1}{\psi(r;\Phi)}.
\]
Taking $t=2r$ we obtain 
\[
\frac{1}{2\psi(2r;\Phi)}\leq\alpha^{\star}(r;\Phi)\leq\frac{1}{\psi(r;\Phi)},
\]
and we may also deduce that $\lim_{r\downarrow0}\alpha^{\star}(r;\Phi)=\psi(0;\Phi)^{-1}$.
We see that $\alpha^{\star}$ is intimately connected to the rate
at which $\psi$ decreases as $t$ decreases, i.e. as the variance
of functions $f$ with $\Phi(f)=1$ is allowed to decrease to $0$.
We will see in Theorem~\ref{thm:WPI_from_cond} that, when $\Phi=\left\Vert \cdot\right\Vert _{{\rm osc}}^{2}$,
upper and lower bounds may also be obtained by considering only indicator
functions. One can also bound $\beta^{\star}$ in a similar manner
using the function $\psi^{-}(u):=\sup\{t:\psi(t)\leq u\}$, in which
case one finds
\[
\frac{1}{2}\psi^{-}\left(\frac{1}{2s};\Phi\right)\leq\beta^{\star}(s;\Phi)\leq\psi^{-}\left(\frac{1}{s};\Phi\right).
\]
\end{rem}

In fact, if $\Phi$ defines a subspace $\mathcal{F}$ of $\ELL_{0}(\mu)$
then one may view $\psi(0;\Phi)$ as the right spectral gap associated
with $T$ as an operator on the closure of $\mathcal{F}$; see Lemma~\ref{lem:phi-spectral-gap}.
In the case where $T=P^{*}P$ and $\psi(0;\Phi)>0$ then this implies
$\left\Vert P^{n}f\right\Vert _{2}^{2}\leq\left\{ 1-\psi(0;\Phi)\right\} ^{n}\left\Vert f\right\Vert _{2}^{2}$
for functions $f\in\mathcal{F}$; see Remark~\ref{rem:phi-operator-norm}.
This is also natural by observing that if we define $\alpha^{\star}(0;\Phi):=\lim_{r\downarrow0}\alpha^{\star}(r;\Phi)=\psi(0;\Phi)^{-1}$
we observe that a $(\Phi,\alpha^{\star})$-WPI implies that $\left\Vert f\right\Vert _{2}^{2}\leq\alpha^{\star}(0;\Phi)\mathcal{E}(T,f)$
for all $f\in\mathcal{F}$, from which the same bound on $\left\Vert P^{n}f\right\Vert _{2}^{2}$
may be directly obtained. Finally, when $\Phi=\left\Vert \cdot\right\Vert _{{\rm osc}}^{2}$
then $\psi(0;\Phi)$ is the $\ELL_{0}(\mu)$ spectral gap; see Lemma~\ref{lem:phi-spectral-gap}.

\subsection{Lower bounds on convergence rates \label{subsec:optimal-choice-Lower-bounds-on}}

In principle, noting that $\alpha^{\star}$ and $\beta^{\star}$ are
pointwise minimal functions, any function $f\in\ELL_{0}(\mu)$ with
$\Phi(f)=1$ may be used to construct a lower bound. For example,
for any such function, $\beta^{\star}$ satisfies
\[
\beta^{\star}(s)\geq\|f\|_{2}^{2}-s\mathcal{E}(T,f),\qquad s>0.
\]
In practice, to produce an informative lower bound for the whole function
$\beta^{\star}$, one will need to identify an appropriate sequence
of functions. Indicator functions of measurable sets are always in
$\ELL_{0}(\mu)$, have finite oscillation, and they can provide a
tractable source of such functions as $\mathcal{E}(P,{\bf 1}_{A})$
has a natural probabilistic interpretation. We show that such functions
can provide both lower and upper bounds for $\beta^{\star}$ in Section~\ref{subsec:Cheeger-meets-Poincar=0000E9}. 

We now show that a lower bound on $\beta_{1}$ in a $(\Phi,\beta_{1})$-WPI
for $P$ can imply a lower bound on $\beta_{2}$ in a $(\Phi,\beta_{2})$-WPI
for $P^{*}P$. 
\begin{lem}[{\cite[Remark~3.1]{diaconis1996nash}}]
\label{lem:PP-dirichlet-form-ub}Let $P$ be $\mu$-invariant. Then
\[
\mathcal{E}(P^{*}P,f)\leq2\mathcal{E}(P,f),\qquad f\in\ELL_{0}(\mu).
\]
\end{lem}

\begin{rem}
If $P$ is $\mu$-reversible, one can obtain $\mathcal{E}(P^{2},f)\leq(1+\lambda_{\star})\mathcal{E}(P,f)$
by using the spectral theorem, where $\lambda_{\star}=\sup\sigma_{0}(P)$.
However, since the focus here is on WPIs, the case $\lambda_{\star}<1$
is less relevant.
\end{rem}

We note that a converse may be obtained when $P$, and therefore $P^{*}$,
satisfies $P(x,\{x\})\geq\varepsilon$ on a $\mu$-full set; see Lemma~\ref{lem:stick-dirichlet-lower-bound}. 

\begin{lem}
\label{lem:PP-beta-star-from-P-beta-star}Let $P$ be $\mu$-invariant,
and assume it satisfies a $(\Phi,\beta_{1}^{\star})$-WPI, where $\beta_{1}^{\star}$
is pointwise minimal. Assume $P^{*}P$ satisfies a $(\Phi,\beta_{2}^{\star})$-WPI
where $\beta_{2}^{\star}$ is pointwise minimal. Then $\beta_{2}^{\star}(s)\geq\beta_{1}^{\star}(2s)$.
\end{lem}

\begin{proof}
Let $\mathcal{F}=\{f\in\ELL_{0}(\mu):\Phi(f)=1\}$. By Lemma~\ref{lem:PP-dirichlet-form-ub}
we have $\mathcal{E}(P^{*}P,f)\leq2\mathcal{E}(P,f)$. We may write
\[
\beta_{1}^{\star}(s)=0\vee\sup_{f\in\mathcal{F}}{\rm var}_{\mu}(f)-s\mathcal{E}(P,f).
\]
We then have
\begin{align*}
\beta_{2}^{\star}(s) & =0\vee\sup_{f\in\mathcal{F}}{\rm var}_{\mu}(f)-s\mathcal{E}(P^{*}P,f)\\
 & \geq0\vee\sup_{f\in\mathcal{F}}{\rm var}_{\mu}(f)-2s\mathcal{E}(P,f)\\
 & =\beta_{1}^{\star}(2s),
\end{align*}
and we conclude.
\end{proof}
In the case where $P$ is $\mu$-reversible, we can then deduce from
a $(\Phi,\beta_{1})$-WPI for $P$ a lower bound on a separable rate
of convergence for $\left\Vert P^{n}f\right\Vert $. 
\begin{prop}
\label{prop:rev-L2-conv-rate-lower-bound}Assume $P$ is $\mu$-reversible,
satisfies (\ref{eq:beta-WPI}) and the pointwise minimal $\beta^{\star}$
satisfies $\beta^{\star}(s)\in\Omega(s^{-p})$ for some $p>0$. Then
it cannot hold that with $q>p$, $\left\Vert P^{n}f\right\Vert _{2}^{2}\in\mathcal{O}(n^{-q})$
for all $f\in\ELL_{0}(\mu)$ with $\Phi(f)<\infty$.
\end{prop}

\begin{proof}
If $\beta^{\star}(s)\in\Omega(s^{-p})$ then we may deduce by Lemma~\ref{lem:PP-beta-star-from-P-beta-star}
that if $P^{2}$ satisfies (\ref{eq:beta-WPI}), its pointwise minimal
$\beta_{2}^{\star}$ also satisfies $\beta_{2}^{\star}(s)\in\Omega(s^{-p})$.
Now assume for the sake of contradiction that $\left\Vert P^{n}f\right\Vert _{2}^{2}\in\mathcal{O}(n^{-q})$
for all $f\in\mathrm{L}_{0}^{2}(\mu)$ such that $\Phi(f)<\infty$.
Then by \cite[Proposition 24 and Remark 25]{ALPW2021}, we deduce
that a WPI for $P^{2}$ holds with $\beta_{2}(s)\in\mathcal{O}(s^{-q})$,
which contradicts $\beta_{2}^{\star}(s)\in\Omega(s^{-p})$ being pointwise
minimal.
\end{proof}
The following result establishes a lower bound on $\beta^{\star}$
for Markov kernels $P$ that can exhibit sticky behaviour in regions
of the state space. \cite[Theorem~5.1]{roberts1996geometric} showed
that for a $\mu$-invariant Markov kernel $P$ with $\mu$ not concentrated
at a single point, that ${\rm ess}_{\mu}\sup_{x}P(x,\{x\})=1$ implies
that $P$ cannot converge geometrically. In \cite[Theorem 1]{lee-latuszynski-2014}
conductance is used to prove the same when $P$ is $\mu$-reversible,
and the following provides a quantitative refinement. 
\begin{thm}
\label{thm:lower-bound-beta-star}Let $P$ be $\mu-$reversible satisfying
a $(\Phi,\beta)\shortminus$WPI for $\Phi=\|\cdot\|_{\osc}^{2}$.
For any $\varepsilon>0$, define the set $A_{\varepsilon}:=\big\{ x\in\mathsf{X}\colon P(x,\{x\})\geq1-\varepsilon\big\}$.
Then for any $s>0$, 
\[
\beta(s)\geq\beta^{\star}(s)\geq\sup_{\varepsilon\in(0,1)}\left\{ \mu(A_{\varepsilon})(1-s\varepsilon-\mu(A_{\varepsilon}))\right\} .
\]
\end{thm}

\begin{proof}
For any $A\subset\mathsf{X}$, from Lemma~\ref{lem:dirichlet-form-indicator},
we have $\mathcal{E}(P,\mathbf{1}_{A})=\mu\otimes P\big(A\times A^{\complement}\big)$
and ${\rm var}\big(\mathbf{1}_{A}\big)=\mu\otimes\mu\big(A\times A^{\complement}\big)$.
Since $\Phi(\mathbf{1}_{A_{\varepsilon}})\leq1$, for any $\varepsilon>0$
we have
\begin{align*}
\beta^{\star}(s) & :=\sup\left\{ \|f\|_{2}^{2}-s\mathcal{E}(P,f)\colon f\in\ELL_{0}(\mu),\Phi(f)\leq1\right\} \\
 & \geq{\rm var}_{\mu}(\mathbf{1}_{A_{\varepsilon}})-s\int\mu({\rm d}x)P(x,{\rm d}y)\mathbf{1}_{A_{\varepsilon}}(x)\mathbf{1}_{A_{\varepsilon}^{\complement}}(y)\\
 & \geq{\rm var}_{\mu}(\mathbf{1}_{A_{\varepsilon}})-s\int\mu({\rm d}x)P(x,\{x\}^{\complement})\mathbf{1}_{A_{\varepsilon}}(x)\\
 & \geq\mu(A_{\varepsilon})\mu(A_{\varepsilon}^{\complement})-s\mu(A_{\varepsilon})\varepsilon\\
 & =\mu(A_{\varepsilon})(1-s\varepsilon-\mu(A_{\varepsilon})).
\end{align*}
Thus we conclude.
\end{proof}
\begin{example}
Assume $C\varepsilon^{\alpha}\geq\mu(A_{\varepsilon})\geq c\varepsilon^{\alpha}$
for some $\alpha,c,C>0$ and for all $\varepsilon>0$ sufficiently
small. Then for $s>0$, we seek to maximize $\zeta(\varepsilon)=\varepsilon^{\alpha}(1-s\varepsilon-C\varepsilon^{\alpha})$.
One can check that 
\begin{align*}
\zeta'(\varepsilon) & =(1+\alpha)\varepsilon^{\alpha-1}\left[\frac{\alpha}{1+\alpha}-s\varepsilon-C\varepsilon^{\alpha}\right],
\end{align*}
and since $\mathbb{R}_{+}\ni\varepsilon\mapsto s\varepsilon+c\varepsilon^{\alpha}$
is increasing, there is a unique $\varepsilon^{*}$ such that $\zeta'(\varepsilon_{*})=0$,
$\zeta(\varepsilon)>0$ (resp. $\zeta(\varepsilon)<0$ ) for $\varepsilon<\varepsilon_{*}$
(resp. $\varepsilon>\varepsilon_{*}$). Note that for $s\geq\alpha/(1+\alpha)$,
$\varepsilon_{*}\in(0,1)$ and let
\begin{align*}
\varepsilon_{0}:=\frac{\alpha}{1+\alpha}s^{-1},
\end{align*}
from above. Then notice that $\zeta(\varepsilon_{0})\leq0$ and $\varepsilon_{0}'=\varepsilon_{0}-Cs^{-1}\varepsilon_{0}$
is such that $\zeta'(\varepsilon_{0}^{'})\geq0$, implying $\varepsilon_{0}-cs^{-1}\varepsilon_{0}\leq\varepsilon_{*}\leq\varepsilon_{0}$
and we obtain the lower bound, for $s>0$
\[
\beta^{\star}(s)\geq\text{\ensuremath{\underbar{\ensuremath{\beta}}}}^{\star}(s):=c\left(\frac{\alpha}{1+\alpha}\right)^{\alpha}s^{-\alpha}\left[\frac{1}{1+\alpha}-C\left(\frac{\alpha}{1+\alpha}\right)^{\alpha}s^{-\alpha}\right]\,,
\]
which is positive for $s$ sufficiently large. Therefore, since from
earlier results $\beta^{\star}\geq\text{\ensuremath{\underbar{\ensuremath{\beta}}}}^{\star}$
implies $\underline{\gamma}^{\star}(n)\leq\gamma^{\star}(n)$ if $\mu(A_{\varepsilon})\geq c\varepsilon^{\alpha}$
then the corresponding Markov chain cannot converge at a rate faster
than the polynomial rate $\text{\ensuremath{\underbar{\ensuremath{\gamma}}}}^{\star}(n)\propto n^{-\alpha}$. 
\end{example}

\begin{example}
In the case of the Independent Metropolis-Hastings (IMH) we are interested
in lower bounding the probability
\[
\varpi(\varepsilon):=\pi\left(\int\pi({\rm d}y)\min\left\{ w^{-1}(X),w^{-1}(y)\right\} <\varepsilon\right).
\]
Note that for any $x\in\mathsf{X}$ we have
\[
\int\pi({\rm d}y)\min\left\{ w^{-1}(x),w^{-1}(y)\right\} \leq w^{-1}(x),
\]
therefore, since for random variables $Z(\omega)\leq Z'(\omega)$
implies $\mathbb{P}(Z(\omega)<\varepsilon)\geq\mathbb{P}(Z'(\omega)<\varepsilon)$
\[
\varpi(\varepsilon)\geq\pi\left(w^{-1}(X)<\varepsilon\right)=\pi\left(w(X)>\varepsilon^{-1}\right).
\]
As a result for $s>0$
\[
\beta^{\star}(s)\geq\sup_{\varepsilon\in(0,1)}\left\{ \pi\left(w(X)>\varepsilon^{-1}\right)\left(1-s\varepsilon-\pi\big(w(X)>\varepsilon^{-1}\big)\right)\right\} ,
\]
therefore implying a lower bound on the fastest rate of convergence
possible.
\end{example}

\subsection{Ordering of $\alpha$'s, $\beta$'s and $\gamma$'s and Peskun--Tierney
ordering \label{subsec:optimal-choice-Ordering-of-rates}}
\begin{thm}
\label{thm:order-alpha-beta-gamma}Let $P_{1}$ and $P_{2}$ be $\mu-$invariant
Markov kernels such that for a sieve $\Phi$, $P_{1}^{*}P_{1}$ satisfies
a $(\Phi,\alpha_{1},\beta_{1})\shortminus$WPI and $P_{2}^{*}P_{2}$
a $(\Phi,\alpha_{2},\beta_{2})\shortminus$WPI respectively. Then
we have
\begin{enumerate}
\item $\alpha_{2}(\cdot;\Phi)\geq\alpha_{1}(\cdot;\Phi)$ if and only if
$\beta_{2}(\cdot;\Phi)\geq\beta_{1}(\cdot;\Phi)$; 
\item $\beta_{2}(\cdot;\Phi)\geq\beta_{1}(\cdot;\Phi)$ implies $\gamma_{2}(\cdot;\Phi)\geq\gamma_{1}(\cdot;\Phi)$.
\end{enumerate}
\end{thm}

\begin{proof}
First statement: we drop $\Phi$ for notational simplicity. For the
direction $(\implies):$ for any $s>0$ we have $\{r>0\colon\alpha_{2}(r)\leq s\}\subset\{r>0\colon\alpha_{1}(r)\leq s\}$
and hence $\beta_{2}=\alpha_{2}^{\shortminus}\geq\alpha_{1}^{\shortminus}=\beta_{1}$;
$(\Longleftarrow)$ follows along the same lines. For the second statement:
from their definitions, $K_{1}\leq K_{2}$ and hence $K_{1}^{*}\geq K_{2}^{*}$.
As a result, $F_{1,\mathfrak{a}}\leq F_{2,\mathfrak{a}}$ and consequently
$\gamma_{1}:=F_{1,\mathfrak{a}}^{-1}\leq F_{2,\mathfrak{a}}^{-1}=:\gamma_{2}$.
\end{proof}
We know from \cite{tierney-1998} that for $P_{1},P_{2}$ $\mu-$reversible,
then $\mathcal{E}(P_{1},g)\geq\mathcal{E}(P_{2},g)$ for any $g\in\ELL(\mu)$
implies ${\rm {\rm var}}(P_{1},f)\leq{\rm {\rm var}}(P_{2},f)$ for
$f\in\ELL(\mu)$ and ${\rm Gap}_{\mathrm{R}}(P_{1})\geq{\rm Gap}_{\mathrm{R}}(P_{2})$,
the latter being useful when ${\rm Gap}_{\mathrm{R}}(P_{1})>0$, and
say $P_{1}$ and $P_{2}$ are positive, since this implies faster
convergence to equilibrium in most scenarios of interest. The following
generalizes the latter statement to the subgeometric setup -- the
statement on asymptotic the variances remains naturally true.
\begin{thm}
Let $P_{1},P_{2}$ be $\mu-$invariant Markov kernels such that for
a sieve $\Phi$, 
\begin{enumerate}
\item $P_{1}^{*}P_{1}$ (resp. $P_{2}^{*}P_{2}$) satisfies a $(\Phi,\alpha_{1},\beta_{1})\shortminus$WPI
(resp. a $(\Phi,\alpha_{2},\beta_{2})$--WPI), 
\item $\mathcal{E}(P_{1}^{*}P_{1},g)\geq\mathcal{E}(P_{2}^{*}P_{2},g)$
for any $g\in\mathrm{L}^{2}(\mu)$ such that $\Phi(g)\leq1$.
\end{enumerate}
Then with $\alpha_{i}^{\star}(\cdot;\Phi)$ and $\beta_{i}^{\star}(\cdot;\Phi)$
for $i=1,2$ defined as in Definition~\ref{def:alpha-beta-star},
a $(\Phi,\alpha_{i}^{\star},\beta_{i}^{\star})\shortminus$WPI holds
for $i=1,2$ and we have for the corresponding convergence rates $\gamma_{1}^{\star}\leq\gamma_{2}^{\star}$. 
\end{thm}

\begin{proof}
From the ordering of Dirichlet forms we have for any $g\in\ELL(\mu)$
\[
\|g\|_{2}^{2}-s\mathcal{E}(P_{1}^{*}P_{1},g)\leq\|g\|_{2}^{2}-s\mathcal{E}(P_{2}^{*}P_{2},g),
\]
from Definition~\ref{def:alpha-beta-star} we deduce $\beta_{1}^{\star}(\cdot;\Phi)\leq\beta_{2}^{\star}(\cdot;\Phi)$
and from Theorem~\ref{thm:order-alpha-beta-gamma} we conclude $\gamma_{1}^{\star}\leq\gamma_{2}^{\star}$.
\end{proof}

\subsection{Optimal $\Phi$ \label{subsec:optimal-choice-phi}}

On the other hand, we can fix a bounded $\beta$, say and seek the
optimal class of functions defined by a sieve $\Phi$ for this $\beta$.
As a starting point, we assume that some $(\Phi,\beta)-$WPI holds
for $T=P^{*}P$:
\[
\|f\|_{2}^{2}\le s\calE(P^{*}P,f)+\beta(s)\Phi(f),\quad\forall s>0,f\in\ELL_{0}(\mu),
\]
for a given $\Phi$. By Theorem~\ref{thm:WPI_F_bd}, we obtain the
convergence bound:
\[
\|P^{n}f\|_{2}^{2}\le\Phi(f)\gamma(n),
\]
for a function $\gamma:\mathbb{N}_{0}\to\R_{+}$ which satisfies $\gamma(n)\to0$
as $n\to\infty$.

We now seek the smallest sieve $\Phi_{\beta}^{\star}$ such that a
$(\Phi_{\beta}^{\star},\beta)-$WPI still holds.
\begin{defn}
\label{def:Phi_star}We define for any $f\in\ELL_{0}(\mu)$,
\[
\Phi_{\beta}^{\star}(f):=\sup_{n\in\mathbb{N}_{0}}\Phi_{\beta}(P^{n}f),
\]
where 
\[
\Phi_{\beta}(f):=\sup_{s>0}\frac{\|f\|_{2}^{2}-s\calE(P^{*}P,f)}{\beta(s)}=\|f\|_{2}^{2}\cdot\sup_{s>0}\frac{1-s\delta(f)}{\beta(s)},
\]
where $\delta(f):=\calE(P^{*}P,f)/\|f\|_{2}^{2}$ and satisfies $0<\delta(f)\le1$.
\end{defn}

\begin{lem}
The functional $\Phi_{\beta}^{\star}$ is a nonexpansive sieve for
$P$.
\end{lem}

\begin{proof}
Note that for any $f\in\ELL_{0}(\mu)$, and $\Phi_{\beta}(cf)=c^{2}\Phi_{\beta}(f)$,
and furthermore $\Phi_{\beta}(f)\ge\|f\|_{2}^{2}/\beta(0)$, where
$\beta(0):=\lim_{s\to0}\beta(s)$, which exists and is finite and
nonzero by monotonicity and boundedness of $\beta$. Thus $\Phi_{\beta}$
satisfies condition~\ref{enu:def_Phi1} from Definition~\ref{def:Phi_fn}.

Now $\Phi_{\beta}^{\star}(f)\ge\Phi_{\beta}(f)$, and hence $\Phi_{\beta}^{\star}$
also satisfies condition~\ref{enu:def_Phi1} from Definition~\ref{def:Phi_fn}.
Finally, $\Phi_{\beta}^{\star}$ is nonexpansive for $P$ by construction,
and hence is a nonexpansive sieve.
\end{proof}
With this definition of $\Phi_{\beta}^{\star}$, it is clear that
we have a $(\Phi_{\beta}^{\star},\beta)-$WPI: for all $s>0$, $f\in\ELL_{0}(\mu)$,
\[
\|f\|_{2}^{2}\le s\calE(P^{*}P,f)+\beta(s)\Phi_{\beta}^{\star}(f),
\]
and so we can obtain the convergence bound
\[
\|P^{n}f\|_{2}^{2}\le\Phi_{\beta}^{\star}(f)\gamma(n),
\]
for the same $\gamma$, and by construction $\Phi_{\beta}^{\star}\le\Phi$.

\begin{example}
When $\beta(s)=s^{-\alpha}$, we can calculate that
\[
\Phi_{\beta}(f)=\|f\|_{2}^{2}\frac{\alpha^{\alpha}}{(\alpha+1)^{\alpha+1}}\left[\delta(f)\right]^{-\alpha}.
\]
Then we have
\begin{align*}
\delta(P^{n}f) & =\frac{\calE(P^{2},P^{n}f)}{\|P^{n}f\|_{2}^{2}}=\frac{\langle(\Id-P^{2})P^{n}f,P^{n}f\rangle}{\|P^{n}f\|_{2}^{2}}\\
 & =1-\frac{\int_{\sigma(P)}\lambda^{2n+2}\,\nu_{f}(\dif\lambda)}{\int_{\sigma(P)}\lambda^{2n}\,\nu_{f}(\dif\lambda)}.
\end{align*}
Thus the mapping $n\mapsto\int_{\sigma(P)}\lambda^{2n}\,\nu_{f}(\dif\lambda)$
will dictate for a given $f\in\ELL_{0}(\mu)$ whether or not $\Phi_{\beta}^{*}(f)$
is finite or infinite. As a concrete example, consider the situation
when $\sigma(P)=[0,1]$ and when $\nu_{f}(\dif\lambda)$ has density
proportional to $\lambda^{a-1}\dif\lambda$ for some $a>1$. Then
the $2n$th moment is $\prod_{r=0}^{2n-1}\frac{a+r}{a+1+r}=\frac{a}{a+2n}$,
and so the ratio is
\[
\frac{\int_{\sigma(P)}\lambda^{2n+2}\,\nu_{f}(\dif\lambda)}{\int_{\sigma(P)}\lambda^{2n}\,\nu_{f}(\dif\lambda)}=\frac{a+2n}{a+2+2n}.
\]
In particular, we find
\begin{align*}
1-\frac{\int_{\sigma(P)}\lambda^{2n+2}\,\nu_{f}(\dif\lambda)}{\int_{\sigma(P)}\lambda^{2n}\,\nu_{f}(\dif\lambda)} & =\frac{2}{a+2+2n}.
\end{align*}
Thus we see that asymptotically, $\Phi_{\beta}\left(P^{n}f\right)$
must grow like $\|P^{n}f\|_{2}^{2}\cdot n^{\alpha}$. This will diverge
to infinity as $n\to\infty$ if $n^{\alpha}$ dominates the rate of
convergence to 0 of $\|P^{n}f\|_{2}^{2}$. So informally speaking,
if we consider the set $\left\{ f\in\ELL_{0}\left(\mu\right):\Phi_{\beta}^{\star}\left(f\right)<\infty\right\} $,
$\Phi_{\beta}^{\star}$ is in effect `sieving out' functions $f\in\ELL_{0}(\mu)$
whose spectral measures $\nu_{f}(\dif\lambda)$ place too much mass
close to 1.

To be more explicit, by applying Chernoff's inequality to (\ref{eq:nu_f-moment-bd}),
we can conclude that for $f\in\ELL_{0}\left(\mu\right)$ with $\|f\|_{2}^{2}=1$,
for any $\delta>0$,
\[
\int_{1-\delta}^{1}\lambda^{2}\,\nu_{f}\left(\dif\lambda\right)\le C\cdot\Phi\left(f\right)\cdot\delta^{\alpha},
\]
for a constant $C>0$ independent of $f$, thus demonstrating that
$\nu_{f}$ cannot place mass in an arbitrary fashion in a neighbourhood
of $1$.
\end{example}

\subsection{Duality \label{subsec:optimal-choice-Duality}}

The preceding two sections suggest the following natural approach
to deriving convergence bounds and then refining them:
\begin{enumerate}
\item Choose a class of functions we seek convergence bounds for, and the
corresponding $\Phi$. For example, we could consider the class of
bounded functions and correspondingly take $\Phi=\|\cdot\|_{\mathrm{osc}}^{2}$.
As argued in Section~\ref{subsec:bounded_to_p}, this choice is in
a sense canonical.
\item Given this function class and its $\Phi$, derive an optimal $\beta^{\star}(\cdot;\Phi)$
for this class, as given in Definition~\ref{def:alpha-beta-star}.
\item Given this optimal $\beta^{\star}(\cdot;\Phi)$, find the optimal
$\Phi^{\star}:=\Phi_{\beta^{\star}(\cdot;\Phi)}^{\star}$, given in
Definition~\ref{def:Phi_star}.
\end{enumerate}
This procedure in fact is optimal after a single iteration; recursing
these steps does not lead to any improvement.
\begin{prop}
We have that
\[
\beta^{\star}(\cdot;\Phi^{\star})=\beta^{\star}(\cdot;\Phi).
\]
\end{prop}

\begin{proof}
By definition,
\begin{equation}
\beta^{\star}(s;\Phi^{\star})=\sup_{f\in\ELL_{0}(\mu),\Phi^{\star}(f)\le1}\left\{ \|f\|_{2}^{2}-s\calE(P^{*}P,f)\right\} .\label{eq:beta*_Phi*}
\end{equation}
Firstly, note that since $\Phi^{\star}$ is optimal,
\[
\Phi^{\star}(f)\le\Phi(f),\quad\forall f\in\ELL_{0}(\mu).
\]
Therefore,
\[
\left\{ f\in\ELL_{0}(\mu):\Phi(f)\le1\right\} \subset\left\{ f\in\ELL_{0}(\mu):\Phi^{\star}(f)\le1\right\} .
\]
Thus the supremum in the definition of $\beta^{\star}(s;\Phi^{\star})$
(\ref{eq:beta*_Phi*}) is over a larger class of functions than that
of $\beta^{\star}(s;\Phi)$ in Definition~\ref{def:alpha-beta-star}.
Therefore we can immediately conclude that
\begin{equation}
\beta^{\star}(s;\Phi^{\star})\ge\beta^{\star}(s;\Phi),\quad\forall s>0.\label{eq:beta**lebeta*}
\end{equation}
However, by definition, if $\Phi^{\star}(f)\le1$, we have that
\[
\sup_{s>0,n\in\mathbb{N}_{0}}\frac{\|P^{n}f\|_{2}^{2}-s\calE(P^{*}P,P^{n}f)}{\beta^{\star}(s;\Phi)}\le1,
\]
which in particular (taking $n=0$) implies that for any $s>0$,
\[
\|f\|_{2}^{2}-s\calE(P^{*}P,f)\le\beta^{\star}(s;\Phi).
\]
Thus 
\[
\beta^{\star}(s;\Phi^{\star})\le\beta^{\star}(s;\Phi),\quad\forall s>0,
\]
which taken together with (\ref{eq:beta**lebeta*}), establishes the
result.
\end{proof}

\section{Establishing WPIs\label{sec:Establishing-WPIs}}

\subsection{Cheeger meets Poincaré\label{subsec:Cheeger-meets-Poincar=0000E9}}

In this section we discuss the connections between weak Poincaré inequalities
and methods based on the concept of conductance. In particular, we
define the notion of weak conductance, which extends the traditional
definition of conductance to the subgeometric setting. Similar ideas
were proposed in \cite[Sections 4, 5]{Rockner2001} in the (continuous
time) diffusion setting, but our arguments differ significantly and
are inspired by the discrete-time proofs of \cite{lawler1988bounds,douc2018markov}.
We fix a $\mu$-reversible Markov transition kernel $P$ on our measure
space $(\E,\mathscr{E})$.
\begin{defn}
\label{def:capacity-function} For a $\mu$-reversible kernel $P$,
we define the \textit{weak conductance} $\kappa:[0,\infty)\to[0,\infty]$
to be 
\[
\kappa(u):=\inf_{A\in\mathscr{E}:u<\mu\otimes\mu(A\times A^{\complement})}\frac{\mathcal{E}(P,\mathbf{1}_{A})}{\|\mathbf{1}_{A}-\mu(A)\|_{2}^{2}}=\inf_{A\in\mathscr{E}:u<\mu\otimes\mu(A\times A^{\complement})}\frac{\mu\otimes P(A\times A^{\complement})}{\mu\otimes\mu(A\times A^{\complement})}.
\]
The last inequality follows from Lemma~\ref{lem:dirichlet-form-indicator}
in the Appendix. Note that since for any $A\in\calE$, $\mu\otimes\mu(A\times A^{\complement})\le1/4$,
by convention we have $\kappa(u)=\infty$ for $u\ge1/4$.
\end{defn}

The  definition of (strong) conductance \cite{lawler1988bounds}
is recovered by taking $u=0$; $\kappa(0)$ is Cheeger's constant,
which in the subgeometric case is 0. 
\begin{rem}
Following \cite{jerrum1988conductance} rather than \cite{lawler1988bounds},
some authors use a slightly different definition of conductance:
\[
\kappa_{*}:=\inf_{A\in\mathscr{E},\mu(A)\leq1/2}\frac{\mu\otimes P\big(A\times A^{\complement}\big)}{\mu(A)},
\]
which possesses a clear probabilistic interpretation. We note however
that $\kappa_{*}\leq\kappa(0)\leq2\kappa_{*}$, and the key quantity
used to establish Cheeger's inequalities, and our generalization,
relies on $\kappa$ as in Definition~\ref{def:capacity-function}.
\end{rem}

There is some resemblance between the weak conductance $\kappa$ and
the $s$-conductance introduced by \cite{lovasz1993random}. However,
it is not straightforward to compare the two or the type of convergence
results obtained; see, e.g., \cite[Lemma~2.1]{atchade2021approximate}.

Cheeger's inequality \cite{lawler1988bounds} obtains a lower bound
on $\mathcal{E}(P,f)/\|f\|_{2}^{2}$ for all $f\in\text{\ensuremath{\ELL_{0}(\mu)}}$,
$f\neq0$, from a lower bound on this same quantity when restricted
to functions $f=\mathbf{1}_{A}-\mu(A)$ for $A\in\mathscr{E}$ (namely,
$\kappa(0)$). This leads to the following celebrated inequalities
when $\kappa(0)>0$:
\begin{equation}
\kappa^{2}(0)/8\leq{\rm Gap}_{\mathrm{R}}(P)\le\kappa(0).\label{eq:cheeger-ineqs}
\end{equation}
We generalize this idea to the scenario where the quantity $\kappa(0)$
is zero, so there is no right-spectral gap. As we shall see, this
generalization involves an upper and lower bound for the function
$\alpha$ in (\ref{eq:WPI}). 

This generalization will be particularly useful when we seek to establish
the existence of WPIs from the abstract RUPI condition in Section~\ref{subsec:establish-WPI-WPIs-from-RUPI}.

\begin{thm}
\label{thm:WPI_from_cond}Let $P$ be a $\mu$-reversible kernel and
$\Phi=\|\cdot\|_{\osc}^{2}$. 

Provided that $\kappa(u)>0$ for all $u\in(0,1/4)$, a $(\Phi,\alpha)\shortminus$\textup{WPI}
holds for $P$, with
\[
\alpha(r):=\frac{16}{\kappa^{2}(r/16)},\quad r>0.
\]
 Conversely, if a $(\|\cdot\|_{{\rm osc}}^{2},\alpha)\shortminus$WPI
holds for some $\alpha:(0,\infty)\to[0,\infty)$, we have the bound
\begin{equation}
\frac{1}{\alpha(r)}\le\inf_{u>1}\left\{ \kappa(ur)\frac{u}{u-1}\right\} \leq2\kappa(2r),\quad r>0.\label{eq:cheeger-lower}
\end{equation}
\end{thm}

\begin{rem}
\label{rem:cheeger-remark}In the notation of Section~\ref{sec:Optimal-choices},
and in analogue with (\ref{eq:cheeger-ineqs}), we can succinctly
express this theorem in terms of the optimal $\alpha^{\star}$ as:
\[
\frac{\kappa^{2}(r/16)}{16}\leq1/\alpha^{\star}(r)\leq\inf_{s>1}\left\{ \frac{s}{s-1}\kappa(sr)\right\} \leq2\kappa(2r),\quad r>0.
\]
From Theorem~\ref{thm:order-alpha-beta-gamma}, inequality (\ref{eq:cheeger-lower})
implies that convergence to equilibrium cannot occur at a rate $\gamma$
faster fast than that obtained with $\underline{\alpha}(r)=[2\kappa(2r)]^{-1}$.
\end{rem}

The proof is a direct consequence of Propositions~\ref{prop:cheeger-lower}
and \ref{prop:cheeger-upper}. We first show that the conductance
always provides a lower bound for $\alpha$ if a $(\|\cdot\|_{{\rm osc}}^{2},\alpha)\shortminus$WPI
holds.
\begin{prop}
\label{prop:cheeger-lower}Let $P$ be a $\mu$-reversible kernel
satisfying a $(\|\cdot\|_{\osc}^{2},\alpha)\shortminus$WPI. We have
the bound (\ref{eq:cheeger-lower}).
\end{prop}

\begin{proof}
Consider the function 
\[
f=\frac{\mathbf{1}_{A}-\mu(A)}{\sqrt{\mu(A)\mu(A^{\complement})}},
\]
for a measurable set $A\in\mathscr{E}$ such that $1>\mu(A)>0$. By
construction, $f\in\ELL_{0}(\mu)$ with $\|f\|_{2}^{2}=1$. Plugging
this into the weak Poincaré inequality, we find that for any $r>0$,
\[
1\le\alpha(r)\frac{\mu\otimes P(A\times A^{\complement})}{\mu\otimes\mu(A\times A^{\complement})}+\frac{r}{\mu\otimes\mu(A\times A^{\complement})}.
\]
Rearranging this, we obtain that for any $r>0$,
\[
\frac{1}{\alpha(r)}\left(1-\frac{r}{\mu\otimes\mu(A\times A^{\complement})}\right)\leq\frac{\mu\otimes P(A\times A^{\complement})}{\mu\otimes\mu(A\times A^{\complement})}.
\]
Now for any $s>r>0$, we consider only $A\in\mathscr{E}$ such that
$\mu\otimes\mu(A\times A^{\complement})>s$, yielding
\[
1/\alpha(r)\leq\left(1-\frac{r}{s}\right)^{-1}\frac{\mu\otimes P(A\times A^{\complement})}{\mu\otimes\mu(A\times A^{\complement})}.
\]
Therefore for $r>0$ we have
\[
1/\alpha(r)\leq\inf_{s>r}\frac{s}{s-r}\kappa(s)=\inf_{u>1}\frac{u}{u-1}\kappa(ru)\leq2\kappa(2r).
\]
where we have used the change of variable $u=s/r$ for the equality
and taken $u=2$ for the final inequality.
\end{proof}
We now prove the trickier converse: we show that the weak conductance
gives rise to an $\alpha$ such that a $(\|\cdot\|_{{\rm osc}}^{2},\alpha)\shortminus$WPI
holds. We make use of the fundamental Lemma~\ref{lem:lawler-sokal-beautiful}
of \cite{lawler1988bounds} which provides a bridge between Dirichlet
forms of indicator functions and general functions and can be found
in the appendix for the reader's convenience. 
\begin{prop}
\label{prop:cheeger-upper}Let $P$ be a $\mu$-reversible kernel.
Then provided $\kappa(u)>0$ for all $u\in(0,1/4)$, a $(\|\cdot\|_{{\rm osc}}^{2},\alpha)-$WPI
holds with
\begin{equation}
\alpha(r):=\frac{16}{\kappa^{2}\left(r/16\right)},\quad r>0.\label{eq:alpha_k}
\end{equation}
\end{prop}

\begin{proof}
Let us fix $f\in\ELL_{0}(\mu)$ with $\|f\|_{2}^{2}=1$. Our goal
is to show that $\alpha$ as defined in (\ref{eq:alpha_k}) gives
rise to a valid weak Poincaré inequality for $P$ with $\Phi=\|\cdot\|_{\osc}^{2}$;
since we have fixed $\|f\|_{2}^{2}=1$ this amounts to showing that
for $r>0$,
\[
1\le\frac{16}{\kappa^{2}\left(r/16\right)}\calE(P,f)+r\|f\|_{\osc}^{2}.
\]
We make use of the following two results, the proof of which can be
found in \cite{lawler1988bounds,douc2018markov,Sherlock2018}. Let
$g:=f+c$ for $c\in\R$. Firstly, it can be shown using the Cauchy--Schwarz
inequality that
\begin{equation}
\frac{\mathbb{E}_{\mu\otimes P}\big[|g^{2}(X)-g^{2}(Y)|\big]^{2}}{\mathbb{E}_{\mu}\left[g^{2}(X)\right]}\leq8\mathcal{E}(P,f).\label{eq:cheeger_dirichlet}
\end{equation}
Note that since $\|f\|_{2}^{2}=1$ and $\mu(f)=0$, $\mathbb{E}_{\mu}\left[g^{2}(X)\right]=\|g-c+c\|^{2}=1+c^{2}$.
Secondly, it can also be established (following the proof in \cite{Sherlock2018},
say) that
\begin{equation}
\max\left\{ \lim_{c\to\infty}\frac{\mathbb{E}_{\mu\otimes\mu}\left[|g^{2}(X)-g^{2}(Y)|\right]^{2}}{\mathbb{E}_{\mu}[g^{2}(X)]},\frac{\mathbb{E}_{\mu\otimes\mu}\left[|f^{2}(X)-f^{2}(Y)|\right]^{2}}{1}\right\} \geq1,\label{eq:cheeger_max_c}
\end{equation}
where the second term in the braces corresponds to the choice $c=0$.
The bound in (\ref{eq:cheeger_max_c}) is used below to lower bound
the left-hand side of (\ref{eq:cheeger_dirichlet}). Consider the
family of sets $\mathcal{T}_{s}:=\{t\ge0:\mu\otimes\mu(A_{t},A_{t}^{\complement})>s\}\subset[0,\infty)$
for $s>0$. Then using successively Lemma~\ref{lem:lawler-sokal-beautiful}
with $\nu=\mu\otimes\mu$, the bound $(a+b)^{2}\le2a^{2}+2b^{2}$,
the definition of $\kappa(s)$, Lemma~\ref{lem:lawler-sokal-beautiful}
with $\nu=\mu\otimes P$ and (\ref{eq:cheeger_dirichlet}), we obtain
for any $c\in\mathbb{R}$ and $s>0$,
\begin{align*}
\frac{\mathbb{E}_{\mu\otimes\mu}\left[|g^{2}(X)-g^{2}(Y)|\right]^{2}}{2\mathbb{E}_{\mu}[g^{2}(X)]}\\
 & \hspace{-2cm}=\frac{\left(2\int_{0}^{\infty}\mu\otimes\mu(A_{t},A_{t}^{\complement})\,{\rm d}t\right)^{2}}{2\mathbb{E}_{\mu}[g^{2}(X)]}\\
 & \hspace{-2cm}=\frac{\left(2\int_{\mathcal{T}_{s}}\mu\otimes\mu(A_{t},A_{t}^{\complement})\,{\rm d}t+2\int_{\mathcal{T}_{s}^{\complement}}\mu\otimes\mu(A_{t},A_{t}^{\complement})\,{\rm d}t\right)^{2}}{2\mathbb{E}_{\mu}[g^{2}(X)]}\\
 & \hspace{-2cm}\le\frac{\left(2\int_{\mathcal{T}_{s}}\mu\otimes\mu(A_{t},A_{t}^{\complement})\,{\rm d}t\right)^{2}}{\mathbb{E}_{\mu}[g^{2}(X)]}+\frac{\left(2\int_{\mathcal{T}_{s}^{\complement}}\mu\otimes\mu(A_{t},A_{t}^{\complement})\,{\rm d}t\right)^{2}}{\mathbb{E}_{\mu}[g^{2}(X)]}\\
 & \hspace{-2cm}\le\frac{\left(\frac{1}{\kappa(s)}2\int_{\mathcal{T}_{s}}\mu\otimes P(A_{t}\times A_{t}^{\complement})\,{\rm d}t\right)^{2}}{\mathbb{E}_{\mu}[g^{2}(X)]}+\frac{\left(2\int_{\mathcal{T}_{s}^{\complement}}\mu\otimes\mu(A_{t},A_{t}^{\complement})\,{\rm d}t\right)^{2}}{\mathbb{E}_{\mu}[g^{2}(X)]}\\
 & \hspace{-2cm}\le\frac{\frac{1}{\kappa^{2}(s)}\mathbb{E}_{\mu\otimes P}\big[|g^{2}(X)-g^{2}(Y)|\big]^{2}}{\mathbb{E}_{\mu}[g^{2}(X)]}+\frac{\left(2\int_{\mathcal{T}_{s}^{\complement}}\mu\otimes\mu(A_{t},A_{t}^{\complement})\,{\rm d}t\right)^{2}}{\mathbb{E}_{\mu}[g^{2}(X)]}\\
 & \hspace{-2cm}\le\frac{8}{\kappa^{2}(s)}\mathcal{E}(P,f)+\frac{\left(2\int_{\mathcal{T}_{s}^{\complement}}\mu\otimes\mu(A_{t},A_{t}^{\complement})\,{\rm d}t\right)^{2}}{\mathbb{E}_{\mu}[g^{2}(X)]}.
\end{align*}
We now focus on the second term. We begin with the case $c=0$. 
For $t\in\mathcal{T}_{s}^{\complement}$,
\[
\mu(A_{t})\mu(A_{t}^{\complement})=\mu(g^{2}(X)\geq t)\mu(g^{2}(X)<t)\leq s.
\]
In particular, since we are assuming that $\|f\|_{\infty}<\infty$,
if $t>\left(\|f\|_{\infty}+|c|\right)^{2}$, then $\mu(g^{2}(X)\geq t)=0$.
This enables us to bound, in the case $c=0$: since we have $\|f\|_{\infty}^{2}\le\|f\|_{\osc}^{2}$,
\begin{align*}
\int_{\mathcal{T}_{s}^{\complement}}\mu\otimes\mu(A_{t},A_{t}^{\complement})\,{\rm d}t & \le\int_{0}^{\|f\|_{\infty}^{2}}s\,\dif t\\
 & \le s\|f\|_{\osc}^{2}.
\end{align*}
From Lemma~\ref{lem:lawler-sokal-beautiful} we also have the bound
\begin{align*}
\int_{\mathcal{T}_{s}^{\complement}}\mu\otimes\mu(A_{t},A_{t}^{\complement})\,{\rm d}t & \le\mathbb{E}_{\mu\otimes\mu}\left[|g^{2}(X)-g^{2}(Y)|\right]\\
 & \leq2\mathbb{E}_{\mu}\left[f^{2}(X)\right]=2\,.
\end{align*}
Using these two bounds to upper bound the square below, we obtain
that for $c=0$,
\[
\frac{\left(2\int_{\mathcal{T}_{s}^{\complement}}\mu\otimes\mu(A_{t},A_{t}^{\complement})\,{\rm d}t\right)^{2}}{\mathbb{E}_{\mu}[g^{2}(X)]}\le4s\|f\|_{\osc}^{2}\cdot2.
\]
We now consider the case $c\to\infty$. Since we are interested in
the case when $c>\|f\|_{\infty}$, we know that $g>0$ everywhere.
In particular, this implies that if $t>\left(\muess\sup f+c\right)^{2}$,
then $\mu(g^{2}(X)\geq t)=0$. Similarly, if $t<(c+\muess\,\inf f)^{2}$,
then $\mu(g^{2}(X)<t)=0$. Thus we bound
\begin{align*}
\int_{\mathcal{T}_{s}^{\complement}}\mu\otimes\mu(A_{t},A_{t}^{\complement})\,{\rm d}t\\
 & \hspace{-1.5cm}=\int_{0}^{\infty}\mu\otimes\mu(A_{t},A_{t}^{\complement})1_{\mathcal{T}_{s}^{\complement}}(t)\,{\rm d}t\\
 & \hspace{-1.5cm}=\int_{\left(c+\muess\inf f\right)^{2}}^{\left(\muess\sup f+c\right)^{2}}\mu\otimes\mu(A_{t},A_{t}^{\complement})1_{\mathcal{T}_{s}^{\complement}}(t)\,{\rm d}t\\
 & \hspace{-1.5cm}\le s\left[\left(\muess\sup f+c\right)^{2}-\left(c+\muess\,\inf f\right)^{2}\right]\\
 & \hspace{-1.5cm}=s\left[\left(\muess\sup f\right)^{2}-\left(\mu\shortminus\muess\,\inf f\right)^{2}+2c\left(\muess\sup f-\muess\,\inf f\right)\right]\\
 & \hspace{-1.5cm}=s\left[\left(\muess\sup f\right)^{2}-\left(\muess\,\inf f\right)^{2}+2c\|f\|_{\osc}\right].
\end{align*}
So ultimately we obtain
\begin{multline*}
\frac{\left(2\int_{\mathcal{T}_{s}^{\complement}}\mu\otimes\mu(A_{t},A_{t}^{\complement})\,{\rm d}t\right)^{2}}{\mathbb{E}_{\mu}[g^{2}(X)]}\le s\frac{4\left[\left(\muess\sup f\right)^{2}-\left(\muess\,\inf f\right)^{2}+2c\|f\|_{\osc}\right]}{\sqrt{1+c^{2}}}\\
\times\frac{\mathbb{E}_{\mu\otimes\mu}\left[|g^{2}(X)-g^{2}(Y)|\right]}{\sqrt{\mathbb{E}_{\mu}[g^{2}(X)]}}.
\end{multline*}
Then taking the limit, we get
\begin{align*}
\limsup_{c\to\infty}\frac{\left(2\int_{\mathcal{T}_{s}^{\complement}}\mu\otimes\mu(A_{t},A_{t}^{\complement})\,{\rm d}t\right)^{2}}{\mathbb{E}_{\mu}[g^{2}(X)]} & \le s\cdot8\|f\|_{\osc}\cdot2\|f\|_{\osc}\\
 & =s\cdot16\|f\|_{\osc}^{2}.
\end{align*}
Rearranging then gives the desired bound.
\end{proof}

\subsection{WPIs from RUPI and $\mu$-irreducibility\label{subsec:establish-WPI-WPIs-from-RUPI}}

Given our notion of a WPI in Definition~\ref{def:WPI}, a natural
question to ask is under what general conditions on a kernel $T$,
a WPI for $T$ will hold. In particular, a WPI for the kernel $T=\left(P^{*}\right)^{k}P^{k}$
for $k\in\mathbb{N}$ enables one to deduce (subgeometric) convergence
bounds for $\|P^{kn}f\|_{2}$, where $f\in\ELL_{0}(\mu)$ is such
that $\Phi(f)<\infty$. Thus, we seek simple conditions on a Markov
kernel $T$ under which (\ref{eq:WPI}) will hold, for sieve $\Phi=\|\cdot\|_{\osc}^{2}$,
with $T=P$ or $T=(P^{*})^{k}P^{k}$ for $k\in\mathbb{N}$, for a
finite-valued function $\alpha$. 

We will see that for a Markov operator $T$, a necessary and sufficient
condition for a $(\|\cdot\|_{\osc}^{2},\alpha)-$WPI to hold is the
\textit{resolvent-uniform-positivity-improving} (RUPI) property. This
property appeared in \cite{Gong2006}, and in \cite{Wang2014} it
was suggested that an equivalence between the RUPI property and the
existence of a WPI was already established in an unpublished manuscript
by L.~Wu. However, we have not been able to access this manuscript,
and so in Section~\ref{subsec:Equivalence-of--WPI} we provide a
direct proof of this equivalence.

In Section~\ref{subsec:Holding-probabilities,-WPIs} we will demonstrate
that arbitrarily small, uniform holding probabilities allow one to
relate the existence of $\|\cdot\|_{\osc}^{2}$-WPIs for $P$, $P^{*}P$
and $\|\cdot\|_{\osc}^{2}$-convergence of $P$ (see Proposition~\ref{prop:hold-equiv-wpi-conv-pp}),
and also to deduce that $\|\cdot\|_{\osc}^{2}$-convergence of $P$
and its additive reversibilization can similarly be closely related
with a non-zero holding probabilities (see Proposition~\ref{prop:P/S-convergence-P/S}).

Furthermore, a simple sufficient condition for RUPI (and hence a WPI)
is \textit{$\mu$-irreducibility}, which we discuss in detail in Section~\ref{subsec:-irreducibility-implies-a};
see Corollary~\ref{cor:WPI_from_irred}. 

Hereafter we may omit the statement $A,B\in\mathscr{E}$ to alleviate
notation; no confusion should be possible.

\subsubsection{Equivalence of $\left\Vert \cdot\right\Vert _{{\rm osc}}^{2}$-WPI
and RUPI\label{subsec:Equivalence-of--WPI}}
\begin{defn}[UPI and RUPI]
A kernel $T$ is \textit{uniform-positivity-improving} (UPI) if for
each $\epsilon>0$,
\[
\inf\{\langle\mathbf{1}_{A},T\mathbf{1}_{B}\rangle:\mu(A)\wedge\mu(B)\ge\epsilon\}>0.
\]
A Markov kernel $T$ is said to be \textit{resolvent-uniform-positivity-improving}
(RUPI) if for some (and hence all) $0<\lambda<1$, we have that the
resolvent
\[
R\left(\lambda,T\right):=\sum_{n=0}^{\infty}\lambda^{n}T^{n}=(\Id-\lambda T)^{-1}\,,
\]
is UPI.
\end{defn}

\begin{thm}
\label{thm:WPI_iff_RUPI}Suppose that $T$ is a $\mu$--invariant
Markov kernel. Then $T$ satisfies an $\left\Vert \cdot\right\Vert _{{\rm osc}}^{2}$-WPI
if and only if $T$ is RUPI.
\end{thm}

\begin{proof}
This follows from Proposition~\ref{prop:RUPI-to-WPI} and Proposition~\ref{prop:wpi-to-rupi}
below.
\end{proof}
We follow \cite{Wang2014} and give an equivalent condition for RUPI
which will be convenient to work with.
\begin{lem}
\label{lem:RUPI_equiv}An equivalent condition for a Markov kernel
$T$ to be RUPI is the following: for any $\epsilon>0$, there exists
$m\in\mathbb{N}$ such that 
\begin{equation}
\inf\left\{ \left\langle \mathbf{1}_{A},\sum_{n=0}^{m}T^{n}\mathbf{1}_{B}\right\rangle :\mu(A)\wedge\mu(B)\ge\epsilon\right\} >0.\label{eq:RUPI-equivalence}
\end{equation}
\end{lem}

\begin{proof}
The condition in Lemma~\ref{lem:RUPI_equiv} directly implies RUPI.
To see this, take $\lambda\in\left(0,1\right)$, $\epsilon>0$ and
let $m\in\mathbb{N}$ such that $\inf\left\{ \left\langle \mathbf{1}_{A},\sum_{n=0}^{m}T^{n}\mathbf{1}_{B}\right\rangle :\mu(A)\wedge\mu(B)\ge\epsilon\right\} >0$,
which exists by assumption. Write

\begin{align*}
\left\langle \mathbf{1}_{A},R\left(\lambda,T\right)\mathbf{1}_{B}\right\rangle  & =\left\langle \mathbf{1}_{A},\sum_{n=0}^{\infty}\lambda^{n}T^{n}\mathbf{1}_{B}\right\rangle \\
 & \geqslant\left\langle \mathbf{1}_{A},\sum_{n=0}^{m}\lambda^{n}T^{n}\mathbf{1}_{B}\right\rangle \\
 & \geqslant\lambda^{m}\cdot\left\langle \mathbf{1}_{A},\sum_{n=0}^{m}T^{n}\mathbf{1}_{B}\right\rangle 
\end{align*}
to deduce that

\begin{multline*}
\inf\{\left\langle \mathbf{1}_{A},R\left(\lambda,T\right)\mathbf{1}_{B}\right\rangle :\mu(A)\wedge\mu(B)\ge\epsilon\}\\
\geqslant\lambda^{m}\cdot\inf\left\{ \left\langle \mathbf{1}_{A},\sum_{n=0}^{m}T^{n}\mathbf{1}_{B}\right\rangle :\mu(A)\wedge\mu(B)\ge\epsilon\right\} >0
\end{multline*}
from which the RUPI condition follows.

Conversely, suppose that $T$ is RUPI, fix $\lambda\in\left(0,1\right)$
and assume that for some $\epsilon>0$, (\ref{eq:RUPI-equivalence})
does not hold for any $m\in\mathbb{N}$. We show that this leads to
a contradiction. By the RUPI assumption we have that $\delta:=\inf\{\langle\mathbf{1}_{A},\sum_{n=0}^{\infty}\lambda^{n}T^{n}\mathbf{1}_{B}\rangle:\mu(A)\wedge\mu(B)\ge\epsilon\}>0$.
Choose $m\in\mathbb{N}$ large enough so that
\[
\sum_{n=m+1}^{\infty}\lambda^{n}<\delta/2.
\]
Since we have assumed that (\ref{eq:RUPI-equivalence}) is violated
for $\epsilon>0$ and $m\in\mathbb{N}$ as chosen above there exists
a sequence $\left\{ (A_{j},B_{j})\right\} _{j=1}^{\infty}$ of sets
all with mass at least $\epsilon$ such that $\langle\mathbf{1}_{A_{j}},\sum_{n=0}^{m}T^{n}\mathbf{1}_{B_{j}}\rangle\to0$,
therefore implying for any $j\in\mathbb{N}$,
\begin{align*}
\delta\leq\left\langle \mathbf{1}_{A_{j}},\sum_{n=0}^{\infty}\lambda^{n}T^{n}\mathbf{1}_{B_{j}}\right\rangle  & =\left\langle \mathbf{1}_{A_{j}},\sum_{n=0}^{m}\lambda^{n}T^{n}\mathbf{1}_{B_{j}}\right\rangle +\left\langle \mathbf{1}_{A_{j}},\sum_{n=m+1}^{\infty}\lambda^{n}T^{n}\mathbf{1}_{B_{j}}\right\rangle \\
 & \leqslant\left\langle \mathbf{1}_{A_{j}},\sum_{n=0}^{m}T^{n}\mathbf{1}_{B_{j}}\right\rangle +\sum_{n=m+1}^{\infty}\lambda^{n}\\
 & \leqslant\left\langle \mathbf{1}_{A_{j}},\sum_{n=0}^{m}T^{n}\mathbf{1}_{B_{j}}\right\rangle +\frac{\delta}{2}\\
 & \overset{j\to\infty}{\rightarrow}\frac{\delta}{2}\,,
\end{align*}
therefore leading to a contradiction. The conclusion follows.
\end{proof}

We first establish that for reversible kernels, RUPI implies a WPI
for the resolvent.
\begin{lem}
\label{prop:RUPI_resolvent}Suppose that a reversible Markov kernel
$T$ is RUPI. Then for any $\lambda\in(0,1)$, the resolvent Markov
kernel $S_{\lambda}:=(1-\lambda)R(\lambda,T)$ is reversible and has
the following property: for any $\epsilon>0$,
\[
\inf_{A:\mu(A)\mu(A^{\complement})\ge\epsilon}\frac{\mathcal{E}(S_{\lambda},\mathbf{1}_{A})}{\mu(A)\mu(A^{\complement})}>0.
\]
Thus by Theorem~\ref{thm:WPI_from_cond}, $S_{\lambda}$ satisfies
an $\|\cdot\|_{\osc}^{2}$-WPI.
\end{lem}

\begin{proof}
Fix $\epsilon>0$ and $\lambda\in(0,1)$. By the RUPI condition, $\inf\{\langle\mathbf{1}_{A},S_{\lambda}\mathbf{1}_{B}\rangle:\mu(A)\wedge\mu(B)\ge\epsilon\}>0$.
In particular, if $A$ is such that $\mu(A)\mu(A^{\complement})\ge\epsilon$,
we must have that both $\mu(A)\ge\epsilon$ and $\mu(A^{\complement})\ge\epsilon$.
Thus since 
\[
\mathcal{E}(S_{\lambda},\mathbf{1}_{A})=\langle\mathbf{1}_{A},S_{\lambda}\mathbf{1}_{A^{\complement}}\rangle,
\]
we must have that
\[
\langle\mathbf{1}_{A},S_{\lambda}\mathbf{1}_{A^{\complement}}\rangle\ge\delta>0
\]
for some $\delta>0$, whenever $\mu(A)\mu(A^{\complement})\ge\epsilon$.
\end{proof}
We now establish one direction of Theorem~\ref{thm:WPI_iff_RUPI}
through a sequence of lemmas: we first consider the case when $T$
is reversible, and then deduce the case for general $T$; see Remark~\ref{rem:nonrev_WPI}.
\begin{lem}
\label{lem:T_rev_RUPI}Suppose $T$ is a reversible Markov kernel
that is RUPI. Then $T$ satisfies an $\|\cdot\|_{\osc}^{2}$-WPI.
\end{lem}

\begin{proof}
Since $T$ is RUPI, we have established above in Lemma~\ref{prop:RUPI_resolvent}
that the resolvent $S_{\lambda}:=(1-\lambda)R(\lambda,T)$ satisfies
a WPI. In other words, we can find some $\alpha_{\lambda}:(0,\infty)\to[0,\infty)$
such that for any $f\in\mathrm{L}_{0}^{2}(\mu)$ and $r>0$, 
\[
\|f\|_{2}^{2}\le\alpha_{\lambda}(r)\langle({\rm Id}-S_{\lambda})f,f\rangle+r\|f\|_{\osc}^{2}.
\]
Now, given a function $g\in\mathrm{L}_{0}^{2}(\mu)$, define $f:=\frac{\Id-\lambda T}{1-\lambda}g\Leftrightarrow g=(1-\lambda)(\Id-\lambda T)^{-1}f$.
(Note that since $0<\lambda<1$, the operator is $(\Id-\lambda T)$
invertible.)

Now since $g\in\mathrm{L}_{0}^{2}(\mu)$, we have that $f\in\mathrm{L}_{0}^{2}(\mu)$;
for instance, consider the power series representation of $R(\lambda,T)$.
Furthermore, we have that 
\begin{align*}
\|g\|_{2}^{2} & =(1-\lambda)^{2}\|(\Id-\lambda T)^{-1}f\|_{2}^{2}\\
 & \le(1-\lambda)^{2}\|(\Id-\lambda T)^{-1}\|^{2}\|f\|_{2}^{2}\\
 & \le\|f\|_{2}^{2},
\end{align*}
since the operator norm $\|({\rm Id}-\lambda T)^{-1}\|\le\frac{1}{\lambda}\cdot\frac{1}{1/\lambda-1}=\frac{1}{1-\lambda}$,
by standard norm bounds for resolvents based on the distance to the
spectrum.

Thus we have
\begin{align*}
\|g\|_{2}^{2} & \le\|f\|_{2}^{2}\le\alpha_{\lambda}(r)\langle(\Id-S_{\lambda})f,f\rangle+r\|f\|_{\osc}^{2}\\
 & =\alpha_{\lambda}(r)\langle(\Id-(1-\lambda)(\Id-\lambda T)^{-1})f,f\rangle+r\|f\|_{\osc}^{2}\\
 & =\alpha_{\lambda}(r)\left\langle \frac{\Id-\lambda T}{1-\lambda}g-g,\frac{\Id-\lambda T}{1-\lambda}g\right\rangle +r\|f\|_{\osc}^{2}\\
 & =\alpha_{\lambda}(r)\left\langle \frac{\lambda}{1-\lambda}(\Id-T)g,g+\frac{\lambda}{1-\lambda}(\Id-T)g\right\rangle +r\|f\|_{\osc}^{2}\\
 & =\alpha_{\lambda}(r)\left\{ \frac{\lambda}{1-\lambda}\langle(\Id-T)g,g\rangle+\left(\frac{\lambda}{1-\lambda}\right)^{2}\|(\Id-T)g\|^{2}\right\} +r\|f\|_{\osc}^{2}.
\end{align*}
Now we have that
\begin{align*}
\|(\Id-T)g\|^{2} & =\langle(\Id-T)g,(\Id-T)g\rangle\\
 & =\langle(\Id-T)g,g\rangle-\langle(\Id-T)g,Tg\rangle.
\end{align*}
It is enough to bound this final term by
\[
-\langle(\Id-T)g,Tg\rangle\le\langle(\Id-T)g,g\rangle.
\]
To see why this inequality is true, note that it is equivalent to
\begin{align*}
0 & \le\langle(\Id-T)g,(\Id+T)g\rangle\\
 & =\langle(\Id+T)(\Id-T)g,g\rangle\\
 & =\langle(\Id-T^{2})g,g\rangle,
\end{align*}
where we have made use of reversibility of $T$. And we certainly
have that $0\le\langle(\Id-T^{*}T)g,g\rangle=\langle(\Id-T^{2})g,g\rangle$. 

Overall, this gives us that 
\begin{align*}
\|g\|^{2} & \le\alpha_{\lambda}(r)\left(\frac{\lambda}{1-\lambda}+2\left(\frac{\lambda}{1-\lambda}\right)^{2}\right)\langle(\Id-T)g,g\rangle+r\left\Vert \frac{\Id-\lambda T}{1-\lambda}g\right\Vert _{\osc}^{2}\\
 & \le\alpha_{\lambda}(r)\left(\frac{\lambda}{1-\lambda}+2\left(\frac{\lambda}{1-\lambda}\right)^{2}\right)\langle(\Id-T)g,g\rangle+r\cdot\frac{(1+\lambda)^{2}}{\left(1-\lambda\right)^{2}}\|g\|_{\osc}^{2}.
\end{align*}
By reparameterizing with $r'=r\cdot\frac{(1+\lambda)^{2}}{\left(1-\lambda\right)^{2}}$,
this is a standard WPI for $T$.
\end{proof}
\begin{prop}
\label{prop:RUPI-to-WPI}Suppose a $\mu$-invariant Markov kernel
$T$ is RUPI. Then $T$ satisfies an $\|\cdot\|_{\osc}^{2}$-WPI.
\end{prop}

\begin{proof}
It suffices to show that $(T+T^{*})/2$ is RUPI, as then by Lemma~\ref{lem:T_rev_RUPI},
$(T+T^{*})/2$ possesses a WPI, which is equivalent to $T$ possessing
a WPI (see Remark~\ref{rem:nonrev_WPI}). Since $T$ is RUPI, for
any $\epsilon>0$, we can find some $\delta>0$ and $N\in\mathbb{N}$
such that whenever $\mu(A)\wedge\mu(B)\ge\epsilon$,
\begin{equation}
\left\langle \mathbf{1}_{A},\sum_{n=0}^{N}T^{n}\mathbf{1}_{B}\right\rangle \ge\delta>0.\label{eq:lem_T_RUPI}
\end{equation}
So now we wish to obtain such a statement for the kernel $(T+T^{*})/2$.
So fix $\epsilon>0$, and consider 
\[
\left\langle \mathbf{1}_{A},\sum_{n=0}^{N}\left(\frac{T+T^{*}}{2}\right)^{n}\mathbf{1}_{B}\right\rangle =\left\langle \mathbf{1}_{A},\sum_{n=0}^{N}\frac{T^{n}}{2^{n}}\mathbf{1}_{B}\right\rangle +\left\langle \mathbf{1}_{A},R\mathbf{1}_{B}\right\rangle ,
\]
where $R$ is a sum of operators of the form $cT^{a_{1}}\left(T^{*}\right)^{b_{1}}\cdot\dots\cdot T^{a_{r}}\left(T^{*}\right)^{b_{r}}$
for some $r\in\mathbb{N}$, $a_{i},b_{i}\in\mathbb{N}_{0}$ for all
$i=1,\dots,r$ and $c\ge0$. Thus since $T$ and $T^{*}$ are Markov
kernels, we have that $\langle\mathbf{1}_{A},R\mathbf{1}_{B}\rangle\ge0.$
So we can continue and have, for any sets with $\mu(A)\wedge\mu(B)\ge\epsilon$,
\begin{align*}
\left\langle \mathbf{1}_{A},\sum_{n=0}^{N}\left(\frac{T+T^{*}}{2}\right)^{n}\mathbf{1}_{B}\right\rangle  & \ge\left\langle \mathbf{1}_{A},\sum_{n=0}^{N}\frac{T^{n}}{2^{n}}\mathbf{1}_{B}\right\rangle \\
 & \ge2^{-N}\langle\mathbf{1}_{A},\sum_{n=0}^{N}T^{n}\mathbf{1}_{B}\rangle\\
 & \ge\delta/2^{N}>0,
\end{align*}
since each summand is positive, and we have used the fact that $T$
is RUPI (\ref{eq:lem_T_RUPI}).
\end{proof}
For the other direction, we first prove some auxiliary lemmas.

Lemma~\ref{lem:stick-dirichlet-lower-bound} is a general state space
extension of the argument referenced by \cite[Remark~2.16]{montenegro2006mathematical}.
\begin{lem}
\label{lem:stick-probs-equal}$P(x,\{x\})=P^{*}(x,\{x\})$ for $\mu$-almost
all $x$.
\end{lem}

\begin{proof}
Let $D=\{(x,y)\in\mathsf{E}^{2}:(x=y)\}$, $s(x):=P(x,\{x\})$ and
$s^{*}(x):=P^{*}(x,\{x\})$ for $x\in\mathsf{E}$. For any $B\in\mathcal{E}$,
we have 
\[
\mu({\bf 1}_{B}\cdot s)=\mu\otimes P(D\cap B^{2})=\mu\otimes P^{*}(D\cap B^{2})=\mu({\bf 1}_{B}\cdot s^{*}),
\]
and so taking $B_{+}=\{x\in\mathsf{E}:s(x)>s^{*}(x)\}$ and $B_{-}=\{x\in\mathsf{E}:s(x)<s^{*}(x)\}$
we deduce
\[
\mu\left(\left(s-s^{*}\right)^{+}\right)=0=\mu\left(\left(s-s^{*}\right)^{-}\right),
\]
and hence $s=s^{*}$ $\mu$-almost everywhere.
\end{proof}
\begin{lem}
\label{lem:stick-dirichlet-lower-bound}Assume $P(x,\{x\})\geq\varepsilon$
for $\mu$-almost all $x$. Then $\mathcal{E}(P^{*}P,f)\geq2\varepsilon\mathcal{E}(P,f)$.
\end{lem}

\begin{proof}
We have $P^{*}(x,\{x\})\geq\varepsilon$ for $\mu$-almost all $x$
by Lemma~\ref{lem:stick-probs-equal}. Hence,
\begin{align*}
\mathcal{E}(P^{*}P,f) & =\frac{1}{2}\int\mu({\rm d}x)P^{*}P(x,{\rm d}y)\left\{ f(x)-f(y)\right\} ^{2}\\
 & \geq\frac{1}{2}\int\mu({\rm d}x)\left\{ \varepsilon P(x,{\rm d}y)+\varepsilon P^{*}(x,{\rm d}y)\right\} \left\{ f(x)-f(y)\right\} ^{2}\\
 & =2\varepsilon\mathcal{E}(P,f).
\end{align*}
\end{proof}
The following is a useful implication of $\left\Vert \cdot\right\Vert _{{\rm osc}}^{2}$-convergence,
that we will rely on below and also in Section~\ref{subsec:Holding-probabilities,-WPIs}.
\begin{lem}
\label{lem:osc-conv-RUPI}Assume $T$ is $\left\Vert \cdot\right\Vert _{{\rm osc}}^{2}$-convergent.
Then for any $\epsilon>0$, there exists $n_{0}\in\mathbb{N}$ such
that for any $N\geq n_{0}$
\[
\inf\left\{ \left\langle \mathbf{1}_{A},T^{N}\mathbf{1}_{B}\right\rangle :\mu(A)\wedge\mu(B)\geq\epsilon\right\} >0.
\]
In particular, for all $k\in\mathbb{N}$, $T^{k}$ is RUPI and $T^{k}$
satisfies an $\left\Vert \cdot\right\Vert _{{\rm osc}}^{2}$-WPI.
\end{lem}

\begin{proof}
Let $\epsilon>0$ be arbitrary. Since $T$ is $\left\Vert \cdot\right\Vert _{{\rm osc}}^{2}$-convergent,
we may take $n_{0}\in\mathbb{N}$ large enough such that $\|T^{N}f\|_{2}\le\|f\|_{{\rm osc}}\epsilon^{2}/2$
for all $N\geq n_{0}$. Let $A,B\in\mathscr{E}$ be such that $\mu(A)\wedge\mu(B)\ge\epsilon$.
For any $N\geq n_{0}$ we have
\begin{align*}
\langle\mathbf{1}_{A},T^{N}\mathbf{1}_{B}\rangle & =\langle\mathbf{1}_{A},(T^{N}-\mu)\mathbf{1}_{B}\rangle+\langle\mathbf{1}_{A},\mu\mathbf{1}_{B}\rangle\\
 & =\langle\mathbf{1}_{A},(T^{N}-\mu)\mathbf{1}_{B}\rangle+\mu(A)\mu(B).
\end{align*}
Let $f={\bf 1}_{B}-\mu(B)$ and we have by Cauchy--Schwarz,
\[
|\langle\mathbf{1}_{A},(T^{N}-\mu)\mathbf{1}_{B}\rangle|=|\langle\mathbf{1}_{A},T^{N}f\rangle|\leq\mu(A)^{1/2}\|f\|_{{\rm osc}}\epsilon^{2}/2\leq\epsilon^{2}/2,
\]
and therefore 
\[
\langle\mathbf{1}_{A},T^{N}\mathbf{1}_{B}\rangle\ge-\epsilon^{2}/2+\epsilon^{2}=\epsilon^{2}/2>0,
\]
from which we can conclude. Now let $k\in\{1,2,\ldots\}$ be arbitrary.
Since we may choose $N$ to be a multiple of $k$ it follows from
Lemma~\ref{lem:RUPI_equiv} that $T^{k}$ is RUPI. Hence, by Proposition~\ref{prop:RUPI-to-WPI}
$T^{k}$ satisfies an $\left\Vert \cdot\right\Vert _{{\rm osc}}^{2}$-WPI.
\end{proof}
\begin{prop}
\label{prop:wpi-to-rupi}Let $T$ be a $\mu$-invariant Markov kernel
satisfying a $(\|\cdot\|_{\osc}^{2},\alpha)-$WPI for some $\alpha:(0,\infty)\to[0,\infty)$.
Then $T$ is RUPI.
\end{prop}

\begin{proof}
Consider the Markov operator $\tilde{T}:=\frac{1}{2}(\Id+T)$, which
satisfies $\tilde{T}(x,\{x\})\geq1/2$ by construction. Note that
\begin{align*}
\mathcal{E}\left(\tilde{T},f\right) & =\langle(\Id-(\Id+T)/2)f,f\rangle\\
 & =\frac{1}{2}\langle(\Id-T)f,f\rangle\\
 & =\frac{1}{2}\mathcal{E}(T,f).
\end{align*}
Therefore, since $T$ satisfies a $(\|\cdot\|_{\osc}^{2},\alpha)-$WPI,
we have that $\tilde{T}$ satisfies a $(\|\cdot\|_{\osc}^{2},2\alpha)-$WPI:
\begin{align*}
\|f\|_{2}^{2} & \le\alpha(r)\mathcal{E}(T,f)+r\|f\|_{\osc}^{2}\\
 & =2\alpha(r)\mathcal{E}\left(\tilde{T},f\right)+r\|f\|_{\osc}^{2}.
\end{align*}
Since $\mathrm{ess_{\mu}\inf}_{x}\tilde{T}(x,\{x\})\geq1/2$, by Lemma~\ref{lem:stick-dirichlet-lower-bound}
we have the inequality $\mathcal{E}(\tilde{T},f)\le\mathcal{E}(\tilde{T}^{*}\tilde{T},f)$,
so we deduce a $(\|\cdot\|_{\osc}^{2},2\alpha)-$WPI for $\tilde{T}^{*}\tilde{T}$.
Hence, $\tilde{T}$ is $\left\Vert \cdot\right\Vert _{{\rm osc}}^{2}$-convergent
by Theorem~\ref{thm:WPI_F_bd}.

We will now verify the condition for RUPI in Lemma~\ref{lem:RUPI_equiv}.
Let $\epsilon\in(0,1)$ be arbitrary. Since $\tilde{T}$ is $\left\Vert \cdot\right\Vert _{{\rm osc}}^{2}$-convergent,
Lemma~\ref{lem:osc-conv-RUPI} implies that there exists $N\in\mathbb{N}$
such that 
\[
\delta=\inf\left\{ \left\langle \mathbf{1}_{A},\tilde{T}^{N}\mathbf{1}_{B}\right\rangle :\mu(A)\wedge\mu(B)\geq\epsilon\right\} >0.
\]
Now, 
\[
\tilde{T}^{N}=\left(\frac{\Id+T}{2}\right)^{N}=\frac{1}{2^{N}}\sum_{k=0}^{N}a_{k}T^{k},
\]
for binomial coefficients $\{a_{i}\}$. Since $\sum_{i=0}^{N}a_{i}=2^{N}$
and ${\bf 1}_{A}$, ${\bf 1}_{B}$ are non-negative, we have 
\[
\left\langle \mathbf{1}_{A},\sum_{k=0}^{N}T^{k}\mathbf{1}_{B}\right\rangle \geq\left\langle \mathbf{1}_{A},\left(\frac{\Id+T}{2}\right)^{N}\mathbf{1}_{B}\right\rangle ,\qquad A,B\in\mathscr{E},
\]
and this implies that 
\[
\inf\left\{ \left\langle \mathbf{1}_{A},\sum_{k=0}^{N}T^{k}\mathbf{1}_{B}\right\rangle :\mu(A)\wedge\mu(B)\geq\epsilon\right\} \ge\delta>0,
\]
so $T$ is RUPI.
\end{proof}

\subsubsection{\label{subsec:Holding-probabilities,-WPIs}Holding probabilities,
WPIs and $\left\Vert \cdot\right\Vert _{{\rm osc}}^{2}$-convergence}
\begin{defn}
For a $\mu$-invariant Markov kernel $T$, and $\epsilon\in(0,1)$
we denote by $T_{\epsilon}$ the $\mu$-invariant kernel $T_{\epsilon}=\epsilon{\rm Id}+(1-\epsilon)T$.
\end{defn}

We show in this section that there are close connections between existence
of an $\left\Vert \cdot\right\Vert _{{\rm osc}}^{2}$-WPI for a Markov
kernel $P$, and existence of an $\left\Vert \cdot\right\Vert _{{\rm osc}}^{2}$-WPI
for $P_{\epsilon}^{*}P_{\epsilon}$, where $\epsilon$ is any non-trivial\emph{
holding probability}. This is also closely connected to $\left\Vert \cdot\right\Vert _{{\rm osc}}^{2}$-convergence.

Throughout this section, we write $S:=(P+P^{*})/2$ for the additive
reversibilization of $P$.

\begin{prop}
\label{prop:wpi-P-PP-iff}Let $\epsilon\in(0,1)$. Then $P$ satisfies
an $\left\Vert \cdot\right\Vert _{{\rm osc}}^{2}$-WPI if and only
if $P_{\epsilon}^{*}P_{\epsilon}$ satisfies an $\left\Vert \cdot\right\Vert _{{\rm osc}}^{2}$-WPI. 
\end{prop}

\begin{proof}
This follows from Lemmas~\ref{lem:P-to-PP-WPI} and~\ref{lem:PP-to-P-WPI}.
\end{proof}
\begin{rem}
Proposition~\ref{prop:wpi-P-PP-iff}, and some of the results below
could also be phrased in terms of the alternative multiplicative reversibilizations
of $P_{\epsilon}$, i.e. $P_{\epsilon}P_{\epsilon}^{*}$.
\end{rem}

\begin{prop}
\label{prop:P/S-convergence-P/S}The following hold:
\begin{enumerate}
\item If $P$ is $\left\Vert \cdot\right\Vert _{{\rm osc}}^{2}$-convergent,
then $S^{2}$ satisfies an $\left\Vert \cdot\right\Vert _{{\rm osc}}^{2}$-WPI
and $S$ is $\left\Vert \cdot\right\Vert _{{\rm osc}}^{2}$-convergent.
\item Let $\epsilon\in(0,1)$. If $S$ or $P$ are $\left\Vert \cdot\right\Vert _{{\rm osc}}^{2}$-convergent
then $P_{\epsilon}^{*}P_{\epsilon}$ satisfies an $\left\Vert \cdot\right\Vert _{{\rm osc}}^{2}$-WPI
and $P_{\epsilon}$ is $\left\Vert \cdot\right\Vert _{{\rm osc}}^{2}$-convergent.
\end{enumerate}
\end{prop}

\begin{proof}
For the first part, if $P$ is $\left\Vert \cdot\right\Vert _{{\rm osc}}^{2}$-convergent,
then Lemma~\ref{lem:osc-conv-RUPI} implies that $P^{2}$ is RUPI
and satisfies an $\left\Vert \cdot\right\Vert _{{\rm osc}}^{2}$-WPI.
We may then deduce that $S^{2}$ satisfies an $\left\Vert \cdot\right\Vert _{{\rm osc}}^{2}$-WPI
because for any $f\in{\rm L}_{0}^{2}(\mu)$,
\[
\mathcal{E}(S^{2},f)=\frac{1}{4}\left\{ \mathcal{E}(P^{2},f)+\mathcal{E}((P^{*})^{2},f)+\mathcal{E}(PP^{*},f)+\mathcal{E}(P^{*}P,f)\right\} \geq\frac{1}{4}\mathcal{E}(P^{2},f).
\]
It follows that $S$ is $\left\Vert \cdot\right\Vert _{{\rm osc}}^{2}$
-convergent by Theorem~\ref{thm:WPI_F_bd}. For the second part,
if $S$ or $P$ are $\left\Vert \cdot\right\Vert _{{\rm osc}}^{2}$
-convergent then Lemma~\ref{lem:osc-conv-RUPI} implies that $S$,
or equivalently $P$, satisfies an $\left\Vert \cdot\right\Vert _{{\rm osc}}^{2}$-WPI.
Hence, by Proposition~\ref{prop:wpi-P-PP-iff}, $P_{\epsilon}^{*}P_{\epsilon}$
satisfies an $\left\Vert \cdot\right\Vert _{{\rm osc}}^{2}$-WPI,
from which we can deduce $\left\Vert \cdot\right\Vert _{{\rm osc}}^{2}$-convergence
by Theorem~\ref{thm:WPI_F_bd}.
\end{proof}
\begin{rem}
The appearance of $\epsilon\in(0,1)$ in the implication $S$ is $\left\Vert \cdot\right\Vert _{{\rm osc}}^{2}$-convergent
$\Rightarrow$ $P_{\epsilon}$ is $\left\Vert \cdot\right\Vert _{{\rm osc}}^{2}$-convergent
cannot be removed, since it is possible that $S$ is $\left\Vert \cdot\right\Vert _{{\rm osc}}^{2}$-convergent
but $P$ is not; see Example~\ref{exa:circle-walks}. On the other
hand, $S$ being $\left\Vert \cdot\right\Vert _{{\rm osc}}^{2}$-convergent
is a necessary condition for $P$ to be $\left\Vert \cdot\right\Vert _{{\rm osc}}^{2}$-convergent.
The appearance of $\epsilon\in(0,1)$ in the implication $P$ is $\left\Vert \cdot\right\Vert _{{\rm osc}}^{2}$-convergent
$\Rightarrow$ $P_{\epsilon}^{*}P_{\epsilon}$ satisfies an $\left\Vert \cdot\right\Vert _{{\rm osc}}^{2}$-WPI
also cannot be removed; see Proposition~\ref{prop:nonrev-geo-reducible}
and note that in that example $S^{2}$ is $\mu$-irreducible and so
$S$ is $\left\Vert \cdot\right\Vert _{{\rm osc}}^{2}$-convergent
by Corollary~\ref{cor:WPI_from_irred}.
\end{rem}

\begin{example}[Walks on the circle]
\label{exa:circle-walks}For $x,y\in\mathsf{E}=\{1,\ldots,m\}$ let
$P(x,y)={\bf 1}_{\{1,\ldots,m-1\}}(x){\bf 1}_{\{x+1\}}(y)+{\bf 1}_{\{m\}}(x){\bf 1}_{\{1\}}(y)$
so that $P^{*}(x,y)={\bf 1}_{\{2,\ldots,m\}}(x){\bf 1}_{\{x-1\}}(y)+{\bf 1}_{\{1\}}(x){\bf 1}_{\{m\}}(y)$.
Then the Markov chain associated with $P$ is deterministic and one
can deduce that $P$ is not $\left\Vert \cdot\right\Vert _{{\rm osc}}^{2}$-convergent.
On the other hand, $S=(P+P^{*})/2$ encodes a random walk on $\{1,\ldots,m\}$
and is $\left\Vert \cdot\right\Vert _{{\rm osc}}^{2}$-convergent.
\end{example}

In practice, the following result may be useful.
\begin{prop}
\label{prop:hold-equiv-wpi-conv-pp}Assume $P$ is $\mu$-invariant
and satisfies ${\rm ess}_{\mu}\inf_{x}P(x,\{x\})\in(0,1)$. Then the
following are equivalent.
\begin{enumerate}
\item $P$ satisfies an $\left\Vert \cdot\right\Vert _{{\rm osc}}^{2}$-WPI;
\item $P^{*}P$ satisfies an $\left\Vert \cdot\right\Vert _{{\rm osc}}^{2}$-WPI;
\item $P$ is $\left\Vert \cdot\right\Vert _{{\rm osc}}^{2}$-convergent;
\item $PP^{*}$ satisfies an $\left\Vert \cdot\right\Vert _{{\rm osc}}^{2}$-WPI;
\item $P^{*}$ is $\left\Vert \cdot\right\Vert _{{\rm osc}}^{2}$-convergent.
\end{enumerate}
\end{prop}

\begin{proof}
(b. $\Rightarrow$ c.) follows from Theorem~\ref{thm:WPI_F_bd},
and (c. $\Rightarrow$ a.) follows from Lemma~\ref{lem:osc-conv-RUPI}.
We now show (a. $\Rightarrow$ b.). Let $\varepsilon={\rm ess}_{\mu}\inf_{x}P(x,\{x\})\in(0,1)$.
Then $T:=(P-\varepsilon{\rm Id})/(1-\varepsilon)$ is also a $\mu$-invariant
Markov kernel and also satisfies an $\left\Vert \cdot\right\Vert _{{\rm osc}}^{2}$-WPI
since $\mathcal{E}(T,f)=(1-\varepsilon)^{-1}\mathcal{E}(P,f)$. Since
$P=T_{\varepsilon}$, we deduce by Proposition~\ref{prop:wpi-P-PP-iff}
that $P^{*}P=T_{\varepsilon}^{*}T_{\varepsilon}$ satisfies an $\left\Vert \cdot\right\Vert _{{\rm osc}}^{2}$-WPI.

We now show that the cycle (a. $\Rightarrow$ d. $\Rightarrow$ e.
$\Rightarrow$ a.) can also be deduced. Observe that ${\rm ess}_{\mu}\inf_{x}P^{*}(x,\{x\})={\rm ess}_{\mu}\inf_{x}P(x,\{x\})$
by Lemma~\ref{lem:stick-probs-equal}, and $P$ satisfying an $\left\Vert \cdot\right\Vert _{{\rm osc}}^{2}$-WPI
is equivalent to $P^{*}$ satisfying an $\left\Vert \cdot\right\Vert _{{\rm osc}}^{2}$-WPI,
since $\mathcal{E}(P,f)=\mathcal{E}(P^{*},f)$. Because $(P^{*})^{*}P^{*}=PP^{*}$,
we have that (a. $\Rightarrow$ d.) is equivalent to (a. $\Rightarrow$
b.) and (d. $\Rightarrow$ e.) is equivalent to (b. $\Rightarrow$
c.) and (e. $\Rightarrow$ a.) is equivalent to (c. $\Rightarrow$
a.). 
\end{proof}
\begin{lem}
\label{lem:P-to-PP-WPI}Let $P$ be $\mu$-invariant and assume $P$,
or equivalently $S$, satisfies a $(\Phi,\alpha)$-WPI. For $\epsilon\in(0,1)$,
$P_{\epsilon}^{*}P_{\epsilon}$ satisfies a $(\Phi,\frac{1}{2\epsilon(1-\epsilon)}\alpha)$-WPI.
\end{lem}

\begin{proof}
It is straightforward to verify that $P_{\epsilon}^{*}=\epsilon{\rm Id}+(1-\epsilon)P^{*}$,
and therefore
\[
P_{\epsilon}^{*}P_{\epsilon}=\epsilon^{2}{\rm Id}+\epsilon(1-\epsilon)(P^{*}+P)+(1-\epsilon)^{2}P^{*}P.
\]
It follows that
\[
\mathcal{E}(P_{\epsilon}^{*}P_{\epsilon},f)\geq2\epsilon(1-\epsilon)\mathcal{E}\left(\frac{P^{*}+P}{2},f\right)=2\epsilon(1-\epsilon)\mathcal{E}\left(P,f\right),
\]
from which we may conclude.
\end{proof}
\begin{lem}
\label{lem:PP-to-P-WPI}Let $\epsilon\in(0,1)$. If $P_{\epsilon}^{*}P_{\epsilon}$
satisfies an $\left\Vert \cdot\right\Vert _{{\rm osc}}^{2}$-WPI then
$P$ satisfies an $\left\Vert \cdot\right\Vert _{{\rm osc}}^{2}$-WPI.
\end{lem}

\begin{proof}
First, suppose we have a Markov kernel $T$ such that $T^{*}T$ satisfies
an $\left\Vert \cdot\right\Vert _{{\rm osc}}^{2}$-WPI. Then for $S_{T}:=(T+T^{*})/2$
we have that 
\[
S_{T}^{2}=\frac{1}{4}\left\{ T^{2}+(T^{*})^{2}+TT^{*}+T^{*}T\right\} ,
\]
and so $\mathcal{E}(S_{T}^{2},f)\geq\mathcal{E}(T^{*}T,f)/4$. This
implies that $S_{T}^{2}$ also satisfies an $\left\Vert \cdot\right\Vert _{{\rm osc}}^{2}$-WPI. 

Hence, $P_{\epsilon}^{*}P_{\epsilon}$ satisfying an $\left\Vert \cdot\right\Vert _{{\rm osc}}^{2}$-WPI
implies that $(S_{\epsilon})^{2}$ satisfies an $\left\Vert \cdot\right\Vert _{{\rm osc}}^{2}$-WPI,
where $S_{\epsilon}:=(P_{\epsilon}+P_{\epsilon}^{*})/2$. It follows
from Theorem~\ref{thm:WPI_iff_RUPI} that $(S_{\epsilon})^{2}$ is
RUPI, which implies that $S_{\epsilon}$ is RUPI since for any $m\in\mathbb{N}$,
\[
\left\langle {\bf 1}_{A},\sum_{k=0}^{m}(S_{\epsilon})^{2k}{\bf 1}_{B}\right\rangle \leq\left\langle {\bf 1}_{A},\sum_{k=0}^{2m}(S_{\epsilon})^{k}{\bf 1}_{B}\right\rangle .
\]
Since $S_{\epsilon}=(P_{\epsilon}+P_{\epsilon}^{*})/2=\epsilon{\rm Id}+(1-\epsilon)S$
with $S=(P+P^{*})/2$, we may further deduce that $S$ is RUPI since
for any $m\in\mathbb{N}$,
\[
\left\langle {\bf 1}_{A},\sum_{k=0}^{m}(S_{\epsilon})^{k}{\bf 1}_{B}\right\rangle =\left\langle {\bf 1}_{A},\sum_{k=0}^{m}\sum_{j=0}^{k}a_{k,j}S^{j}{\bf 1}_{B}\right\rangle \leq m\left\langle {\bf 1}_{A},\sum_{k=0}^{m}S^{k}{\bf 1}_{B}\right\rangle ,
\]
where we have used the fact that for each $k\in\{0,\ldots,m\}$, $a_{k,0},a_{k,1},\ldots,a_{k,k}\geq0$
and $\sum_{j}a_{k,j}=1$. It follows that $S$ and therefore $P$
satisfy an $\left\Vert \cdot\right\Vert _{{\rm osc}}^{2}$-WPI. 
\end{proof}

\subsubsection{$\mu$-irreducibility implies a WPI\label{subsec:-irreducibility-implies-a}}

To establish that a given kernel $T$ is RUPI, it is sufficient to
show a simple irreducibility condition.
\begin{defn}
We say that a Markov kernel $T$ on $(\E,\mathscr{E})$ is $\nu$-irreducible
for a measure $\nu$ on $(\E,\mathscr{E})$ if for any measurable
set $A\in\mathscr{E}$ with $\nu(A)>0$, we have that
\[
\sum_{n=0}^{\infty}\lambda^{n}T^{n}(x,A)>0,\quad\forall x\in\mathsf{E},
\]
for some (and hence all) $0<\lambda<1$.
\end{defn}

\begin{prop}[{\cite[Corollary 4.5]{Gong2006}}]
\label{prop:irred_RUPI}Suppose that $T$ is $\mu$-irreducible.
Then $T$ is RUPI.
\end{prop}

Thus we immediately obtain by Theorem~\ref{thm:WPI_iff_RUPI} that
$\mu$-irreducibility is a sufficient condition for the existence
of an $\|\cdot\|_{\osc}^{2}$-WPI.

\begin{cor}
\label{cor:WPI_from_irred}Suppose the Markov kernel $T$ is $\mu$-irreducible.
Then $T$ possesses an $\left\Vert \cdot\right\Vert _{{\rm osc}}^{2}$-WPI
by Proposition~\ref{prop:irred_RUPI} and Theorem~\ref{thm:WPI_iff_RUPI}.
Moreover, if ${\rm ess}_{\mu}\inf_{x}T(x,\{x\})>0$ then $T^{*}T$
possesses an $\left\Vert \cdot\right\Vert _{{\rm osc}}^{2}$-WPI and
$T$ is $\left\Vert \cdot\right\Vert _{{\rm osc}}^{2}$-convergent
by Proposition~\ref{prop:hold-equiv-wpi-conv-pp}.
\end{cor}

It is important to note that $P$ possessing a WPI does not necessarily
imply that $\left\Vert P^{n}f\right\Vert _{2}\to0$ for all relevant
functions. Indeed, a reversible, periodic Markov kernel may satisfy
an $\left\Vert \cdot\right\Vert _{{\rm osc}}^{2}$-WPI yet $\left\Vert P^{n}f\right\Vert _{2}$
cannot converge to $0$ for all bounded functions.
\begin{rem}
When $P$ is reversible, it is possible to deduce the existence of
a WPI for $P^{2}=P^{*}P$ from a WPI for $P$, provided that one has
some additional control on the left spectral gap; see \cite[Section 2.2.1]{ALPW2021}.
In turn, the existence of a WPI for $P$ can often be straightforwardly
deduced from Corollary~\ref{cor:WPI_from_irred} by establishing
irreducibility of $P$.
\end{rem}

\begin{rem}
Corollary~\ref{cor:WPI_from_irred} allows us to guarantee the existence
of a $\left\Vert \cdot\right\Vert _{{\rm osc}}^{2}$-WPI for $T=\left(P^{*}\right)^{k}P^{k}$,
for some $k\in\mathbb{N}$, in many situations. If $P$ is reversible
then $P^{*}P=P^{2}$ being $\mu$-irreducible implies that a WPI exists
for $P^{*}P$. Note that if $P^{2}$ is not $\mu-$irreducible, then
neither is $\left(P^{*}\right)^{k}P^{k}=(P^{2})^{k}$ for any $k>1$.
If $P$ is $\mu$-invariant and admits an $\ELL_{0}(\mu)$-spectral
gap then there exists some $k\in\mathbb{N}$ such that $\left\Vert P^{k}f\right\Vert _{2}\leq C\left\Vert f\right\Vert _{2}$
for some $C<1$ and all $f\in\ELL_{0}(\mu)$ and hence $\left(P^{*}\right)^{k}P^{k}$
admits a strong Poincaré inequality. However, if $P$ is nonreversible
then even if $\left\Vert P^{n}f\right\Vert _{2}$ decays geometrically
for bounded functions, it is possible that for all $k\in\mathbb{N}$,
$\left(P^{*}\right)^{k}P^{k}$ is not $\mu$-irreducible and does
not admit a WPI; see Example~\ref{exa:irred-p*kpk} and Proposition~\ref{prop:nonrev-geo-reducible}. 
\end{rem}

The following example demonstrates (in case 1) that for an arbitrary
$k\in\mathbb{N}$, there exists nonreversible $P$ such that $\left(P^{*}\right)^{k}P^{k}$
is $\mu$-irreducible while $\left(P^{*}\right)^{i}P^{i}$ is not
$\mu$-irreducible for any positive integer $i<k$. It also demonstrates
(case 2) that $(P^{*})^{k}P^{k}$ cannot be $\mu$-irreducible for
any $k\in\mathbb{N}$ even though $P$ is $\mu$-irreducible, and
in this case it is not clear that one can define an appropriate WPI
that provides a vanishing upper bound on $\left\Vert P^{n}f\right\Vert _{2}^{2}$.

In fact, similar examples have been considered by \cite{haggstrom2005central}
and \cite[Section~6]{stadje2011three}, who are essentially interested
in geometrically ergodic Markov chains for which a CLT fails to hold
for an $\ELL$ function, or which do not admit an $\ELL_{0}$ spectral
gap. We construct such an example in Proposition~\ref{prop:nonrev-geo-reducible}.
Our consideration of the following family of examples is very natural;
because there can be arbitrarily long periods of deterministic behaviour,
lack of $\mu$-irreducibility is straightforward to deduce.
\begin{example}
\label{exa:irred-p*kpk}Let $\mathsf{E}=\{1,2,\ldots\}^{2}$, and
$\nu$ a probability mass function on $\{1,2,\ldots\}$ such that
$\nu(1)\in(0,1)$ and $\nu$ has a finite mean. Define
\[
P(i,j;i',j')=\begin{cases}
1 & j<i,i'=i,j'=j+1,\\
\nu(i') & j=i,j'=1,\\
1 & \nu(i)\mathbf{1}\{j\leq i\}=0,(i',j')=(1,1),\\
0 & \text{otherwise}.
\end{cases}
\]
The intuition is that the Markov chain moves to the right along ``level''
$i$ deterministically until it reaches the point $(i,i)$, at which
point it jumps to the start of another level $(K,1)$ where $K\sim\nu$.
The third statement is concerned with initialization of the chain
outside the support of the invariant distribution $\mu$, which one
can verify directly is given by
\[
\mu(i,j)=\frac{\nu(i)\mathbf{1}\{j\leq i\}}{\sum_{k=1}^{\infty}\nu(k)k}.
\]
$P$ is $\mu$-irreducible with an accessible, aperiodic atom $(1,1)$
and its Markov chain converges to $\mu$ in total variation from any
starting point (by, e.g., \cite[Theorem~7.6.4]{douc2018markov}).

By viewing $P^{*}$ as the time-reversal of $P$, and satisfying $\mu(i,j)P^{*}(i,j;i',j')=\mu(i',j')P(i',j';i,j)$,
we may define
\[
P^{*}(i,j;i',j')=\begin{cases}
1 & j>1,i'=i,j'=j-1,\\
\nu(i') & j=1,j'=i',\\
1 & \nu(i)\mathbf{1}\{j\leq i\}=0,(i',j')=(1,1),\\
0 & \text{otherwise}.
\end{cases}
\]

Case 1:\textbf{ }Assume that for some $i_{0}\in\{2,3,\ldots\}$, $\nu(i)>0$
for all $i\leq i_{0}$ and $\nu(i)=0$ for $i>i_{0}$. This means
there is a maximum level length of $i_{0}$. We see that if $k<i_{0}$
then
\[
(P^{*})^{k}P^{k}(i_{0},i_{0};i_{0},i_{0})=1,
\]
since $(P^{*})^{k}(i_{0},i_{0};i_{0}-k,i_{0}-k)=1$ and $P^{k}(i_{0}-k,i_{0}-k;i_{0},i_{0})=1$.
Hence $(P^{*})^{k}P^{k}$ is reducible for any $k<i_{0}-1$. On the
other hand, for $k\geq i_{0}$, we may deduce that $(P^{*})^{k}P^{k}$
is $\mu$-irreducible. In particular, since $P^{*}(1,1;1,1)=P(1,1;1,1)=\nu(1)\in(0,1)$,
we see that $(P^{*})^{k}(i,j;1,1)>0$ for all $(i,j)\in\mathsf{E}$,
from which one may deduce that $(P^{*})^{k}P^{k}(i,j;i',j')>0$ for
all $i,j,i',j'\in\mathsf{E}$ such that $\mu(i',j')>0$. Note that
since $P(i,j;1,1)=1$ for all $(i,j)$ such that $\mu(i,j)=0$, this
is essentially a finite state space Markov chain after 1 step, and
hence convergence is geometric. 

Case 2: Assume that $\nu(i)>0$ for all $i\in\{1,2,\ldots\}$. For
any $k\in\mathbb{N}$ we may consider level $i>k$ and we see that
$(P^{*})^{k}P^{k}(i,i;i,i)=1$ so $(P^{*})^{k}P^{k}$ is reducible.
Hence, there does not exist $k\in\mathbb{N}$ such that $(P^{*})^{k}P^{k}$
is $\mu$-irreducible.
\end{example}

Our final result in this section shows that $P$ being $\Phi$-convergent
does not imply that there exists $k\in\mathbb{N}$ such that $(P^{*})^{k}P^{k}$
admits a $\Phi$-WPI when $P$ is nonreversible, even in the case
where $\gamma$ decays geometrically. We note that by Proposition~\ref{prop:cattiaux-et-al-gamma-p},
geometric convergence can be extended to all functions in ${\rm L}_{0}^{p}(\mu)$
for any $p>2$.
\begin{prop}
\label{prop:nonrev-geo-reducible}For the chain in Example~\ref{exa:irred-p*kpk},
let $\nu(i)=(1-a)a^{i-1}$ for some $a\in(0,1)$. Then
\begin{enumerate}
\item $P$ is geometrically ergodic and 
\[
\left\Vert P^{n}f\right\Vert _{2}^{2}\leq\left\Vert f\right\Vert _{{\rm osc}}^{2}C\rho^{2n},\qquad f\in\ELL_{0}(\mu),
\]
for some $C>0$ and $\rho\in(0,1)$; 
\item $(P^{*})^{k}P^{k}$ does not admit an $\left\Vert \cdot\right\Vert _{{\rm osc}}^{2}$-WPI
for any $k\in\mathbb{N}$.
\end{enumerate}
\end{prop}

\begin{proof}
We will apply \cite[Theorem~1.1]{baxendale2005renewal}. We may consider
the state space to be the $\mu$-full set $\mathsf{E}=\{(i,j)\in\{1,2,\ldots\}^{2}:j\leq i\}$
for simplicity. We now verify the assumptions (A1)-(A3) in \cite{baxendale2005renewal}.
We first define the set $C=\{(i,j)\in E:i=j\}$. We define the probability
measure $\tilde{\nu}(i,j)=\nu(i)\mathbb{\mathbf{1}}\{j=1\}$, and
we have that 
\[
P(x,A)=\tilde{\nu}(A),\qquad x\in C.
\]
For any $\lambda\in(\sqrt{a},1)$ we may define the Lyapunov function
$V(i,j):=\lambda^{j-i}$ and we observe that $V\geq1$ on $\mathsf{E}$,
with $PV(x)\leq\lambda V(x){\bf 1}_{C^{\complement}}(x)+K{\bf 1}_{C}(x)$,
where
\[
K=\tilde{\nu}(V)=\sum_{i\geq1}\nu(i)V(i,1)\propto\sum_{i\geq1}a^{i}\lambda^{1-i}<\infty,
\]
since $\lambda>\sqrt{a}>a$. Finally, we note that $\tilde{\nu}(C)=\nu(1)>0$.
It then follows by the theorem that there exist $M>0$, $\rho\in(0,1)$
such that
\[
\sup_{f:\left|f\right|\leq V,\mu(f)=0}\left|P^{n}f(x)\right|\leq MV(x)\rho^{n}.
\]
Hence, we may deduce that for $f$ such that $\left|f\right|\leq V$
\[
\left\Vert P^{n}f\right\Vert _{2}^{2}\leq M^{2}\mu(V^{2})\rho^{2n},
\]
and since $\lambda>\sqrt{a}$, we have
\begin{align*}
\mu(V^{2}) & =\sum_{i\geq1}\sum_{j=1}^{i}\mu(i,j)V(i,j)^{2}\\
 & \propto\sum_{i\geq1}a^{i}\sum_{j=1}^{i}\lambda^{2(j-i)}\\
 & \leq\frac{\lambda^{2}}{1-\lambda^{2}}\sum_{i\geq1}\frac{a^{i}}{\lambda^{2i}}\\
 & <\infty.
\end{align*}
The bound on $\left\Vert P^{n}f\right\Vert _{2}^{2}$ for bounded
functions then follows since $V\geq1$.

For the second part, let $k\in\mathbb{N}$ be arbitrary. Let $A_{k}=\{(i,i):i>k\}$
and $f_{k}={\bf 1}_{A_{k}}-\mu(A_{k})$, which satisfies $\mu(f_{k})=0$.
Then $\left\Vert P^{k}f_{k}\right\Vert _{2}^{2}=\left\langle \left(P^{*}\right)^{k}P^{k}f_{k},f_{k}\right\rangle =\left\langle f_{k},f_{k}\right\rangle =\left\Vert f_{k}\right\Vert _{2}^{2}$,
so $\mathcal{E}((P^{*})^{k}P^{k},f_{k})=0$. Since $\left\Vert f_{k}\right\Vert _{2}>0$,
$(P^{*})^{k}P^{k}$ cannot satisfy a $\left\Vert \cdot\right\Vert _{{\rm osc}}^{2}$-WPI.
\end{proof}

\subsection{Lyapunov meets Poincaré\label{subsec:establish-WPI-Lyapunov-meets-Poincar=0000E9}}

A difficulty with functional-analytic approaches to the study of Markov
chains is the challenge posed by unbounded supports for $\mu$; in
particular, handling the tails of $\mu$. A general strategy consists
of splitting the state space $\mathsf{E}$ into a distinguished set
$C$ on which a form of strong Poincaré inequality is established,
while the behaviour of the chain on $C^{\complement}$ is handled
with a Lyapunov drift function. Such ideas have been primarily explored
for certain classes of continuous-time Markov processes, with \cite{bakry2008simple}
establishing quantitative strong Poincaré inequalities for the overdamped
Langevin process; these results were later extended to heavy-tailed
target distributions in \cite{cattiaux2010functional} to establish
WPIs. It is only recently that some of these ideas were extended to
discrete-time Markov chains in \cite{taghvaei2021lyapunov} where
a strategy to establish strong Poincaré inequalities is proposed;
we note also the recent contribution of \cite{canizo2021harris}.
In this subsection we first briefly review the key results of \cite{taghvaei2021lyapunov},
show how they can be improved in the spirit of \cite{bakry2008simple}
by using local Poincaré inequalities (Subsection~\ref{subsec:Lyapunov-The-geometric-scenario}).
In Subsection~\ref{subsec:Lyapunov-The-subgeometric-scenario}, we
show how these results can be extended to \textit{subgeometric} drift
conditions in order to establish WPIs. 

We first define precisely the restriction of the $\mu$-invariant
kernel $P$ to the set $C$ and the notion of \textit{local} Poincaré
inequality. 
\begin{defn}
\label{def:restricted_C}For some $C\in\mathscr{E}_{+}$, we define
the \textit{restriction of $\mu$ to $C$} to be the probability measure
$\mu_{C}$ supported on $C$ given by
\[
\mu_{C}(A):=\frac{\mu(A\cap C)}{\mu(C)},\qquad A\in\mathscr{E},
\]
and \textit{the restriction of $P$ to $C$} is defined to be the
kernel $P_{C}$ defined as: for each $x\in\mathsf{E}$,
\[
P_{C}f(x):=P(f\cdot\mathbf{1}_{C})(x)+f(x)P(x,C^{\complement})\,.
\]
We will say that a \textit{restricted Poincaré inequality holds for}
$P$ on $C$ if a strong Poincaré inequality holds for $P_{C}$: for
some $\mathrm{C_{r}}>0$ and all $f\in\ELL_{0}(\mu_{C})$,
\begin{align}
\|f\|_{\mu_{C},2}^{2} & =\mu(C)^{-1}\int f^{2}(x)\mathbf{1}_{C}(x)\,\mu(\dif x)\nonumber \\
 & \le\mathrm{C}_{\mathrm{r}}\calE(P_{C},f),\label{eq:restricted_PI}
\end{align}
where $\calE(P_{C},f):=\int\mu_{C}({\rm d}x)P_{C}(x,{\rm d}y)\left[f(y)-f(x)\right]^{2}$
for any $f\in\ELL(\mu_{C})$.

This can equivalently be expressed as requiring: for any $f\in\ELL(\mu)$,
\[
\mu_{C}\big(f_{m}^{2}\big)\leq\mathrm{C}_{\mathrm{r}}\calE(P_{C},f),
\]
with $m:=\mu(f\cdot\mathbf{1}_{C})/\mu(C)$ and $f_{m}:=f-m$.

Finally, we will say that a \textit{local Poincaré inequality holds
for $P$ on $C$} if for some $\mathrm{C}_{l}>0$, for all $f\in\ELL(\mu)$,
there is some $m\in\R$ such that setting $f_{m}:=f-m$, we have
\begin{equation}
\|f_{m}\mathbb{\mathbf{1}}_{C}\|_{2}^{2}\leq\mathrm{C}_{l}\mathcal{E}(P,f).\label{eq:local-SPI}
\end{equation}
\end{defn}

We note that when $P$ is reversible, the restriction $P_{C}$ is
simply a Metropolis--Hastings Markov kernel targeting $\mu_{C}$
and using proposal distribution $P$. The $\mu_{C}$-reversibility
of such restrictions is well-known; a proof is provided for completeness.
For a nonreversible, $\mu$-invariant $P$ it is also well-known that
$P_{C}$ is not necessarily $\mu_{C}$-invariant. 

\begin{lem}
\label{lem:local_SPI_to_SPI} Let $P$ be a $\mu-$reversible Markov
kernel. Then the kernel $P_{C}$ is $\mu_{C}-$reversible, and furthermore
if a restricted Poincaré inequality (\ref{eq:restricted_PI}) holds
for $P_{C}$, then the the following local Poincaré inequality for
$P$ on $C$ holds: for any $f\in\ELL(\mu)$,
\[
\|f_{m}\mathbb{\mathbf{1}}_{C}\|_{2}^{2}\leq\mathrm{C}_{\mathrm{r}}\mathcal{E}(P,f),
\]
with $m:=\mu(f\mathbf{1}_{C})/\mu(C)$ and $f_{m}:=f-m$.
\end{lem}

\begin{proof}
We first check $\mu_{C}$-reversibility of $P_{C}$. For $f,g\in\ELL(\mu)$,
let
\[
v=\int f(x)g(x)\mu_{C}({\rm d}x)P(x,C^{\complement}),
\]
and by the $\mu-$reversibility of $P$, we have
\begin{align*}
\int f(x)g(y)\mu_{C}({\rm d}x)P_{C}(x,{\rm d}y) & =\frac{1}{\mu(C)}\int f(x)\mathbf{1}_{C}(x)g(y)\mathbf{1}_{C}(y)\mu({\rm d}x)P(x,{\rm d}y)+v\\
 & =\frac{1}{\mu(C)}\int f(y)\mathbf{1}_{C}(y)g(x)\mathbf{1}_{C}(x)\mu({\rm d}x)P(x,{\rm d}y)+v\\
 & =\int f(y)g(x)\mu_{C}({\rm d}x)P_{C}(x,{\rm d}y).
\end{align*}
Now from the restricted Poincaré inequality, we have
\begin{align*}
\int f_{m}^{2}(x)\mathbf{1}_{C}(x){\rm d}\mu({\rm d}x) & \leq\mu(C)\cdot\frac{\mathrm{C}_{\mathrm{r}}}{2}\int\mu_{C}({\rm d}x)P_{C}(x,{\rm d}y)\left[f(y)-f(x)\right]^{2}\\
 & =\frac{\mathrm{C}_{\mathrm{r}}}{2}\int\mu(\dif x)P\left(x,\dif y\right)\,\mathbf{1}_{C}(x)\mathbf{1}_{C}\left(y\right)\,\left[f(x)-f(y)\right]^{2}\\
 & \le\frac{\mathrm{C}_{\mathrm{r}}}{2}\int\mu(\dif x)P\left(x,\dif y\right)\,\left[f(x)-f(y)\right]^{2}\\
 & =\mathrm{C}_{{\rm r}}\mathcal{E}(P,f)\,.
\end{align*}
\end{proof}
In Section~\ref{subsec:Local-Poincar=0000E9-inequalities-iso}, we
show how one can deduce local Poincaré inequalities when $\mu$ has
a strongly log-concave density and a coupling argument 

\subsubsection{The geometric scenario\label{subsec:Lyapunov-The-geometric-scenario}}

The following is the adaption of \cite{bakry2008simple} to the discrete
time scenario by \cite{taghvaei2021lyapunov} where we here replace
the minorization condition with a local Poincaré inequality. 
\begin{thm}[\cite{taghvaei2021lyapunov}]
\label{thm:bakry-taghvaei} Assume the existence of $C\subset\mathsf{X}$,
a Lyapunov function $V\colon\mathsf{E}\rightarrow[1,\infty)$, and
constants $K,b,>0$ and $\lambda\in(0,1]$ such that 
\begin{equation}
PV\leqslant\left(1-\lambda\right)V+b\mathbf{1}_{C},\label{eq:drift-condition}
\end{equation}
and that we have the following local Poincaré inequality for $P$
on $C$: for any $f\in\ELL_{0}(\mu)$, there exists some $m>0$ such
that for $f_{m}:=f-m$,
\begin{equation}
\left\langle f_{m}^{2},\mathbf{1}_{C}\right\rangle \leqslant K\left\langle f,\left({\rm Id}-P\right)f\right\rangle .\label{eq:semi-local-poincare}
\end{equation}
Then we have the following (strong) Poincaré inequality for $P$:
for any $f\in\ELL(\mu)$,
\[
\frac{\lambda}{1+Kb}\|f-\mu(f)\|_{2}^{2}\leqslant\left\langle f,\left({\rm Id}-P\right)f\right\rangle .
\]
\end{thm}

\begin{proof}
 From $PV\leq(1-\lambda)V+b\mathbb{\mathbf{1}}_{C}$ we obtain $\lambda V\leq({\rm Id}-P)V+b\mathbb{\mathbf{1}}_{C}$,
and so for $f\in\ELL(\mu)$, 
\[
f_{m}^{2}\leq\frac{1}{\lambda}\frac{f_{m}^{2}({\rm Id}-P)V}{V}+\frac{b}{\lambda}\frac{f_{m}^{2}}{V}\mathbb{\mathbf{1}}_{C}.
\]
We observe also that $PV/V\leq1+b<\infty$. Hence, using the variational
characterization of the mean, key Lemma~\ref{lem:taghvaei-important}
(noting that $\sup_{x\in\mathsf{E}}PV/V(x)<\infty$) and (\ref{eq:semi-local-poincare}),
\begin{align*}
\left\Vert f\right\Vert _{2}^{2} & \leq\left\Vert f_{m}\right\Vert _{2}^{2}\\
 & \leq\frac{1}{\lambda}\left\langle f_{m}^{2},1-PV/V\right\rangle +\frac{b}{\lambda}\left\langle f_{m}^{2},{\bf 1}_{C}\right\rangle \\
 & \leq\frac{1+Kb}{\lambda}\left\langle f,({\rm Id}-P)f\right\rangle ,
\end{align*}
and we conclude.
\end{proof}
The proof relies on the important lemma.
\begin{lem}[\cite{taghvaei2021lyapunov}]
\label{lem:taghvaei-important}Let $P$ be $\mu-$reversible, $V\colon\mathsf{E}\rightarrow[1,\infty)$
such that $\left\Vert PV/V\right\Vert _{\infty}<\infty$. Then for
any $f\in\ELL(\mu)$, $m\in\mathbb{R}$ we have
\[
\left\langle (f-m)^{2},1-PV/V\right\rangle \leq\left\langle f,\left({\rm Id}-P\right)f\right\rangle \,.
\]
\end{lem}

\begin{proof}
We have for any $g\in\ELL(\mu)$,
\begin{align*}
0\leqslant\frac{1}{2}\cdot\int\mu\left(\mathrm{d}x\right)\cdot & P\left(x,\mathrm{d}y\right)\cdot V\left(x\right)\cdot V\left(y\right)\cdot\left(\frac{g\left(y\right)}{V\left(y\right)}-\frac{g\left(x\right)}{V\left(x\right)}\right)^{2}\\
 & =\left\langle g^{2},PV/V\right\rangle -\left\langle g,Pg\right\rangle \\
 & =\left\langle g,g\right\rangle -\left\langle g^{2},1-PV/V\right\rangle -\left\langle g,g\right\rangle +\left\langle g,({\rm Id}-P)g\right\rangle \\
 & =\left\langle g,({\rm Id}-P)g\right\rangle -\left\langle g^{2},1-PV/V\right\rangle .
\end{align*}
Further we notice that for any $f\in\ELL(\mu)$ and $m\in\mathbb{R},$
\begin{align*}
\left\langle f-m,\left({\rm Id}-P\right)(f-m)\right\rangle  & =\left\langle f-m,\left({\rm Id}-P\right)f\right\rangle \\
 & =\left\langle f,\left({\rm Id}-P\right)f\right\rangle -m\mu\big(({\rm Id}-P)f\big)\\
 & =\left\langle f,\left({\rm Id}-P\right)f\right\rangle \,,
\end{align*}
and we conclude.
\end{proof}

\subsubsection{The subgeometric scenario\label{subsec:Lyapunov-The-subgeometric-scenario} }

The following is a useful, simple result, which is related to \cite[Theorem~14.3.7]{meyn:tweedie:1993}
and \cite[Proposition~4.3.2]{douc2018markov} but with slightly different
conditions and conclusions.
\begin{lem}
\label{lem:muf-mus}Let $X$ be a Markov chain with Markov operator
$P$ and unique invariant probability measure $\mu$. Suppose $V$,
$f$ and $\mathfrak{s}$ are nonnegative, finite-valued functions
on $\mathsf{E}$ such that
\[
PV\leq V-f+\mathfrak{s}.
\]
Then $\mu(f)\leq\mu(\mathfrak{s})$, whether or not $\mu(f)=\infty$.
\end{lem}

\begin{proof}
We have
\begin{align*}
0 & \leq P^{n}V(x)\\
 & =V(x)+\sum_{i=1}^{n}\mathbb{E}_{x}\left[V(X_{i})-V(X_{i-1})\right]\\
 & =V(x)+\sum_{i=1}^{n}\mathbb{E}_{x}\left[PV(X_{i-1})-V(X_{i-1})\right]\\
 & \leq V(x)+\sum_{i=1}^{n}\mathbb{E}_{x}\left[\mathfrak{s}(X_{i-1})-f(X_{i-1})\right],
\end{align*}
and hence we find
\[
\mathbb{E}_{x}\left[\sum_{k=0}^{n-1}f(X_{k})\right]\leq V(x)+\mathbb{E}_{x}\left[\sum_{k=0}^{n-1}\mathfrak{s}(X_{k})\right].
\]
Since $P$ has a unique invariant probability measure, we may apply
Birkhoff's ergodic theorem; see, e.g., \cite[Theorems~5.2.6 and~5.2.1]{douc2018markov}.
First suppose $g\geq0$ is such that $\mu(g)=\infty$. Then by the
ergodic theorem, for $\mu$-almost all $x$ and any $m\geq0$,
\[
\lim_{n\to\infty}\frac{1}{n}\mathbb{E}_{x}\left[\sum_{k=0}^{n-1}m\wedge g(X_{k})\right]=\mu(m\wedge g),
\]
and taking $m\to\infty$ we obtain $\lim_{n\to\infty}\frac{1}{n}\mathbb{E}_{x}\left[\sum_{k=0}^{n-1}g(X_{k})\right]=\infty=\mu(g)$.
Next, suppose $g\geq0$ with $\mu(g)<\infty$. Then by the ergodic
theorem, for $\mu$-almost all $x$,
\[
\mu(g)=\lim_{n\to\infty}\frac{1}{n}\mathbb{E}_{x}\left[\sum_{k=0}^{n-1}g(X_{k})\right].
\]
Hence, for $\mu$-almost all (and therefore some) $x$, 
\[
\mu(f)=\lim_{n\to\infty}\frac{1}{n}\mathbb{E}_{x}\left[\sum_{k=0}^{n-1}f(X_{k})\right]\leq\lim_{n\to\infty}\frac{1}{n}\left\{ V(x)+\mathbb{E}_{x}\left[\sum_{k=0}^{n-1}\mathfrak{s}(X_{k})\right]\right\} =\mu(\mathfrak{s}),
\]
and so we may conclude that $\mu(f)\leq\mu(\mathfrak{s})$.
\end{proof}

\begin{thm}
\label{thm:WPI-from-drift-condition}Let $P$ be a $\mu-$reversible
Markov kernel, $\mu$ its unique invariant probability measure, such
that:
\begin{enumerate}
\item there exists a set $C\in\mathscr{E}$, a function $V\colon\mathsf{E}\rightarrow[1,\infty)$
and $b>0$ such that
\begin{align*}
PV & \leq V-\phi\circ V+b\mathbb{\mathbf{1}}_{C}\,,
\end{align*}
where $\phi\colon[1,\infty)\rightarrow(0,\infty)$ is a concave, continuous
and increasing function;
\item a local Poincaré inequality holds: there exists $K>0$ such that for
any $f\in\ELL(\mu)$, 
\begin{equation}
\|f_{m}\mathbb{\mathbf{1}}_{C}\|_{2}^{2}\leq K\mathcal{E}(P,f)\,,\label{eq:semi-local-SPI-1}
\end{equation}
with $m=\mu(f\cdot\mathbf{1}_{C})/\mu(C)$ and $f_{m}:=f-m$.
\end{enumerate}
Then for any $f\in\mathrm{L}_{0}^{2}(\mu)$ and $s>0$,
\[
\|f\|_{2}^{2}\leq s\mathcal{E}(P,f)+\beta(s)\|f\|_{{\rm osc}}^{2},
\]
 where 
\begin{align*}
\beta(s) & :=\frac{b\mu(C)}{\phi\circ({\rm Id}/\phi)^{-1}\big(s/(1+Kb)\big)}\,.
\end{align*}

\end{thm}

\begin{proof}
From $PV\leq V-\phi\circ V+b\mathbb{\mathbf{1}}_{C}$ we obtain $\phi\circ V\leq({\rm Id}-P)V+b\mathbb{\mathbf{1}}_{C}$
and for $f\in\ELL(\mu)$ we obtain 
\[
f_{m}^{2}\leq\frac{f_{m}^{2}({\rm Id}-P)V}{\phi\circ V}+b\frac{f_{m}^{2}}{\phi\circ V}\mathbb{\mathbf{1}}_{C}.
\]
Now with $A(s):=\{x\in\mathsf{E}\colon s\,\phi\circ V(x)\geq V(x)\}$
for $s>0$ we have
\begin{align*}
f_{m}^{2}\mathbf{1}_{A(s)} & \leq\frac{f_{m}^{2}({\rm Id}-P)V}{\phi\circ V}\mathbb{\mathbf{1}}_{A(s)}+b\frac{f_{m}^{2}}{\phi\circ V}\mathbb{\mathbf{1}}_{C\cap A(s)}\,,\\
 & \leq s\frac{f_{m}^{2}({\rm Id}-P)V}{V}\mathbb{\mathbf{1}}_{A(s)}+sb\frac{f_{m}^{2}}{V}\mathbb{\mathbf{1}}_{C\cap A(s)}\,.
\end{align*}
Hence for $s>0$, we have 
\begin{align*}
f_{m}^{2} & =f_{m}^{2}{\bf 1}_{A(s)}+f_{m}^{2}{\bf 1}_{A^{\complement}(s)}\\
 & \leq s\frac{f_{m}^{2}({\rm Id}-P)V}{V}+sb\frac{f_{m}^{2}}{V}\mathbb{\mathbf{1}}_{C}+f_{m}^{2}\mathbf{1}_{A^{\complement}(s)}.
\end{align*}
We observe also that $PV/V\leq1+b<\infty$. Consequently, we can take
expectations with respect to $\mu$, yielding
\begin{align*}
\|f_{m}\|_{2}^{2} & \leq s\left\langle f_{m}^{2},1-PV/V\right\rangle +sb\left\Vert f_{m}\mathbf{1}_{C}\right\Vert _{2}^{2}+\left\Vert f_{m}\right\Vert _{\infty}^{2}\mu(A^{\complement}(s))\\
 & \leq s\mathcal{E}(P,f)+sb\left\Vert f_{m}\mathbf{1}_{C}\right\Vert _{2}^{2}+\left\Vert f\right\Vert _{{\rm osc}}^{2}\mu(A^{\complement}(s)),
\end{align*}
where we have used $V\geq1$, Lemma~\ref{lem:taghvaei-important}
and $\left\Vert f_{m}\right\Vert _{\infty}^{2}\leq\left\Vert f\right\Vert _{{\rm osc}}^{2}$
since ${\rm ess}\inf_{x}\left|f(x)\right|\leq m\leq{\rm ess}\sup_{x}\left|f(x)\right|$.
Since $\phi$ is concave, increasing and continuous, the function
$x\mapsto x/\phi(x)$ is increasing and continuous and therefore invertible,
and we can write $A^{\complement}(s)=\{x\in\mathsf{E}\colon\phi\circ({\rm Id}/\phi)^{-1}(s)<\phi\circ V(x)\}$,
and therefore
\begin{align*}
\mu(A^{\complement}(s)) & =\mu\left(\phi\circ V>\phi\circ({\rm Id}/\phi)^{-1}(s)\right)\\
 & \leq\mu\left(\phi\circ V\geq\phi\circ({\rm Id}/\phi)^{-1}(s)\right)\\
 & \leq\frac{\mu(\phi\circ V)}{\phi\circ({\rm Id}/\phi)^{-1}(s)},
\end{align*}
where $\mu(\phi\circ V)\leq b\mu(C)$ by Lemma~\ref{lem:muf-mus}.
Using $f\in\ELL_{0}(\mu)$ and (\ref{eq:semi-local-SPI-1}), 
\[
\|f\|_{2}^{2}\leq\|f_{m}\|_{2}^{2}\leq(1+Kb)s\mathcal{E}(P,f)+\frac{b\mu(C)}{\phi\circ({\rm Id}/\phi)^{-1}(s)}\|f\|_{{\rm osc}}^{2}.
\]
and we conclude.
\end{proof}
\begin{rem}
The assumption that $P$ is reversible can be relaxed to some extent.
If
\[
\frac{1}{2}(P+P^{*})V\leq V-\phi\circ V+b\mathbb{\mathbf{1}}_{C}\,,
\]
then the conclusion also holds for nonreversible $P$. In particular,
this condition allows for the use of Lemma~\ref{lem:taghvaei-important},
which is the only part of the proof utilizing reversibility.
\end{rem}

As pointed out by \cite{taghvaei2021lyapunov}, a standard minorization
condition yields a local PI \ref{eq:local-SPI}.
\begin{lem}[{\cite[Equation (8)]{taghvaei2021lyapunov}}]
\label{lem:minorization-semi-local-SPI}Let $P$ be a $\mu-$invariant
Markov kernel satisfying
\begin{align*}
P(x,A) & \geq\epsilon\nu(A)\mathbf{1}\{x\in C\}.
\end{align*}
Then with $m=\mu(f\mathbf{1}_{C})/\mu(C)$, 
\[
\|f_{m}\mathbb{\mathbf{1}}_{C}\|_{2}^{2}\leq\frac{2}{\epsilon}\mathcal{E}(P,f).
\]
\end{lem}

While this provides a relatively straightforward route to establishing
a local PI, such an approach may not be sufficiently precise when
one is interested in quantitative estimates. In Lemma~\ref{lem:local_SPI_to_SPI}
the minorization condition is replaced with a local PI on $C$, but
$P$ is assumed $\mu-$reversible. This mirrors \cite[Proof of Theorem 1.4]{bakry2008simple}.
\begin{example}
When $\phi(v)=cv^{\alpha}$ for $\alpha\in[0,1)$, we obtain $({\rm Id}/\phi)^{-1}(s)=(cs)^{1/(1-\alpha)}$
and therefore $\phi\circ({\rm Id}/\phi)^{-1}(s)=c(cs)^{\alpha/(1-\alpha)}$.
We conclude that $\beta(s)\propto s^{-\alpha/(1-\alpha)}$, and thereby
obtain $\gamma(n)\propto n^{-\alpha/(1-\alpha)}$. Drift and minorisation
techniques directly lead to a total variation rate of $n^{-\alpha/(1-\alpha)}$,
which we do not recover since \cite[Remark~12]{ALPW2021} gives a
total variation rate of $\gamma^{1/2}(n)$. On the other hand, Proposition~\ref{prop:clt}
implies a CLT for bounded functions if $\alpha/(1-\alpha)>1$, i.e.
$\alpha>1/2$. This improves upon the condition $\alpha\geq2/3$ in
\cite[Theorem~4.2]{jarner2002polynomial} and is close to the condition
$\alpha\geq1/2$ obtained when the existence of an atom is assumed
\cite[Theorem~4.4]{jarner2002polynomial}. We can straightforwardly
obtain rates of convergence and CLTs for functions in ${\rm L}_{0}^{p}(\mu)$,
for $p>2$ using Proposition~\ref{prop:cattiaux-et-al-gamma-p} and
Remark~\ref{rem:clt-lp}, which may be more convenient than considering
functions dominated by a power of the Lyapunov function $V$.
\end{example}

\begin{example}
If $\phi(v)=cv/\log(v)^{\alpha}$ then 
\[
\frac{v}{\phi(v)}=c^{-1}\log(v)^{\alpha}=s\iff v=\exp\big((cs)^{1/\alpha}\big)
\]
 and therefore 
\[
\beta(s)\propto s\exp\big(-(cs)^{1/\alpha}\big)\leq\exp\big(-(c's)^{1/\alpha}\big)
\]
which leads to a rate of convergence
\[
C'\exp\left(-\left\{ C(1+\alpha)n\right\} ^{1/(1+\alpha)}\right)
\]
 which is similar to what is obtained by \cite{douc2018markov}.
\end{example}

\subsubsection{\label{subsec:Local-Poincar=0000E9-inequalities-iso}Local Poincaré
and isoperimetric inequalities}

We use some general results, largely inspired by their recent use
in \cite{Dwivedi-Chen-Wainwright-JMLR:v20:19-306}.
\begin{lem}[{\cite[Theorem 4.2]{cousins2014cubic}, Isoperimetric inequality}]
\label{lem:cousin-vempala}Let $\mu$ be a probability measure on
$\E\subset\R^{d}$, whose density $\mu(x)\propto\exp(-U(x))$ w.r.t.
Lebesgue is $m$-strongly log-concave, i.e.
\[
U\left(x+z\right)-U\left(x\right)-\left\langle \nabla U\left(x\right),z\right\rangle \geqslant\frac{m}{2}\left|z\right|^{2}.
\]
 Then for any (nonempty) $S_{1},S_{2},S_{3}\subset\mathsf{E}$ defining
a partition of $\mathsf{E}$ we have

\[
\mu\left(S_{3}\right)\geqslant\log2\cdot\sqrt{m}\cdot d\left(S_{1},S_{2}\right)\cdot\mu\left(S_{1}\right)\cdot\mu\left(S_{2}\right)\,,
\]
where $d\left(S_{1},S_{2}\right):=\inf\left\{ \left|z-z'\right|\colon\left(z,z'\right)\in S_{1}\times S_{2}\right\} $.

\end{lem}

\begin{rem}
In the original result of \cite[Theorem 4.2]{cousins2014cubic}, the
hypothesis on $\mu$ is formulated in terms of the log-concavity of
the Radon--Nikodym derivative of $\mu$ with respect to an appropriate
Gaussian measure. We have rephrased the result slightly to emphasize
the relationship with the strong convexity of the potential, which
is consistent with the presentation of \cite[Section 5.4]{Dwivedi-Chen-Wainwright-JMLR:v20:19-306}.

\end{rem}

\begin{thm}[\cite{lovasz1999hit,belloni2009computational,Dwivedi-Chen-Wainwright-JMLR:v20:19-306}]
\label{thm:local-conductance-lower-bound}Let $\mu$ be a probability
measure on $\E\subset\R^{d}$, whose density w.r.t. Lebesgue is $m$-strongly
log-concave, and $C\subseteq\mathsf{E}$ be a convex set. Let $P$
be a $\mu-$invariant Markov kernel and assume that there exist $\delta,\epsilon>0$
such that for $z,z'\in C$, $\left|z-z'\right|\leqslant\delta$ implies
\[
\left\Vert P\left(z,\cdot\right)-P\left(z',\cdot\right)\right\Vert _{\mathrm{TV}}<1-\varepsilon\,.
\]
Then for any $A\in\mathscr{E}$,
\[
\mu\otimes P\big(A\times A^{\complement}\big)\geqslant\frac{\varepsilon}{4}\min\left\{ 1,\frac{\log2}{8}\delta\sqrt{m}\right\} \min\left\{ \mu\left(A\cap C\right),\mu\left(A^{\complement}\cap C\right)\right\} \,,
\]
and
\[
\mu\otimes P\big(A\times A^{\complement}\big)\geqslant\frac{1}{\mu(C)}\frac{\varepsilon}{4}\min\left\{ 1,\frac{\log2}{4}\delta\sqrt{m}\right\} \mu\left(A\cap C\right)\mu\left(A^{\complement}\cap C\right).
\]
\end{thm}

\begin{proof}
  Let $\delta,\epsilon>0$ be as above. For $A\in\mathscr{E}$
define the sets
\begin{align*}
S_{1} & :=\left\{ z\in A\cap C\colon P\left(z,A^{\complement}\right)<\varepsilon/2\right\} \\
S_{2} & :=\left\{ z\in A^{\complement}\cap C\colon P\left(z,A\right)<\varepsilon/2\right\} 
\end{align*}
and $S_{3}:=C\cap\big(S_{1}\cup S_{2}\big)^{\complement}$. We consider
two cases. First we establish that when either $\mu\left(S_{1}\right)\leqslant\frac{1}{2}\mu\left(A\cap C\right)$
or $\mu\left(S_{2}\right)\leqslant\frac{1}{2}\mu\left(A^{\complement}\cap C\right)$,
then 
\[
\mu\otimes P\big(A\times A^{\complement}\big)\geqslant\frac{1}{4}\cdot\varepsilon\cdot\min\left\{ \mu\left(A\cap C\right),\mu\big(A^{\complement}\cap C\big)\right\} \,.
\]
If $\mu\left(S_{1}\right)\leqslant\frac{1}{2}\mu\left(A\cap C\right)$
then
\begin{align*}
\mu\left(A\cap C\right) & =\mu\left(S_{1}\right)+\mu\left(\left(A\cap C\right)\setminus S_{1}\right)\\
 & \leqslant\frac{1}{2}\mu\left(A\cap C\right)+\mu\left(\left(A\cap C\right)\setminus S_{1}\right)\,,
\end{align*}
that is $\frac{1}{2}\mu\left(A\cap C\right)\leqslant\mu\left(\left(A\cap C\right)\setminus S_{1}\right).$
Now, 
\begin{align*}
\mu\otimes P\big(A\times A^{\complement}\big) & \geqslant\mu\otimes P\left(\left(\left(A\cap C\right)\setminus S_{1}\right)\times A^{\complement}\right)\\
 & \geqslant\frac{1}{2}\cdot\varepsilon\cdot\mu\left(\left(A\cap C\right)\setminus S_{1}\right)\\
 & \geqslant\frac{1}{4}\cdot\varepsilon\cdot\mu\left(A\cap C\right).
\end{align*}
Similarly if $\mu\left(S_{2}\right)\leqslant\frac{1}{2}\mu\left(A^{\complement}\cap C\right)$
then
\begin{align*}
\mu\big(A^{\complement}\cap C\big) & =\mu\big(S_{2}\big)+\mu\left(\big(A^{\complement}\cap C\big)\setminus S_{2}\right).\\
 & \leqslant\frac{1}{2}\mu\big(A^{\complement}\cap C\big)+\mu\left(\big(A^{\complement}\cap C\big)\setminus S_{2}\right)
\end{align*}
that is $\frac{1}{2}\mu\big(A^{\complement}\cap C\big)\leqslant\mu\left(\big(A^{\complement}\cap C\big)\setminus S_{2}\right)$
and arguing as before:
\begin{align*}
\mu\otimes P\big(A^{\complement}\times A\big) & \geqslant\mu\otimes P\left(\left(\big(A^{\complement}\cap C\big)\setminus S_{2}\right)\times A\right)\\
 & \geqslant\frac{1}{2}\cdot\varepsilon\cdot\mu\left(\big(A^{\complement}\cap C\big)\setminus S_{2}\right)\\
 & \geqslant\frac{1}{4}\cdot\varepsilon\cdot\mu\left(A^{\complement}\cap C\right).
\end{align*}
As noticed by \cite{Dwivedi-Chen-Wainwright-JMLR:v20:19-306}, reversibility
is not required to establish the following
\begin{align*}
\mu\otimes P\left(A\times A^{\complement}\right) & =\mu\otimes P\left(\mathsf{X}\times A^{\complement}\right)-\left[\mu\otimes P\left(A^{\complement}\times\mathsf{X}\right)-\mu\otimes P\left(A^{\complement}\times A\right)\right]\\
 & =\mu\left(A^{\complement}\right)-\mu\left(A^{\complement}\right)+\mu\otimes P\left(A^{\complement}\times A\right)\\
 & =\mu\otimes P\left(A^{\complement}\times A\right)\,,
\end{align*}
and this allows us to establish our first claim. Using the fact that
for $B\in\mathscr{E}$,
\[
1\geq\frac{\mu(B\cap C)}{\mu(C)},
\]
we may also deduce that if $\mu\left(S_{1}\right)\leqslant\frac{1}{2}\mu\left(A\cap C\right)$
or $\mu\left(S_{2}\right)\leqslant\frac{1}{2}\mu\left(A^{\complement}\cap C\right)$
then
\[
\mu\otimes P\big(A\times A^{\complement}\big)\geqslant\frac{1}{4\mu(C)}\cdot\varepsilon\cdot\mu\left(A\cap C\right)\cdot\mu\big(A^{\complement}\cap C\big)\,.
\]
In the second case, $\mu\left(S_{1}\right)>\frac{1}{2}\mu\left(A\cap C\right)$
and $\mu\big(S_{2}\big)>\frac{1}{2}\mu\left(A^{\complement}\cap C\right)$.
We then compute
\begin{align*}
\mu\otimes P\left(A\times A^{\complement}\right) & =\frac{1}{2}\mu\otimes P\left(A\times A^{\complement}\right)+\frac{1}{2}\mu\otimes P\left(A^{\complement}\times A\right)\\
 & \geqslant\frac{1}{2}\mu\otimes P\left(\left(A\cap C\cap S_{1}^{\complement}\right)\times A^{\complement}\right)+\frac{1}{2}\mu\otimes P\left(\left(A^{\complement}\cap C\cap S_{2}^{\complement}\right)\times A\right)\\
 & \geqslant\frac{1}{4}\cdot\varepsilon\cdot\mu\left(A\cap C\cap S_{1}^{\complement}\right)+\frac{1}{4}\cdot\varepsilon\cdot\mu\left(A^{\complement}\cap C\cap S_{2}^{\complement}\right)\\
 & =\frac{1}{4}\cdot\varepsilon\cdot\mu\left(C\cap\left(S_{1}\cup S_{2}\right)^{\complement}\right)\\
 & =\frac{1}{4}\cdot\varepsilon\cdot\mu\left(S_{3}\right)\,.
\end{align*}
Now for $\left(z,z'\right)\in S_{1}\times S_{2}$ we have
\begin{align*}
\left\Vert P\left(z,\cdot\right)-P\left(z',\cdot\right)\right\Vert _{\mathrm{TV}} & \geqslant P\left(z,A\right)-P\left(z',A\right)\\
 & =1-P\left(z,A^{\complement}\right)-P\left(z',A\right)\\
 & \geqslant1-\varepsilon.
\end{align*}
This implies that $d\left(S_{1},S_{2}\right)=\inf\left\{ \left|z-z'\right|\colon\left(z,z'\right)\in S_{1}\times S_{2}\right\} >\delta$,
since for $z,z'\in\mathsf{X}$, $\left|z-z'\right|\leqslant\delta$
implies
\[
\left\Vert P\left(z,\cdot\right)-P\left(z',\cdot\right)\right\Vert _{\mathrm{TV}}<1-\varepsilon.
\]
From Lemma~\ref{lem:cousin-vempala} applied to the measure $\mu_{C}\left(\cdot\right):=\mu\left(\cdot\cap C\right)/\mu\left(C\right)$,
we can thus write that
\begin{align*}
\mu\left(S_{3}\right) & \geqslant\mu\left(C\right)\cdot\log2\cdot\sqrt{m}\cdot d\left(S_{1},S_{2}\right)\cdot\frac{\mu\left(S_{1}\right)}{\mu\left(C\right)}\cdot\frac{\mu\left(S_{2}\right)}{\mu\left(C\right)}\\
 & \geqslant\frac{\log2}{\mu\left(C\right)}\cdot\sqrt{m}\cdot\delta\cdot\left(\frac{1}{2}\mu\left(A\cap C\right)\right)\cdot\left(\frac{1}{2}\mu\big(A^{\complement}\cap C\big)\right)
\end{align*}
and consequently that
\[
\mu\otimes P\left(A\times A^{\complement}\right)\geqslant\mu\left(C\right)^{-1}\cdot\frac{\log2}{16}\cdot\varepsilon\cdot\sqrt{m}\cdot\delta\cdot\mu\big(A\cap C\big)\cdot\mu\big(A^{\complement}\cap C\big)\,.
\]
The second result then follows. To obtain the first result, since
$p\cdot\left(1-p\right)\geqslant\frac{1}{2}\cdot\min\left(p,1-p\right)$
for $0<p<1$, it holds that
\begin{align*}
\mu\left(A\cap C\right)\cdot\mu\big(A^{\complement}\cap C\big) & =\mu\left(C\right)^{2}\cdot\left[\frac{\mu\left(A\cap C\right)}{\mu\left(C\right)}\right]\cdot\left[1-\frac{\mu\left(A\cap C\right)}{\mu\left(C\right)}\right]\\
 & \geqslant\frac{1}{2}\cdot\mu\left(C\right)^{2}\cdot\min\left\{ \frac{\mu\left(A\cap C\right)}{\mu\big(C\big)},\frac{\mu\left(A^{\complement}\cap C\right)}{\mu\big(C\big)}\right\} \\
 & =\frac{1}{2}\cdot\mu\left(C\right)\cdot\min\left\{ \mu\left(A\cap C\right),\mu\left(A^{\complement}\cap C\right)\right\} .
\end{align*}
\end{proof}
\begin{cor}
\label{cor:iso-to-local-pi}Under the conditions of Theorem~\ref{thm:local-conductance-lower-bound},
we can deduce that if $P$ is also reversible then a restricted Poincaré
inequality holds for $P$ on $C$, and a local Poincaré inequality
holds for $P$ on $C$.
\end{cor}

\begin{proof}
We obtain from the conclusion of Theorem~\ref{thm:local-conductance-lower-bound}
that for any $A\in\mathcal{E}$ with $A\subseteq C$,
\[
\mu_{C}\otimes P_{C}(A\times A^{\complement})\geq\frac{\varepsilon}{4}\min\left\{ 1,\frac{\log2}{8}\delta\sqrt{m}\right\} \min\left\{ \mu_{C}\left(A\right),\mu_{C}\left(A^{\complement}\right)\right\} ,
\]
and
\[
\mu_{C}\otimes P_{C}\big(A\times A^{\complement}\big)\geqslant\frac{\varepsilon}{4}\min\left\{ 1,\frac{\log2}{4}\delta\sqrt{m}\right\} \mu_{C}\left(A\right)\mu_{C}\left(A^{\complement}\right).
\]
Since $P$ is reversible, (\ref{eq:cheeger-ineqs}) implies that $P_{C}$
admits a strong Poincaré inequality. Lemma~\ref{lem:local_SPI_to_SPI}
then implies that a local Poincaré holds for $P$ on $C$.
\end{proof}

\subsection{Restricted Markov chains and vanishing Poincaré constants \label{subsec:establish-WPI-Restricted-Markov-chains}}

In this subsection, we establish a link between the existence of SPIs
for restrictions of a Markov chain $P$ to suitable sets and WPIs
for the unrestricted chain. Roughly speaking, for subgeometric chains,
it is possible that the restriction of the chain to a `nice' set $A$
exhibits a strong Poincaré inequality, but as $\mu\left(A\right)$
grows, the constant in this inequality necessarily degenerates. We
will show that the rate at which this constant degenerates as $A$
grows allows one to deduce a quantitative weak Poincaré inequality
for $P$. In what follows, we let $\Phi=\left\Vert \cdot\right\Vert _{{\rm osc}}^{2}$.
Note that ${\rm var}_{\mu}(f)\leq\Phi(f)$.

In the following result we upper and lower bound $\mathcal{E}(P,f)$
by quantities involving Dirichlet forms associated to the restriction
of $P$ to a set $A$.

\begin{lem}
\label{lem:P-a-rev-e-bound}Let $P$ be $\mu$-reversible. Let $A\in\mathscr{E}_{+}$
and $P_{A}$ be the $\mu_{A}$-reversible restriction of $P$ to $A$.
Then 
\[
\mu(A)\mathcal{E}(P_{A},f)\leq\mathcal{E}(P,f)\leq\mu(A)\mathcal{E}(P_{A},f)+\mu(A^{\complement})\Phi(f).
\]
\end{lem}

\begin{proof}
We have
\begin{align*}
\mathcal{E}(P,f) & =\frac{1}{2}\int\mu({\rm d}x)P(x,{\rm d}y)\left\{ f(x)-f(y)\right\} ^{2}\\
 & \geq\frac{1}{2}\mu(A)\int\mu_{A}({\rm d}x)P_{A}(x,{\rm d}y)\left\{ f(x)-f(y)\right\} ^{2}\\
 & =\mu(A)\mathcal{E}(P_{A},f),
\end{align*}
and
\begin{align*}
\mathcal{E}(P,f) & =\frac{1}{2}\int\mu({\rm d}x)P(x,{\rm d}y)\left\{ f(x)-f(y)\right\} ^{2}\\
 & =\frac{1}{2}\int_{A\times A}\mu({\rm d}x)P(x,{\rm d}y)\left\{ f(x)-f(y)\right\} ^{2}\\
 & \hspace{3cm}+\frac{1}{2}\int_{(A\times A)^{\complement}}\mu({\rm d}x)P(x,{\rm d}y)\left\{ f(x)-f(y)\right\} ^{2}\\
 & \leq\frac{\mu(A)}{2}\int\mu_{A}({\rm d}x)P_{A}(x,{\rm d}y)\left\{ f(x)-f(y)\right\} ^{2}\\
 & \hspace{3cm}+\frac{1}{2}\Phi(f)\left\{ \mu\otimes P(A^{\complement}\times\mathsf{E})+\mu\otimes P(A\times A^{\complement})\right\} \\
 & \leq\mu(A)\mathcal{E}(P_{A},f)+\mu(A^{\complement})\Phi(f),
\end{align*}
where we have used the fact that since $P$ is $\mu$-invariant,
\[
\mu\otimes P(A\times A^{\complement})\leq\mu\otimes P(E\times A^{\complement})=\mu(A^{\complement}).
\]
\end{proof}
Now let $\Pi$ be the Markov kernel such that $\Pi(x,\cdot)=\mu(\cdot)$
for all $x\in\mathsf{E}$, and let $\Pi_{A}$ be the corresponding
restriction as in Definition~\ref{def:restricted_C}: for $B\in\mathscr{E}$,
$\Pi_{A}(x,B)=\mu(A\cap B)+\mu(A^{\complement}){\bf 1}_{B}(x)$, which
is not necessarily equal to $\mu_{A}(B)=\mu(A)^{-1}\mu(A\cap B)$.
In fact, we have the following.
\begin{lem}
For $A\in\mathscr{E}_{+}$, $\mathcal{E}(\Pi_{A},f)=\mu(A){\rm var}_{\mu_{A}}(f)$.
\end{lem}

\begin{proof}
We have
\begin{align*}
\mathcal{E}(\Pi_{A},f) & =\frac{1}{2}\int\mu_{A}({\rm d}x)\Pi_{A}(x,{\rm d}y)\left\{ f(x)-f(y)\right\} ^{2}\\
 & =\frac{1}{2}\int\mu_{A}({\rm d}x)\Pi(x,{\rm d}y){\bf 1}_{A}(y)\left\{ f(x)-f(y)\right\} ^{2}\\
 & =\frac{1}{2}\mu(A)\int\mu_{A}({\rm d}x)\mu_{A}({\rm d}y)\left\{ f(x)-f(y)\right\} ^{2}\\
 & =\mu(A){\rm var}_{\mu}(f).
\end{align*}
\end{proof}
\begin{cor}
\label{cor:mu-a-var-bound}Letting $P=\Pi$ in Lemma~\ref{lem:P-a-rev-e-bound},
we obtain
\[
\mu(A)^{2}{\rm var}_{\mu_{A}}(f)\leq{\rm var}_{\mu}(f)\leq\mu(A)^{2}{\rm var}_{\mu_{A}}(f)+\mu(A^{\complement})\Phi(f),\qquad A\in\mathscr{E}_{+}.
\]
\end{cor}

\begin{thm}
\label{thm:vanishing-gaps-wpi}Let $P$ be $\mu$-reversible. For
$A\in\mathscr{E}_{+}$, define $\gamma_{P}(A)$ to be the (right)
``spectral gap''
\[
\gamma_{P}(A)=\inf_{f\in\mathrm{L}^{2}(\mu)}\frac{\mathcal{E}(P_{A},f)}{{\rm var}_{\mu_{A}}(f)}.
\]
Then a $P$ satisfies a $(\Phi,\beta)$-WPI with
\[
\beta(s)=1\wedge\inf_{A\in\mathscr{E}_{+}}\left\{ \mu(A^{\complement}):\gamma_{P}(A)\geq\frac{\mu(A)}{s}\right\} .
\]
\end{thm}

\begin{proof}
Let $s>0$. If $\mathcal{S}=\left\{ A\in\mathscr{E}_{+}:\gamma_{P}(A)\geq\mu(A)/s\right\} $
is empty, we may take $\beta(s)=1$ since ${\rm var}_{\mu}(f)\leq\Phi(f)$.
Otherwise, let $A\in\mathcal{S}$. By Corollary~\ref{cor:mu-a-var-bound}
and Lemma~\ref{lem:P-a-rev-e-bound}, we obtain that for any $f\in\mathrm{L}^{2}(\mu)$,
\begin{align*}
{\rm var}_{\mu}(f) & \leq\mu(A)^{2}{\rm var}_{\mu_{A}}(f)+\mu(A^{\complement})\Phi(f)\\
 & \leq\frac{\mu(A)^{2}}{\gamma_{P}(A)}\mathcal{E}(P_{A},f)+\mu(A^{\complement})\Phi(f)\\
 & \leq\frac{\mu(A)}{\gamma_{P}(A)}\mathcal{E}(P,f)+\mu(A^{\complement})\Phi(f)\\
 & \leq s\mathcal{E}(P,f)+\mu(A^{\complement})\Phi(f),
\end{align*}
from which we may deduce that one may take $\beta(s)=\mu(A^{\complement})$.
The result then follows by taking the infimum over $A\in\mathcal{S}$.
\end{proof}
We may revisit the WPI obtained for the IMH in \cite{ALPW2021} from
this perspective as follows; the argument is essentially the same.
\begin{example}
Consider the IMH with target $\pi$ and proposal $q$, and let $w={\rm d}\pi/{\rm d}q$.
If we define $A=\{x:w(x)\leq s\}$ then we may write
\begin{align*}
\mathcal{E}(P_{A},f) & =\frac{1}{2}\int\pi_{A}({\rm d}x)P_{A}(x,{\rm d}y)\left\{ f(x)-f(y)\right\} ^{2}\\
 & =\frac{1}{2}\int\pi_{A}({\rm d}x)q({\rm d}y)\left\{ 1\wedge\frac{w(y)}{w(x)}\right\} {\bf 1}_{A}(y)\left\{ f(x)-f(y)\right\} ^{2}\\
 & \geq\frac{1}{2s}\pi(A)\int\pi_{A}({\rm d}x)\pi_{A}({\rm d}y)\left\{ f(x)-f(y)\right\} ^{2}\\
 & =\frac{\pi(A)}{s}{\rm var}_{\pi_{A}}(f).
\end{align*}
It follows that $\gamma_{P}(A)\geq\pi(A)/s$, and so by Theorem~\ref{thm:vanishing-gaps-wpi}
we may take $\beta(s)=\pi(A^{\complement})$ in a $(\left\Vert \cdot\right\Vert _{{\rm osc}}^{2},\beta)$-WPI.
This argument is clearly related to the well-known fact that the IMH
has a spectral gap if and only if $w$ is upper bounded by a finite
constant \cite[Theorem~2.1]{mengersen1996rates}, and we obtain the
subgeometric rate here by considering the measures of a sequence of
sets on which $w$ is upper bounded by an increasing sequence of constants.
\end{example}

\begin{rem}
One may equivalently deduce a $(\Phi,\alpha)$-WPI with 
\[
\alpha(r)=\inf_{A\in\mathscr{E}_{+}}\left\{ \frac{\mu(A)}{\gamma_{P}(A)}:\mu(A)\geq1-r\right\} .
\]
If we define $\gamma_{P}(t)=\sup_{A\in\mathscr{E}_{+}}\{\gamma_{P}(A):\mu(A)\geq1-t\}$
then we see that $\alpha(r)\leq1/\gamma_{P}(r)$ and the rate at which
$\gamma_{P}(r)\to0$ as $r\to0$ provides an upper bound on the convergence
rate.
\end{rem}

To our knowledge, the observation that a subgeometric rate of convergence
can be related to the rate of decay of the spectral gap on an appropriate
sequence of sets is novel. Considering restrictions of $\mu$ and
$P$ to a set $A$ is reminiscent of the notion of spectral profile
introduced by \cite{goel2006mixing}, which involves instead considering
$\mathcal{E}(P,f)/{\rm var_{\mu}}(f)$ when $f\geq0$ has support
restricted to appropriately chosen sets $(S_{t})$, and considering
the decay as $\mu(S_{t})\to1$. However, it is not clear how to relate
the two concepts, and we note that the spectral profile was introduced
to obtain bounds on mixing times whereas we are interested here in
subgeometric rates of convergence.

\begin{rem}
Clearly if $P$ has a (right) spectral gap $\gamma_{P}=\gamma_{P}(\mathsf{E})>0$
then we have $\beta(s)=0$ for $s\geq\gamma_{P}^{-1}$.
\end{rem}

For some Markov kernels $P$ with state space $\mathsf{E}=\mathbb{R}^{d}$,
the restriction of $P$ to a ball around the origin will have a non-zero
right spectral gap. In such cases, the sequence of balls with increasing
radius defines a sequence of restrictions and the rate at which the
gap decreases together with the rate at which the $\mu$-measure of
the balls tends to $1$ can be used to deduce a WPI.
\begin{example}
Assume that for a Markov kernel $P$ there is a family of sets $(A_{t})_{t\geq1}$
constants $C,a,b>0$ such that for all $t\geq1$,
\[
\gamma_{P}(A_{t})\geq Ct^{-a},\qquad\mu(A_{t}^{\complement})\leq Dt^{-b}.
\]
Then we find that for $\gamma_{P}(A_{t})\geq\frac{1}{s}$ is satisfied
by taking $t=(Cs)^{\frac{1}{a}}$, and we then find $\mu(A_{t}^{\complement})\leq D(Cs)^{-\frac{b}{a}}$.
Hence $P$ satisfies a $(\Phi,\beta)$-WPI with $\beta(s)=D(Cs)^{-\frac{b}{a}}$.
This argument may be valid when $P$ is a random-walk Metropolis kernel
on a heavy-tailed target, and $A_{t}$ is a ball of radius $t$ around
the origin, although proving rigorously the lower bounds on $\gamma_{P}(A_{t})$
is not trivial. 
\end{example}

\section{Examples and applications\label{sec:Examples-and-applications}}

\subsection{Lower bounds for pseudo-marginal MCMC\label{subsec:example-Lower-bounds-pseudo}}

We consider a specific and theoretically tractable ABC example covered
by positive results from \cite{ALPW2021}. We show now that there
is a quantitative version of the argument in \cite{lee-latuszynski-2014}
that ABC with local proposals is subgeometric, and that the lower
bound on the polynomial rate matches the upper bound given by \cite{ALPW2021}.

In this subsection, we let $\tilde{P}$ be the pseudo-marginal Markov
kernel, and in particular we focus on complementing the results in
\cite{ALPW2021}. For any measurable $A$ such that $(x,w)\not\in A$,
we may write
\[
\tilde{P}(x,w;A)=\int q(x,{\rm d}y)Q_{y}({\rm d}u)\left\{ 1\wedge r(x,y)\frac{u}{w}\right\} {\bf 1}_{A}(y,u),
\]
where $\{Q_{x}:x\in\mathsf{E}\}$ is a family of probability measures
such that $\int Q_{x}({\rm d}w)w=1$. We focus on the ABC example
in \cite[Section~4.3]{ALPW2021}, with some prior $\nu$ and an approximate,
intractable likelihood $\ell_{{\rm ABC}}$. In particular, for some
$N\in\mathbb{N}$ and any $x\in\mathsf{E}$ we denote
\[
Q_{x}(A)=Q_{x,N}(A)=\mathbb{P}_{x}\left(\frac{1}{N}\sum_{i=1}^{N}W_{i}\in A\right),
\]
where under $\mathbb{P}_{x}$, $W_{i}=\frac{1}{\ell_{{\rm ABC}}(x)}B_{i}$
and $B_{1},\ldots,B_{N}$ are independent ${\rm Bernoulli}(\ell_{{\rm ABC}}(x))$
random variables. The parameter $N$ thereby controls the concentration
of $W\sim Q_{x}$ around $1$, and we use the subscript $N$ to emphasize
this dependence.
\begin{prop}
Consider the general ABC example in \cite[Section~4.3]{ALPW2021},
and take for $a,q\in(0,1)$,

\begin{align*}
\nu\left(x\right) & =\left(1-q\right)q^{x-1}{\bf 1}_{\{1,2,\ldots\}}(x),\\
\ell_{{\rm ABC}}\left(x\right) & =a^{x-1}{\bf 1}_{\{1,2,\ldots\}}(x),
\end{align*}
and $q\left(x,x-1\right)=q\left(x,x+1\right)=1/2$. Then, for any
$N\geqslant1$, if $\tilde{P}$ admits a $\left(\Phi,\beta\right)$-WPI
then $\beta\left(s\right)\in\Omega\left(s^{-\frac{\log(aq)}{\log(a)}}\right)$.
\end{prop}

\begin{proof}
The ABC posterior is $\pi_{{\rm ABC}}\left(x\right)=\left(1-aq\right)\left(aq\right)^{x-1}{\bf 1}_{\{1,2,\ldots\}}(x)$,
i.e. ${\rm Geometric}\left(1-aq\right)$. We define the pseudo-marginal
target distribution on $\left(x,w\right)$ to be $\tilde{\pi}$. We
may define the set, with $\rho\in\mathbb{N}$,
\[
A_{\rho}=\left\{ \left(x,w\right):x>\rho\right\} ,
\]
and we obtain $\tilde{\pi}\left(A_{\rho}\right)=\pi_{{\rm ABC}}\left(x>\rho\right)=\left(aq\right)^{\rho}$. 

Let $u<aq/4$, and take $\rho=\left\lfloor \log\left(2u\right)/\log\left(aq\right)\right\rfloor $.
Since $x-1\leq\left\lfloor x\right\rfloor \leq x$, we deduce that
$2u\leq\tilde{\pi}(A_{\rho})<1/2$, and hence $\tilde{\pi}(A_{\rho})\tilde{\pi}(A_{\rho}^{\complement})>2u\cdot\frac{1}{2}=u$.
Now, we find that for $\left(x,w\right)\in A_{\rho}$, and any $N\in\mathbb{N}$,
we have the bound
\begin{align*}
\tilde{P}\left(x,w;A_{\rho}^{\complement}\right) & \leqslant{\bf 1}_{\{\rho+1\}}(x)q\left(\rho+1,\rho\right)\int Q_{\rho,N}\left({\rm d}u\right)\left\{ 1\wedge\frac{\pi_{{\rm ABC}}\left(\rho\right)}{\pi_{{\rm ABC}}\left(\rho+1\right)}\cdot\frac{u}{w}\right\} \\
 & \leqslant{\bf 1}_{\{\rho+1\}}(x)\frac{1}{2}Q_{\rho,N}\left(u>0\right)\\
 & \leqslant{\bf 1}_{\{\rho+1\}}(x)\frac{N}{2}a^{\rho-1},
\end{align*}
where we have used Bernoulli's inequality to deduce that
\begin{align*}
Q_{\rho,N}\left(u>0\right) & =1-\left(1-a^{\rho-1}\right)^{N}\\
 & \leqslant1-\left(1-Na^{\rho-1}\right)\\
 & \leqslant Na^{\rho-1}.
\end{align*}
Hence, we obtain that
\begin{align*}
\tilde{\pi}\otimes\tilde{P}\left(A_{\rho},A_{\rho}^{\complement}\right) & \leqslant\tilde{\pi}\left(A_{\rho}\right)\frac{N}{2}a^{\rho-1}\\
 & \leqslant\tilde{\pi}\left(A_{\rho}\right)\frac{N}{2}a^{\frac{\log\left(2u\right)}{\log\left(aq\right)}-2}.
\end{align*}
It follows that the weak conductance satisfies
\begin{align*}
\kappa\left(u\right) & \leqslant\frac{\tilde{\pi}\otimes\tilde{P}\left(A_{\rho},A_{\rho}^{\complement}\right)}{\tilde{\pi}\otimes\tilde{\pi}\left(A_{\rho},A_{\rho}^{\complement}\right)}\\
 & \leqslant Na^{\frac{\log\left(2u\right)}{\log\left(aq\right)}-2}\\
 & =Na^{-2}\left(2u\right)^{\frac{\log\left(a\right)}{\log\left(aq\right)}}.
\end{align*}
and so we see that $\kappa\left(u\right)\in\mathcal{O}\left(u^{\frac{\log\left(a\right)}{\log\left(aq\right)}}\right)$.
This then implies that $\alpha^{\star}\left(r\right)\in\Omega\left(r^{-\frac{\log\left(a\right)}{\log\left(aq\right)}}\right)$
as $r\downarrow0$ and $\beta^{\star}\left(s\right)\in\Omega\left(s^{-\frac{\log\left(aq\right)}{\log\left(a\right)}}\right)$. 
\end{proof}
\begin{rem}
\cite{ALPW2021} considered the setting where the marginal chain is
geometric with strong Poincaré constant $C_{{\rm P}}$. They showed
that one may take $\beta(s)=\beta'(C_{{\rm P}}s)/C_{{\rm P}}$ where
$\beta'(s)=\tilde{\pi}\left(w\geqslant s\right)$. In the case $N=1$,
we see that 
\begin{align*}
\tilde{\pi}\left(w\geqslant s\right) & =\tilde{\pi}\left(\left\{ x:\frac{1}{a^{x-1}}\geqslant s\right\} \right)\\
 & =\pi\left(\left\{ x:x\geqslant\frac{\log\left(s\right)}{-\log\left(a\right)}+1\right\} \right)\\
 & =\left(aq\right)^{\left\lceil \frac{\log\left(s\right)}{-\log\left(a\right)}\right\rceil }\\
 & \sim s^{-\frac{\log\left(aq\right)}{\log\left(a\right)}}.
\end{align*}
Alternatively, \cite{ALPW2021} showed that for any $N\in\mathbb{N}$
and $p\in\mathbb{N}$, $\beta\left(s\right)\in\mathcal{O}\left(s^{-p}\right)$
if $\int\nu\left({\rm d}x\right)\ell_{{\rm ABC}}\left(x\right)^{-(p-1)}<\infty$,
which in this case corresponds to $q/a^{p-1}<1$, or equivalently
$p<\log\left(aq\right)/\log\left(q\right)$, matching the lower bound
above.
\end{rem}

\subsection{Lower bounds for RWM targeting heavy-tailed distributions\label{subsec:example-Lower-bounds-RWM-heavy-tail}}

In this subsection, we assume $\left|\cdot\right|$ is a norm. We
consider $P$ a $\mu$-invariant kernel that is local in the sense
that 
\[
b\left(r\right):=\inf_{x\in\mathsf{E}}P\left(x,\mathcal{B}(x,r)\right),
\]
is a real-valued function with $\lim_{r\to\infty}b(r)=1$. We assume
in this subsection that $\Phi(\cdot)=\left\Vert \cdot\right\Vert _{{\rm osc}}^{2}$.

When $\mu$ has polynomial tails, we seek to demonstrate that arguments
used to show that $\kappa\left(0\right)=0$, and hence that $P$ does
not admit a spectral gap, may also be used to lower bound $\alpha$
or $\beta$ in a WPI for $P$. In this sense, such arguments can be
made quantitative, although we require more information on the measure
of suitable sets to deduce rate information. The following argument
is inspired by the approach taken in the proof of \cite[Theorem~6.3]{papaspiliopoulos-roberts-2008}.

The first lemma upper bounds $\mu\otimes P\left(A\times A^{\complement}\right)$.
\begin{lem}
\label{lem:heavy-muP-ball}Let
\[
\phi\left(\rho,K\right):=\frac{\mathbb{P}_{\mu}\left(\left|X\right|>\rho+K\right)}{\mathbb{P}_{\mu}\left(\left|X\right|>\rho\right)}.
\]
Then with $A=\mathcal{B}(0,\rho)^{\complement}$ and any $K>0$, 
\[
\mu\otimes P\left(A\times A^{\complement}\right)\leqslant\mu\left(A\right)\left\{ 1-\phi\left(\rho,K\right)\cdot b\left(K\right)\right\} .
\]
\end{lem}

\begin{proof}
Let $\left(X,Y\right)\sim\mu_{A}\otimes P$, where $\mu_{A}$ is as
defined in Definition~\ref{def:restricted_C}, and consider the representation
$Y=X+\xi_{X}$. We bound
\begin{align*}
\mathbb{P}\left(\left|Y\right|>\rho\right) & \geqslant\mathbb{P}\left(\left|X\right|>\rho+K,\left|Y\right|>\rho\right)\\
 & =\mathbb{P}\left(\left|X\right|>\rho+K,\left|X+\xi_{X}\right|>\rho\right)\\
 & \geqslant\mathbb{P}\left(\left|X\right|>\rho+K,\left|\xi_{X}\right|\leqslant K\right)\\
 & \geqslant\phi\left(\rho,K\right)\cdot b\left(K\right).
\end{align*}
where we have used that the two conditions $\left|X\right|>\rho+K,\left|\xi_{X}\right|\leqslant K\implies\rho+K-K\leq|X|-|\xi_{X}|\leq|X+\xi_{X}|$
and the fact that $X\sim\mu_{A}$. It follows that
\[
\int\mu_{A}\left({\rm d}x\right)P\left(x,A\right)\geqslant\phi\left(\rho,K\right)\cdot b\left(K\right),
\]
and hence
\begin{align*}
\mu\otimes P\left(A\times A^{\complement}\right) & =\int_{A}\mu\left({\rm d}x\right)P\left(x,A^{\complement}\right)\\
 & =\mu\left(A\right)\int\mu_{A}\left({\rm d}x\right)P\left(x,A^{\complement}\right)\\
 & \leqslant\mu\left(A\right)\left\{ 1-\phi\left(\rho,K\right)b\left(K\right)\right\} .
\end{align*}
\end{proof}
In the following, the $\mu$ considered is a multi-dimensional version
of the stylized one-dimensional case considered in \cite[Eq.  52]{jarner-tweedie-2003}.
Although the argument is likely to be useful in other cases, it is
necessary to have fairly precise control on both $\mu\left(\mathcal{B}(0,\rho)\right)$
and $\mu\left(\mathcal{B}(0,\rho)^{\complement}\right)$ in order
to quantify how $\phi\left(\rho,K\right)$ tends to $1$ as $\rho$
and $K$ increase.
\begin{prop}
\label{prop:heavy-lim-kappa}Assume that for some $t>0$,
\[
\mu\left(\mathcal{B}(0,\rho)^{\complement}\right)=\rho^{-t},\qquad\rho\geq1.
\]
Assume there exist $D,\eta>0$ such that $P$ satisfies 
\[
b\left(K\right)\geqslant1-DK^{-\eta},\qquad K>0,
\]
where $b(\cdot)$ is as defined in Lemma~\ref{lem:heavy-muP-ball}.
Then $\beta^{\star}\left(s\right)\in\Omega\left(s^{-t\frac{\eta+1}{\eta}}\right)$.
\end{prop}

\begin{proof}
Let $\rho_{0}=2^{1/t}$, which satisfies $\mu\left(\mathcal{B}(0,\rho_{0})\right)=\frac{1}{2}$,
from which we may deduce that $\mu\left(\mathcal{B}(0,\rho)\right)>1/2$
for all $\rho>\rho_{0}$. This will be the smallest $\rho$ which
we consider, and it satisfies
\[
\mu\otimes\mu\left(B_{\rho_{0}}\left(0\right)^{\complement}\times B_{\rho_{0}}\left(0\right)\right)=\frac{1}{2}\rho_{0}^{-t}=\frac{1}{4}=:u_{0}.
\]
Given any $u<u_{0}$, we take $\rho=\left(2u\right)^{-\frac{1}{t}}>\rho_{0}$
and $A=\mathcal{B}(0,\rho)^{\complement}$, which satisfies $\mu\left(A^{\complement}\right)>1/2$,
and so it holds that
\[
\mu\otimes\mu\left(A\times A^{\complement}\right)>\frac{1}{2}\rho^{-t}=u.
\]
By Lemma~\ref{lem:heavy-muP-ball}, we obtain that for any $K>0$,
\begin{align*}
\kappa\left(u\right) & \leqslant\frac{\mu\otimes P\left(A\times A^{\complement}\right)}{\mu\otimes\mu\left(A\times A^{\complement}\right)}\\
 & =\frac{1-\phi\left(\rho,K\right)b\left(K\right)}{\mu\left(A^{\complement}\right)}\\
 & \leqslant2\left\{ 1-\phi\left(\rho,K\right)b\left(K\right)\right\} .
\end{align*}
Letting $v=2u$, we thus find that
\begin{align*}
\phi\left(\rho,K\right) & =\frac{\mathbb{P}_{\mu}\left(\left|X\right|>\rho+K\right)}{\mathbb{P}_{\mu}\left(\left|X\right|>\rho\right)}\\
 & =\frac{\left(\rho+K\right)^{-t}}{\rho^{-t}}\\
 & =\frac{1}{v\left(v^{-\frac{1}{t}}+K\right)^{t}}\\
 & =\frac{1}{\left(1+v^{\frac{1}{t}}K\right)^{t}}.
\end{align*}
Hence, we have the bound
\[
1-\phi\left(\rho,K\right)b\left(K\right)\leqslant1-\frac{1-DK^{-\eta}}{\left(1+v^{\frac{1}{t}}K\right)^{t}},
\]
and by taking $K=v^{-\frac{1}{t+\eta t}}$, we may deduce that
\[
\lim_{v\downarrow0}\left\{ v^{-\frac{1}{t}\cdot\frac{\eta}{\eta+1}}\left\{ 1-\frac{1-DK^{-\eta}}{\left(1+v^{\frac{1}{t}}K\right)^{t}}\right\} \right\} =t+D,
\]
from which we may conclude that $\kappa\left(u\right)\in\mathcal{O}\left(u^{\frac{1}{t}\cdot\frac{\eta}{\eta+1}}\right)$.
Since $\alpha^{\star}\left(r\right)\geqslant\frac{1}{2\kappa\left(2r\right)}$
by Remark~\ref{rem:cheeger-remark}, we obtain $\alpha^{\star}\left(r\right)\in\Omega\left(r^{-\frac{1}{t}\cdot\frac{\eta}{\eta+1}}\right)$
as $r\downarrow0$, and so $\beta^{\star}\left(s\right)\in\Omega\left(s^{-t\frac{\eta+1}{\eta}}\right)$.
\end{proof}
\begin{rem}
If $P$ is $\mu$-reversible, one may then deduce that by Proposition~\ref{prop:rev-L2-conv-rate-lower-bound}
if $\epsilon>0$ then $\left\Vert P^{n}f\right\Vert ^{2}$ cannot
be in $\mathcal{O}\left(n^{-t\frac{\eta+1}{\eta}-\epsilon}\right)$
for all $f\in\mathrm{L}_{0}^{2}(\mu)$ with $\Phi(f)<\infty$. We
see that, similar to \cite{jarner2007convergence} and \cite{jarner-tweedie-2003},
the lower bounds suggest that faster rates are possible if $\eta$
is close to $0$, i.e. $P\left(x,\cdot\right)$ is heavy-tailed for
all $x$.
\end{rem}

\subsection{Spectral gap of the RWM in high-dimensions\label{subsec:Spectral-gap-of-RWM}}

We \nocite{miclo2000trous} let $\mathsf{X}=\mathsf{Z}=\mathbb{R}^{d}$
throughout. Let $P$ be the Markov transition probability of the Random
Walk Metropolis (RWM) with Gaussian proposal, defined for any $(x,A)\in\mathsf{X}\times\mathscr{X}$
\[
Q_{x}(A)=\int{\bf 1}_{A}(x+d^{-1/2}z)\,Q({\rm d}z),
\]
where $Q=\mathcal{N}\left(0,\sigma^{2}\mathrm{Id}\right)$. Then,
for any $\left(x,A\right)\in\mathsf{X}\times\mathscr{X}$,
\[
P\left(x,A\right)=\int_{A}\alpha\left(x,d^{-1/2}z\right)Q\left({\rm d}z\right)+\mathbf{1}_{A}\left(x\right)\left[1-\alpha\left(x\right)\right]\,,
\]
with for any $\left(x,z\right)\in\mathsf{X}\times\mathsf{Z}$, $\alpha\left(x,z\right):=\min\left\{ 1,\mathsf{r}\left(x,z\right)\right\} $
and
\begin{align}
\mathsf{r}\left(x,z\right) & :=\frac{\pi\left(x+z\right)}{\pi\left(x\right)},\nonumber \\
\alpha(x) & :=\int\alpha\left(x,d^{-1/2}z\right)Q\left({\rm d}z\right),\label{eq:alpha_RWM_def}
\end{align}
and $\pi:\mathsf{X}\to[0,\infty)$ is a target density with respect
to Lebesgue measure with $\pi\left(x\right)\propto\exp\left(-U\left(x\right)\right)$.
In this section and in Section~\ref{subsec:Spectral-gap-for-gaussian},
we denote by $\left|\cdot\right|$ the Euclidean norm in $\mathbb{R}^{d}$,
i.e. $\left|x\right|=\left(\sum_{i=1}^{d}x_{i}^{2}\right)^{1/2}$.
\begin{assumption}
\label{assu:target-distribution}We assume the following properties
of our target distribution:
\begin{enumerate}
\item $U$ is spherically symmetric with $U(x)=u(\left|x\right|^{2})$,
for some increasing function $u:[0,\infty)\to[0,\infty)$. In particular,
$U$ attains its minimum at $0$. 
\item For some $L\geqslant m>0$, the potential $U$ is $m$-strongly convex
and $L$-smooth, i.e. for all $x,z$, one has the bounds
\[
\frac{m}{2}\left|z\right|^{2}\leqslant U\left(x+z\right)-U\left(x\right)-\left\langle \nabla U\left(x\right),z\right\rangle \leqslant\frac{L}{2}\left|z\right|^{2}.
\]
\end{enumerate}
\end{assumption}

We impose here spherical symmetry on the potential to make our proof
simple, noting that similar results could be expected to hold without
this assumption. A very natural example of $\pi$ satisfying the above
is the normal distribution with covariance matrix $\sigma_{0}^{2}\mathrm{Id}$,
for which one can take $m=L=\frac{1}{\sigma_{0}^{2}}$. 
\begin{example}
\label{exa:gaussian-mL}Assume $\pi$ is $\mathcal{N}(0,\sigma_{0}^{2}I_{d})$,
so $U(x)=\frac{1}{2\sigma_{0^{2}}}\left|x\right|^{2}$. Then 
\[
U(x+z)-U(x)-\left\langle \nabla U(x),z\right\rangle =\frac{1}{2\sigma_{0}^{2}}\left|z\right|^{2},
\]
so we have $L=m=1/\sigma_{0}^{2}$.
\end{example}

Another natural class of examples with strongly convex and smooth
potentials (but not spherical symmetry) comes from considering Bayesian
posterior measures for which the prior is normal, and the log-likelihood
is concave with bounded Hessian.
\begin{example}
Consider the task of Bayesian logistic regression, taking as prior
$\pi_{0}=\mathcal{N}(0,\sigma_{0}^{2}I_{d})$, and observing covariate-response
pairs $\left\{ \left(a_{i},y_{i}\right)\right\} _{i=1}^{N}\subset\mathbb{R}^{d}\times\left\{ 0,1\right\} $.
The potential corresponding to the posterior measure is then given
by
\[
U\left(x\right)=\frac{1}{2\sigma_{0^{2}}}\left|x\right|^{2}+\sum_{i=1}^{N}\left\{ \log\left(1+\exp\left(-\left\langle a_{i},x\right\rangle \right)\right)-y_{i}\left\langle a_{i},x\right\rangle \right\} 
\]
Writing $A$ for the $n\times d$ matrix with columns given by the
$\left\{ a_{i}\right\} $, one can check that $U$ is $m$-strongly
convex and $L$-smooth with $m\geqslant\frac{1}{\sigma_{0}^{2}}$
and $L\leqslant\frac{1}{\sigma_{0}^{2}}+\frac{1}{4}\lambda_{{\rm Max}}\left(AA^{\top}\right)$.
\end{example}

The strategy of the proof of the following is to combine two different
coupling arguments, in combination with a global application of Theorem~\ref{thm:local-conductance-lower-bound},
which itself rests on the isoperimetric inequality of Lemma~\ref{lem:cousin-vempala}.
Recall that the proposal increments are $\mathcal{N}(0,\sigma d^{-1/2})$.
We define ``the centre'' of the space to be $\{x:\left|x\right|\leq b_{\kappa}\sigma d^{1/2}\}$
for some constant $b_{\kappa}>0$, and we always consider points that
are close to each other, in that $\left|x-y\right|\leq b_{\delta}\sigma d^{-1/2}$
for some (small) constant $b_{\delta}$. The proposals $Q_{x}$ and
$Q_{y}$ can be made close in total variation by Pinsker's inequality
for sufficiently small $b_{\delta}$.
\begin{enumerate}
\item When $x$ and $y$ are both in ``the centre'', we can then ensure
that $P(x,\cdot)$ and $P(y,\cdot)$ are close in total variation
by additionally ensuring that the acceptance probability is uniformly
lower bounded in the centre by a constant strictly above $1/2$ since
then
\[
\left\Vert P(x,\cdot)-P(y,\cdot)\right\Vert _{{\rm TV}}\leq\left\Vert P(x,\cdot)-Q_{x}\right\Vert _{{\rm TV}}+\left\Vert Q_{x}-Q_{y}\right\Vert _{{\rm TV}}+\left\Vert P(y,\cdot)-Q_{y}\right\Vert _{{\rm TV}}
\]
can be made less than $1$ by taking $b_{\delta}$ and $b_{\kappa}\sigma^{2}$
sufficiently small. This part of the proof that imposes a maximal
value of $\sigma$, which is slightly at odds with the common practice
of making the acceptance probability close to $1/4$ rather than larger
than $1/2$.
\item When at least one of $x$ and $y$ are not in ``the centre'', we
can use a different coupling argument that takes advantage of the
fact that the set of points $\{w:\left|w\right|\leq\left|x\right|\wedge\left|y\right|\}$
will be accepted as proposals from both $x$ and $y$, and is sufficiently
large if $b_{\kappa}$ is large enough. This overlap allows one to
obtain a non-trivial bound on $\left\Vert P(x,\cdot)-P(y,\cdot)\right\Vert _{{\rm TV}}$
with an acceptance rate that is less than $1/2$, which is important
because in the tails of the distribution one cannot obtain an acceptance
rate larger than $1/2$.
\end{enumerate}
\begin{thm}
Let Assumption~\ref{assu:target-distribution} hold. Let $\sigma=\varsigma/\sqrt{L}$
with $\varsigma\leq\varsigma_{\star}=0.073$. Then the conductance
(see equation (\ref{eq:cheeger-ineqs})) is lower bounded as follows:
\[
\kappa(0)\geq8.46\times10^{-5}\varsigma\sqrt{\frac{m}{Ld}},
\]
and hence
\[
{\rm Gap}(P)={\rm Gap}_{{\rm R}}(P)\geq8.94\times10^{-10}\cdot\varsigma^{2}\cdot\frac{m}{Ld}.
\]
\end{thm}

\begin{proof}
Let $\kappa:=(4+1/16)\sigma$ and $\delta:=\sigma/16$. Let $S:=\mathcal{B}\left(0,\kappa\cdot d^{1/2}\right)$.
Assume $x,y\in\mathsf{E}$ satisfy $\left|x-y\right|\leq\delta_{d}:=\delta d^{-1/2}$.
In either case $(x,y)\in S\times S$ or $(x,y)\not\in S\times S$,
then $\left\Vert P(x,\cdot)-P(y,\cdot)\right\Vert _{{\rm TV}}\leq\frac{31}{32}$
by Lemma~\ref{lem:inside-tv-bound} Lemma~\ref{lem:tv-outside-bound}
respectively. Hence, we may apply Theorem~\ref{thm:local-conductance-lower-bound}
with $C=\mathsf{E}$ to deduce 
\begin{align*}
\kappa(0) & =\inf_{A\in\mathscr{E}}\frac{\mu\otimes P(A\times A^{\complement})}{\mu\otimes\mu(A\times A^{\complement})}\\
 & \geq\frac{\varepsilon}{4}\min\left\{ 1,\frac{\log2}{4}\frac{\sigma}{16}\sqrt{\frac{m}{d}}\right\} \\
 & \geq\frac{1}{4\cdot32}\min\left\{ 1,\frac{\log2}{4}\frac{\varsigma}{16}\sqrt{\frac{m}{Ld}}\right\} \\
 & \geq\frac{1}{128}\min\left\{ 1,0.01083\varsigma\sqrt{\frac{m}{Ld}}\right\} \\
 & \geq8.46\times10^{-5}\varsigma\sqrt{\frac{m}{Ld}},
\end{align*}
noting that $m\leq L$ by Assumption~\ref{assu:target-distribution},
$\varsigma\leq\varsigma_{\star}<1$ and $d\geq1$. The bound on ${\rm Gap_{R}(P)}$
follows by (\ref{eq:cheeger-ineqs}), and we have ${\rm Gap}(P)={\rm Gap}_{{\rm R}}(P)$
by \cite[Lemma~3.1]{baxendale2005renewal}, since $Q$ is Gaussian. 
\end{proof}
\begin{rem}
If $\pi=\mathcal{N}(0,\sigma_{0}^{2}I_{d})$, then $m/L=1$ and $L=1/\sigma_{0}^{2}$.
Hence, we see that $\sigma$ should scale proportionally with $\sigma_{0}$
as one would expect by a reparametrization argument, and that the
bound is then independent of $\sigma_{0}$. The conductance/spectral
gap lower bound is maximized by taking $\sigma=0.073/\sqrt{L}$, and
for our argument one cannot take $\sigma$ larger than this. Theorem~\ref{thm:gaussian-rwm-conductance}
below shows that a more specific argument allows for a stronger statement
allowing arbitrary $\varsigma>0$ while retaining the same dimension
dependence.
\end{rem}

\begin{rem}
In several places in the proof we have adopted dimension-independent
bounds, e.g. by taking $d=1$, which are certainly sub-optimal for
large $d$. Similarly, for the sake of clarity we have made a few
choices of constants that are certainly not optimal. Hence, we can
expect that a more refined analysis would produce a larger lower bound
on the conductance and a larger maximum value of $\sigma\sqrt{L}$.
However, the proof strategy of ensuring a high acceptance rate in
the centre does seem to naturally force $\sigma$ to be artificially
small.
\end{rem}

\begin{rem}
We can inspect Assumption~\ref{assu:target-distribution} when $u\colon\mathbb{R}_{+}\rightarrow\mathbb{R}$
is continuously differentiable. It is useful to understand conditions
on the function $u$ which will guarantee that the desired estimates
hold. First, compute explicitly that
\end{rem}

\begin{align*}
\nabla U\left(x\right) & =2\cdot\dot{u}\left(\left|x\right|^{2}\right)\cdot x,\\
\nabla^{2}U\left(x\right) & =2\cdot\left[2\cdot\ddot{u}\left(\left|x\right|^{2}\right)\cdot x\cdot x^{\top}+\dot{u}\left(\left|x\right|^{2}\right)\cdot{\rm Id}\right].
\end{align*}
For sufficiently smooth potentials, strong convexity and smoothness
can be formulated in terms of the first two derivatives of $u$. In
particular, $m$-strong convexity requires that for all $x$, it holds
that
\begin{align*}
m & \leqslant\inf_{v}\left\{ \frac{v^{\top}\nabla^{2}U\left(x\right)v}{\left|v\right|^{2}}\right\} \\
 & =2\cdot\inf_{v}\left\{ \frac{2\ddot{u}\left(\left|x\right|^{2}\right)\left(v^{\top}x\right)^{2}+\dot{u}\left(\left|x\right|^{2}\right)\left|v\right|^{2}}{\left|v\right|^{2}}\right\} \\
 & =2\cdot\inf_{v}\left\{ \frac{2\ddot{u}\left(\left|x\right|^{2}\right)\left(v^{\top}x\right)^{2}}{\left|v\right|^{2}}+\dot{u}\left(\left|x\right|^{2}\right)\right\} \\
 & =2\cdot\left\{ 2\cdot\min\left(0,\ddot{u}\left(\left|x\right|^{2}\right)\cdot\left|x\right|^{2}\right)+\dot{u}\left(\left|x\right|^{2}\right)\right\} ,
\end{align*}
i.e. that $\inf_{s\geqslant0}\left\{ 2\cdot\min\left(0,\ddot{u}\left(s\right)\cdot s\right)+\dot{u}\left(s\right)\right\} \geqslant\frac{m}{2}$.
Similar calculations show that $L$-smoothness requires that $\sup_{s\geqslant0}\left\{ 2\cdot\max\left(0,\ddot{u}\left(s\right)\cdot s\right)+\dot{u}\left(s\right)\right\} \leqslant\frac{L}{2}$.
To be more concrete, suppose that $u$ satisfies $0<m_{1}\leqslant\dot{u}\left(s\right)\leqslant L_{1}$
and $\left|\ddot{u}\left(s\right)\right|\leqslant L_{2}s^{-1}$ with
$L_{2}\leqslant\frac{m_{1}}{2}$, i.e. it is increasing, essentially
sandwiched between two affine functions, and its derivative has slow
variation at infinity. It then follows that
\begin{align*}
2\cdot\min\left(0,\ddot{u}\left(s\right)\cdot s\right)+\dot{u}\left(s\right) & \geqslant m_{1}-2\cdot L_{2}\\
2\cdot\max\left(0,\ddot{u}\left(s\right)\cdot s\right)+\dot{u}\left(s\right) & \leqslant L_{1}+2\cdot L_{2},
\end{align*}
i.e. that we can take $m=m_{1}-2\cdot L_{2}>0$, $L=L_{1}+2\cdot L_{2}$.

The following two lemmas are known and useful bounds on the total
variation distance between two normal distributions, and tail probabilities
for $\chi^{2}$ random variables.
\begin{lem}
\label{lem:RWM-continuity-proposal}For any $\epsilon>0$ and $x,y\in\mathsf{X}$
such that $|x-y|\leqslant\epsilon\cdot d^{-1/2}$ it holds that
\[
\left\Vert Q_{x}-Q_{y}\right\Vert _{\mathrm{TV}}\leq\frac{\epsilon}{2\sigma}.
\]
\end{lem}

\begin{proof}
This is obtained via Pinsker's inequality. Compute that
\begin{align*}
\mathrm{KL}\left(Q_{x},Q_{y}\right) & =\mathbb{E}_{u\sim\mathcal{N}\left(x,d^{-1}\sigma^{2}{\rm Id}\right)}\left[\frac{|u-y|^{2}}{2\sigma^{2}/d}-\frac{|u-x|^{2}}{2\sigma^{2}/d}\right]\\
 & =\frac{d}{2\cdot\sigma^{2}}\cdot\mathbb{E}_{\xi\sim\mathcal{N}\left(0,\mathrm{{\rm Id}}\right)}\left[|x-y+\sigma d^{-1/2}\cdot\xi|^{2}-|\sigma d^{-1/2}\cdot\xi|^{2}\right]\\
 & =\frac{d}{2\cdot\sigma^{2}}\cdot|x-y|^{2}.
\end{align*}
Hence, if $|x-y|\leqslant\epsilon\cdot d^{-1/2}$ then it follows
that $\mathrm{KL}\left(Q_{x},Q_{y}\right)\leqslant\frac{\epsilon^{2}}{2\cdot\sigma^{2}}\cdot$
Recalling Pinsker's inequality, we deduce that
\begin{align*}
\left\Vert Q_{x}-Q_{y}\right\Vert _{\mathrm{TV}} & \leqslant\sqrt{\mathrm{KL}\left(Q_{x},Q_{y}\right)/2}\\
 & =\frac{\epsilon}{2\sigma}\,,
\end{align*}
as claimed.
\end{proof}
\begin{lem}[{\cite[Lemma 1]{laurent-massart-2000}}]
\label{lem:laurent-massart} If $W\sim\chi_{d}^{2}$ then for $u>0$
we have 
\[
\mathbb{P}\left(W\geqslant d+2\sqrt{du}+2u\right)\leq\exp\left(-u\right).
\]
In particular, for $\epsilon\in(0,1)$, with $\exp\left(-u\right)=\epsilon$,
$\chi\left(\epsilon,d\right):=1+2\sqrt{\frac{\log\epsilon^{-1}}{d}}+2\frac{\log\epsilon^{-1}}{d}$
and $\chi\left(\epsilon\right):=\chi\left(\epsilon,1\right)$, we
have 
\[
\mathbb{P}\left(W\geqslant d\cdot\chi\left(\epsilon\right)\right)\leqslant\mathbb{P}\left(W\geqslant d\cdot\chi\left(\epsilon,d\right)\right)\leqslant\epsilon\,.
\]
We also have, for $u>0$,
\[
\mathbb{P}\left(W\leq d-2\sqrt{du}\right)\leq\exp\left(-u\right).
\]
\end{lem}

\begin{lem}
\label{lem:RWM-coupling-Q-P}Assume that $U$ attains its minimum
at $0$, and is $L-$smooth. For any $\epsilon>0$, if $\kappa\geq\sigma$
and 
\[
\kappa\sigma\leq\frac{1}{L}\cdot\frac{-\log(1-\frac{\epsilon}{2})}{\chi(\frac{\epsilon}{4})}\cdot\frac{2}{3},
\]
then for all $x\in\mathcal{B}\left(0,\kappa\cdot d^{1/2}\right)$
\[
\left\Vert Q_{x}\left(\cdot\right)-P\left(x,\cdot\right)\right\Vert {}_{{\rm TV}}\leq\epsilon\,.
\]
\end{lem}

\begin{proof}
First, note that for $(x,A)\in\mathsf{E}\times\mathscr{E}$,
\[
\left|Q_{x}(A)-P(x,A)\right|=\left|\int\mathbf{1}\{x+d^{-1/2}z\in A\}[1-\alpha(x,d^{-1/2}z)]Q({\rm d}z)-[1-\alpha(x)]\mathbf{1}\{x\in A\}\right|\,,
\]
with $\alpha\left(x\right):=\int\alpha\left(x,d^{-1/2}z\right)\cdot Q\left(\mathrm{d}z\right)$
as in (\ref{eq:alpha_RWM_def}), which is maximized for $A=\mathsf{E}\setminus\{x\}$
or $A=\{x\}$ since we are considering the difference of non-negative
terms. Therefore
\[
\left\Vert Q_{x}\left(\cdot\right)-P\left(x,\cdot\right)\right\Vert _{{\rm TV}}=\int|1-\alpha\left(x,d^{-1/2}z\right)|\cdot Q\left({\rm d}z\right)=1-\alpha\left(x\right)\,.
\]
As suggested in \cite{Dwivedi-Chen-Wainwright-JMLR:v20:19-306} we
use Markov's inequality, that is for $a\in(0,1]$,
\begin{equation}
\alpha\left(x\right)\geqslant a\cdot Q\left(\mathsf{r}\left(x,d^{-1/2}z\right)\geqslant a\right),\label{eq:RWM_markov_ineq}
\end{equation}
which motivates seeking a lower bound for
\[
\mathsf{r}\left(x,d^{-1/2}z\right)=\frac{\pi\left(x+d^{-1/2}z\right)}{\pi\left(x\right)}=\exp\left(U\left(x\right)-U\left(x+d^{-1/2}z\right)\right).
\]
We begin by noting that for $\left(x,z\right)\in\mathsf{X}\times\mathsf{Z}$,
\begin{align*}
U\left(x+d^{-1/2}z\right)-U\left(x\right) & \leq\left\langle \nabla U\left(x\right),d^{-1/2}z\right\rangle +\frac{L}{2}\left|d^{-1/2}z\right|^{2}.
\end{align*}
If $Z\sim Q=\mathcal{N}(0,\sigma^{2}I_{d})$, then $\left\langle \nabla U\left(x\right),Z\right\rangle \sim\mathcal{N}\left(0,\sigma^{2}\cdot\left|\nabla U\left(x\right)\right|^{2}\right)$
and from the equivalent characterization of $L-$smoothness \cite[Lemma 9]{Dwivedi-Chen-Wainwright-JMLR:v20:19-306}
with $\nabla U\left(0\right)=0$ we have $\left|\nabla U\left(x\right)\right|\leqslant L\cdot\left|x\right|.$
Hence $\sup_{x\in\mathcal{B}\left(0,\kappa\cdot d^{1/2}\right)}\left|\nabla U\left(x\right)\right|\leqslant L\cdot\kappa\cdot d^{1/2}$,
and from Chernoff's inequality for a normal random variable $\bar{Z}\sim\mathcal{N}\left(0,1\right)$,
that is $\mathbb{P}\left(\bar{Z}\geqslant u\right)\leq\exp\left(-\frac{1}{2}u^{2}\right)$
for $u>0$, we can write that
\begin{align*}
Q\left(\left\langle \nabla U\left(x\right),d^{-1/2}z\right\rangle >u\right) & =\mathbb{P}\left(\bar{Z}\cdot\frac{\sigma\cdot\left|\nabla U\left(x\right)\right|}{d^{1/2}}\geqslant u\right)\\
 & \leqslant\mathbb{P}\left(\bar{Z}\cdot\frac{\sigma\cdot L\cdot\kappa\cdot d^{1/2}}{d^{1/2}}\geqslant u\right)\\
 & =\mathbb{P}\left(\bar{Z}\geqslant\frac{u}{\sigma\cdot L\cdot\kappa}\right)\leqslant\exp\left(-\frac{u^{2}}{2\cdot\sigma^{2}\cdot L^{2}\cdot\kappa^{2}}\right).
\end{align*}
In particular, taking $u=\sigma\cdot L\cdot\kappa\cdot\sqrt{2\cdot\log\left(\frac{4}{\epsilon}\right)}$,
we see that
\[
Q\left(\left\langle \nabla U\left(x\right),d^{-1/2}z\right\rangle >\sigma\cdot L\cdot\kappa\cdot\sqrt{2\cdot\log\left(\frac{4}{\epsilon}\right)}\right)\leqslant\frac{\epsilon}{4}.
\]
From Lemma~\ref{lem:laurent-massart}, we have that
\[
Q\left(\frac{L}{2}\cdot\left|d^{-1/2}\cdot z\right|^{2}>\sigma^{2}\cdot\frac{L}{2}\cdot\chi\left(\frac{\epsilon}{4}\right)\right)\leqslant\frac{\epsilon}{4}.
\]
Note that for random variables $X,Y$ we have $\mathbb{P}\left(X+Y>a+b\right)\leqslant\mathbb{P}\left(X>a\right)+\mathbb{P}\left(Y>b\right)$
for $a,b\in\mathbb{R}$, because
\begin{align*}
\mathbb{P}\left(X+Y>a+b\right)= & \mathbb{P}\left(X>a,X+Y>a+b\right)+\mathbb{P}\left(X<a,Y>a+b-X\right)\\
\leq & \mathbb{P}\left(X>a\right)+\mathbb{P}\left(Y>b\right).
\end{align*}
Consequently for $x\in\mathcal{B}\left(0,\kappa\cdot d^{1/2}\right)$,

\begin{align*}
Q\left(U\left(x+d^{-1/2}z\right)-U\left(x\right)>\sigma\cdot L\cdot\chi\left(\frac{\epsilon}{4}\right)\cdot\left(\kappa+\frac{\sigma}{2}\right)\right)\\
 & \hspace{-8cm}\leqslant Q\left(\left\langle \nabla U\left(x\right),d^{-1/2}z\right\rangle +\frac{L}{2}\left|d^{-1/2}z\right|^{2}\geqslant\sigma\cdot L\cdot\kappa\cdot\chi\left(\frac{\epsilon}{4}\right)+\frac{\sigma^{2}}{2}\cdot L\cdot\chi\left(\frac{\epsilon}{4}\right)\right)\\
 & \hspace{-8cm}\leqslant Q\left(\left\langle \nabla U\left(x\right),d^{-1/2}z\right\rangle \geqslant\sigma\cdot L\cdot\kappa\cdot\chi\left(\frac{\epsilon}{4}\right)\right)\\
 & \hspace{-2cm}+Q\left(\frac{L}{2}\left|d^{-1/2}z\right|^{2}\geqslant\frac{\sigma^{2}}{2}\cdot L\cdot\chi\left(\frac{\epsilon}{4}\right)\right)\\
 & \hspace{-8cm}\leqslant Q\left(\left\langle \nabla U\left(x\right),d^{-1/2}z\right\rangle \geqslant\sigma\cdot L\cdot\kappa\cdot\sqrt{2\cdot\log\left(\frac{4}{\epsilon}\right)}\right)\\
 & \hspace{-2cm}+Q\left(\frac{L}{2}\left|d^{-1/2}z\right|^{2}\geqslant\frac{\sigma^{2}}{2}\cdot L\cdot\chi\left(\frac{\epsilon}{4}\right)\right)\\
 & \hspace{-8cm}\leqslant\frac{\epsilon}{4}+\frac{\epsilon}{4}=\frac{\epsilon}{2},
\end{align*}
that is, 
\begin{align*}
Q\left(r\left(x,d^{-1/2}z\right)\geqslant\exp\left(-\sigma L\cdot\chi\left(\frac{\epsilon}{4}\right)\cdot\left(\kappa+\frac{\sigma}{2}\right)\right)\right) & \geqslant1-\frac{\epsilon}{2}.
\end{align*}
It follows that by taking $a=\exp\left(-\sigma\cdot L\cdot\chi\left(\frac{\epsilon}{4}\right)\cdot\left(\kappa+\frac{\sigma}{2}\right)\right)$
in Markov's inequality (\ref{eq:RWM_markov_ineq}) and assuming $\kappa\geq\sigma$
we can bound 
\begin{align*}
\alpha\left(x\right) & \geqslant\exp\left(-\sigma\cdot L\cdot\chi\left(\frac{\epsilon}{4}\right)\cdot\left(\kappa+\frac{\sigma}{2}\right)\right)\cdot\left(1-\frac{\epsilon}{2}\right)\\
 & \geqslant\exp\left(-L\cdot\chi\left(\frac{\epsilon}{4}\right)\cdot\frac{3}{2}\kappa\sigma\right)\cdot\left(1-\frac{\epsilon}{2}\right).
\end{align*}
Now if 
\[
\kappa\sigma\leq\frac{1}{L}\cdot\frac{-\log(1-\frac{\epsilon}{2})}{\chi(\epsilon/4)}\cdot\frac{2}{3},
\]
then $\exp\left(-L\cdot\chi\left(\frac{\epsilon}{4}\right)\cdot\frac{3}{2}\kappa\sigma\right)\geqslant1-\frac{\epsilon}{2}$,
so that $\alpha\left(x\right)\geqslant\left(1-\frac{\epsilon}{2}\right)^{2}\geqslant1-\epsilon$,
and hence that $1-\alpha\left(x\right)\leqslant\epsilon$, as claimed.
\end{proof}
\begin{lem}
\label{lem:inside-tv-bound}Assume that $U$ attains its minimum at
$0$, and is $L-$smooth. Let $\sigma\leq\varsigma/\sqrt{L}$ with
$\varsigma\leq\varsigma_{\star}:=0.073$, $\kappa:=(4+1/16)\sigma$,
and $\delta:=\sigma/16$. Let $S=\mathcal{B}\left(0,\kappa\cdot d^{1/2}\right)$.
Then for $(x,y)\in S\times S$ such that $\left|x-y\right|\leq\delta d^{-1/2}$
we have
\[
\left\Vert P(x,\cdot)-P(y,\cdot)\right\Vert _{{\rm TV}}\leq\frac{31}{32}.
\]
\end{lem}

\begin{proof}
We have
\[
\left\Vert P(x,\cdot)-P(y,\cdot)\right\Vert _{{\rm TV}}\leq\left\Vert P(x,\cdot)-Q_{x}\right\Vert _{{\rm TV}}+\left\Vert Q_{x}-Q_{y}\right\Vert _{{\rm TV}}+\left\Vert P(y,\cdot)-Q_{y}\right\Vert _{{\rm TV}}.
\]
We have $\left\Vert Q_{x}-Q_{y}\right\Vert _{{\rm TV}}\leq1/32$ by
Lemma~\ref{lem:RWM-continuity-proposal}. We wish to show that for
$x\in S$, 
\[
\left\Vert P(x,\cdot)-Q_{x}\right\Vert _{{\rm TV}}\leq15/32.
\]
This is ensured by Lemma~\ref{lem:RWM-coupling-Q-P}: taking $\epsilon=15/32$
we need to verify that for $b_{\kappa}:=4+1/16$,
\[
\kappa\sigma=b_{\kappa}\sigma^{2}\leq\frac{1}{L}\cdot\frac{-\log(1-\frac{\epsilon}{2})}{\chi(\frac{\epsilon}{4})}\cdot\frac{2}{3},
\]
and so it is sufficient to take 
\[
\sigma^{2}\leq\frac{0.073^{2}}{L}\leq\frac{1}{L}\cdot\frac{-\log(\frac{49}{64})}{\chi(15/128)}\cdot\frac{2}{3}\cdot\frac{1}{4+1/16}.
\]
\end{proof}
\begin{lem}
\label{lem:min-dist-x-given-y-delta}If $\left|y-x\right|\leq\delta$
and $\left|y\right|\geq\delta$ then $\left|x\right|^{2}\geq(\left|y\right|-\delta)^{2}$.
\end{lem}

\begin{proof}
Let $x=y+r$ where $\left|r\right|\leq\delta\leq\left|y\right|$.
Then by Cauchy--Schwarz,
\begin{align*}
\left|x\right|^{2} & =\left|y\right|^{2}+2\left\langle y,r\right\rangle +\left|r\right|^{2}\\
 & \geq\left|y\right|^{2}-2\left|y\right|\left|r\right|+\left|r\right|^{2}\\
 & =\left(\left|y\right|-\left|r\right|\right)^{2},
\end{align*}
from which we can conclude.
\end{proof}

\begin{lem}
\label{lem:outside-prob-smaller-norm}For any $\beta\in(0,1)$, let
$z_{\beta}$ denote the $\beta$-quantile of the $\mathcal{N}(0,1)$
distribution, namely $\mathbb{P}(Z_{1}\ge z_{\beta})=1-\beta$ for
$Z_{1}\sim N(0,1)$. For any $\alpha<1/2$, let $x\in\mathbb{R}^{d}$
satisfy $\left|x\right|\geq c\sigma d^{1/2}$ for some $c>(2\alpha z_{1-\alpha})^{-1}$.
Then if $W\sim Q_{x}$, 
\[
\mathbb{P}(\left|W\right|\leq\left|x\right|)\geq\alpha-\frac{1}{2cz_{1-\alpha}}>0.
\]
In particular, if $\alpha=1/4$ and $c>3$, $\mathbb{P}\left(Z\in A\right)>\frac{1}{4}-\frac{3}{4c}=\frac{1}{4}(1-\frac{3}{c})>0$.
\end{lem}

\begin{proof}
Without loss of generality, we may assume that $x=(-\left|x\right|,0,\ldots,0)$.
Let $W=x+\sigma d^{-1/2}Z$ where $Z\sim\mathcal{N}(0,I_{d})$, and
$w=x+\sigma d^{-1/2}z$. Then 
\begin{align*}
A & =\{z:\left|w\right|\leq\left|x\right|\}\\
 & =\{z:\left|x+\sigma d^{-1/2}z\right|\leq\left|x\right|\}\\
 & =\left\{ z:2\frac{\sigma}{d^{1/2}}\sum_{i=1}^{d}x_{i}z_{i}+\frac{\sigma^{2}}{d}\sum_{i=1}^{d}z_{i}^{2}\leq0\right\} \\
 & =\left\{ z:\frac{\sigma}{d^{1/2}}\sum_{i=1}^{d}z_{i}^{2}\leq-2\sum_{i=1}^{d}x_{i}z_{i}\right\} \\
 & =\left\{ z:\sum_{i=1}^{d}z_{i}^{2}\leq\frac{2}{\sigma}\left|x\right|z_{1}d^{1/2}\right\} \\
 & \supseteq\left\{ z:\sum_{i=1}^{d}z_{i}^{2}\leq2cv_{1}d\right\} \cap\left\{ z:z_{1}\geq v_{1}\right\} ,
\end{align*}
for any $v_{1}>0$. Now take $v_{1}=z_{1-\alpha}$. Then,
\begin{align*}
\mathbb{P}\left(Z\in A\right) & \geq\mathbb{P}\left(\left|Z\right|^{2}\leq2cz_{1-\alpha}d,Z_{1}\geq z_{1-\alpha}\right)\\
 & \geq\mathbb{P}\left(\left|Z\right|^{2}\leq2cz_{1-\alpha}d\right)+\mathbb{P}(Z_{1}\geq z_{1-\alpha})-1\\
 & =\mathbb{P}\left(\left|Z\right|^{2}\leq2cz_{1-\alpha}d\right)+\alpha-1.
\end{align*}
By Markov's inequality, we have
\[
\mathbb{P}\left(\left|Z\right|^{2}>2cz_{1-\alpha}d\right)\leq\frac{1}{2cz_{1-\alpha}},
\]
and so
\[
\mathbb{P}\left(Z\in A\right)\geq\alpha-\frac{1}{2cz_{1-\alpha}}.
\]
For the last part, observe that if $\alpha=1/4$ then $z_{1-\alpha}>2/3$,
and the conclusion follows.
\end{proof}
\begin{lem}
\label{lem:tv-outside-bound}Assume $U(x)=u(\left|x\right|^{2})$
with $u:[0,\infty)\to[0,\infty)$ increasing. For $\sigma>0$, let
$\kappa=b_{\kappa}\sigma$, $\delta=b_{\delta}\sigma$ for some constants
$b_{\kappa}>b_{\delta}$. Let $(x,y)\in\left(\mathcal{B}\left(0,\kappa\cdot d^{1/2}\right)\times\mathcal{B}\left(0,\kappa\cdot d^{1/2}\right)\right)^{\complement}$.
Then if $\left|x-y\right|\leq\delta d^{-1/2}$, we have that 
\begin{equation}
\left\Vert P(x,\cdot)-P(y,\cdot)\right\Vert _{{\rm TV}}\leq\frac{3}{4}+\frac{3}{4(b_{\kappa}-b_{\delta})}+\frac{b_{\delta}}{2}.\label{eq:lem_tv_out}
\end{equation}
In particular, if we take $b_{\kappa}=4+1/16$ and $b_{\delta}=1/16$,
then we obtain
\[
\left\Vert P(x,\cdot)-P(y,\cdot)\right\Vert _{{\rm TV}}\leq\frac{31}{32}.
\]
\end{lem}

\begin{proof}
For $x,y\in\left(\mathcal{B}\left(0,\kappa\cdot d^{1/2}\right)\times\mathcal{B}\left(0,\kappa\cdot d^{1/2}\right)\right)^{\complement}$,
we construct a coupling $(X',Y')$ such that $X'\sim P(x,\cdot)$
and $Y'\sim P(y,\cdot)$, and will show that $\mathbb{P}(X'=Y')\geq1-\epsilon$,
with $\epsilon$ as in the right-hand side of (\ref{eq:lem_tv_out}).
Without loss of generality, assume $\left|x\right|\leq\left|y\right|$.
Hence, we have by Lemma~\ref{lem:min-dist-x-given-y-delta} the (crude)
bound $\left|y\right|\geq\left|x\right|\geq\kappa d^{1/2}-\delta d^{-1/2}\geq(\kappa-\delta)d^{1/2}$.
Let $(W_{x},W_{y})$ be distributed according to a maximal coupling
of $Q_{x}$ and $Q_{y}$. By Lemma~\ref{lem:RWM-continuity-proposal},
\[
\mathbb{P}(W_{x}=W_{y})=1-\left\Vert Q_{x}-Q_{y}\right\Vert _{{\rm TV}}\geq1-\delta/2\sigma.
\]
On the event $W_{x}=W_{y}$, we have $X'=Y'=W_{x}$ if $\left|W_{x}\right|\leq\left|x\right|$,
since $U(x)=u(\left|x\right|)$ so the proposals will be accepted
with probability one. Note that $W_{x}=x+\sigma d^{-1/2}Z$, where
$Z\sim\mathcal{N}(0,I_{d})$. Hence, by Lemma~\ref{lem:outside-prob-smaller-norm},
\[
\mathbb{P}(\left|W_{x}\right|\leq\left|x\right|)\geq\alpha-\frac{1}{2cz_{1-\alpha}},
\]
for any $\alpha<1/2$ and $c=(\kappa-\delta)/\sigma$. Hence we have
the bound
\begin{align*}
\mathbb{P}(X'=Y') & \geq\mathbb{P}(W_{x}=W_{y},\left|W_{x}\right|\leq\left|X\right|)\\
 & \geq\mathbb{P}(W_{x}=W_{y})+\mathbb{P}(\left|W_{x}\right|\leq\left|X\right|)-1\\
 & \geq\mathbb{P}(\left|W_{x}\right|\leq\left|X\right|)-\frac{\delta}{2\sigma}\\
 & \geq\alpha-\frac{\sigma}{2(\kappa-\delta)z_{1-\alpha}}-\frac{\delta}{2\sigma}.
\end{align*}
Now, taking $\alpha=1/4$, we obtain
\[
\mathbb{P}(X'=Y')\geq\frac{1}{4}-\frac{3}{4(b_{\kappa}-b_{\delta})}-\frac{b_{\delta}}{2},
\]
and we conclude by the coupling inequality $\left\Vert P(x,\cdot)-P(y,\cdot)\right\Vert _{{\rm TV}}\leq\mathbb{P}(X'\neq Y')$.
\end{proof}

\subsection{\label{subsec:Spectral-gap-for-gaussian}Spectral gap for the RWM
on a Gaussian target}

When $\pi$ is $\mathcal{N}(0,\sigma_{0}^{2}I_{d})$, it is possible
to obtain more precise bounds on the conductance and spectral gap,
and also for the proposal standard deviation to be an arbitrary multiple
of $\sigma_{0}$, when scaled appropriately by $d^{-1/2}$. 
\begin{lem}
\label{lem:normal-accept-x}Assume $U(x)=\frac{1}{2\sigma_{0}^{2}}\left|x\right|^{2}$.
Let $X'\sim P(x,\cdot)$ with proposal $W=x+\sigma_{d}Z$, where $\sigma_{d}=\varsigma d^{-1/2}\sigma_{0}$
for some $\varsigma>0$ and $Z\sim\mathcal{N}(0,I_{d})$. Then
\[
\mathbb{P}(X'=W)\geq\exp\left\{ -\frac{\varsigma^{2}}{2}\left[1+2d^{-1/2}+2d^{-1}\right]\right\} \cdot\frac{1}{2}\cdot(1-{\rm e}^{-1}).
\]
\end{lem}

\begin{proof}
Since the proposal and target are spherically symmetric, we may assume
without loss of generality that $x=(x_{1},0,\ldots,0)$. Then
\begin{align*}
\left|x+\sigma d^{-1/2}z\right|^{2} & =\left|x\right|^{2}+2\sigma_{d}\left\langle x,z\right\rangle +\frac{\sigma^{2}}{d}\left|z\right|^{2}\\
 & =\left|x\right|^{2}+2\sigma_{d}x_{1}z_{1}+\sigma_{d}^{2}\left|z\right|^{2}.
\end{align*}
Hence,
\begin{equation}
U(W)-U(x)=\frac{1}{2\sigma_{0}^{2}}\left\{ 2\sigma_{d}x_{1}Z_{1}+\sigma_{d}^{2}\left|Z\right|^{2}\right\} .\label{eq:U(W)-U(x)}
\end{equation}
Now, for $r_{d}>0$,
\[
\mathbb{P}\left(x_{1}Z_{1}\leq0,\frac{\sigma_{d}^{2}}{2\sigma_{0}^{2}}\left|Z\right|^{2}\leq r_{d}\right)=\frac{1}{2}\mathbb{P}\left(\left|Z\right|^{2}\leq dr_{d}\cdot\frac{2}{\varsigma^{2}}\right),
\]
since $\mathbb{I}(Z_{1}>0)$ is independent of $\left|Z\right|^{2}$.
By Lemma~\ref{lem:laurent-massart}, we have
\[
\mathbb{P}\left(\left|Z\right|^{2}\leq d\left\{ 1+2\sqrt{\frac{u}{d}}+2\frac{u}{d}\right\} \right)\geq1-\exp(-u).
\]
So, taking $u=1$, we set
\[
r_{d}=\frac{\varsigma^{2}}{2}(1+2d^{-1/2}+2d^{-1}),
\]
which gives 
\[
\mathbb{P}\left(x_{1}Z_{1}\leq0,\frac{\sigma_{d}^{2}}{2\sigma_{0}^{2}}\left|Z\right|^{2}\leq r_{d}\right)\geq\frac{1}{2}\cdot\left\{ 1-\exp(-1)\right\} .
\]
It thus follows from (\ref{eq:U(W)-U(x)}) that 
\[
\mathbb{P}(U(W)-U(x)\leq r_{d})\geq\frac{1}{2}\cdot(1-{\rm e}^{-1}),
\]
and so
\[
\mathbb{P}\left(U(W)-U(x)\leq\frac{\varsigma^{2}}{2}\left\{ 1+2\sqrt{\frac{1}{d}}+2\frac{1}{d}\right\} \right)\geq\frac{1}{2}\cdot(1-{\rm e}^{-1}),
\]
from which we may conclude, since on this event the proposal is accepted
with probability at least $\exp\left\{ -\frac{\varsigma^{2}}{2}\left[1+2d^{-1/2}+2d^{-1}\right]\right\} $.
\end{proof}
\begin{thm}
\label{thm:gaussian-rwm-conductance}Assume $U(x)=\frac{1}{2\sigma_{0}^{2}}\left|x\right|^{2}$,
and let $\sigma=\varsigma\sigma_{0}$ for any $\varsigma>0$. Then
the conductance
\[
\kappa(0)\geq0.00216\exp\left\{ -\varsigma^{2}\left[1+2d^{-1/2}+2d^{-1}\right]\right\} \cdot\varsigma d^{-1/2},
\]
and hence
\[
{\rm Gap}(P)={\rm Gap}_{{\rm R}}(P)\geq5.83\times10^{-7}\cdot\exp\left\{ -2\varsigma^{2}\left[1+2d^{-1/2}+2d^{-1}\right]\right\} \cdot\varsigma^{2}d^{-1}.
\]
\end{thm}

\begin{proof}
Let $v=\exp\left\{ -\frac{\varsigma^{2}}{2}\left[1+2d^{-1/2}+2d^{-1}\right]\right\} \cdot\frac{1}{2}\cdot(1-{\rm e}^{-1})$.
Let $\delta_{d}=v\sigma d^{-1/2}$, and $x,y\in\mathsf{X}$ such that
$\left|x-y\right|\leq\delta_{d}$. Then $\left\Vert Q_{x}-Q_{y}\right\Vert _{{\rm TV}}\leq v/2$
by Lemma~\ref{lem:RWM-continuity-proposal}. We construct a specific
coupling of $(X',Y')$ such that $X'\sim P(x,\cdot)$ and $Y'\sim P(y,\cdot)$.
Without loss of generality, we may assume that $\left|x\right|\leq\left|y\right|$.
First, let $(W_{x},W_{y})$ be distributed according to a maximal
coupling of $Q_{x}$ and $Q_{y}$. Then, with $\mathcal{U}\sim{\rm Uniform}(0,1)$
we define
\[
X'\mid\{W_{x}=w_{x},W_{y}=w_{y},\mathcal{U}=u\}=\begin{cases}
w_{x} & u\leq\pi(w_{x})/\pi(x),\\
x & u>\pi(w_{x})/\pi(x).
\end{cases}
\]
Similarly we define
\[
Y'\mid\{W_{x}=w_{x},W_{y}=w_{y},\mathcal{U}=u\}=\begin{cases}
w_{y} & u\leq\pi(w_{y})/\pi(y),\\
y & u>\pi(w_{y})/\pi(y).
\end{cases}
\]
By Lemma~\ref{lem:RWM-continuity-proposal}, 
\[
\mathbb{P}(W_{x}=W_{y})=1-\left\Vert Q_{x}-Q_{y}\right\Vert _{{\rm TV}}\geq1-\frac{v}{2}.
\]
On the event $\{W_{x}=W_{y}\}\cap\{X'=W_{x}\}$, we have $X'=Y'=W_{x}$
since $\pi(y)\leq\pi(x)$. Hence, using Lemma~\ref{lem:normal-accept-x},
we have

\begin{align*}
\mathbb{P}(X'=Y') & \geq\mathbb{P}(W_{x}=W_{y},X'=W_{x})\\
 & \geq\mathbb{P}(W_{x}=W_{y})+\mathbb{P}(X'=W_{x})-1\\
 & =1-\left\Vert Q_{x}-Q_{y}\right\Vert _{{\rm TV}}-1+\mathbb{P}(X'=W_{x})\\
 & \ge-\frac{v}{2}+v\\
 & =\frac{v}{2}.
\end{align*}
Hence, $\left\Vert P(x,\cdot)-P(y,\cdot)\right\Vert _{{\rm TV}}\leq\mathbb{P}(X'\neq Y')\leq1-\frac{v}{2}$
by the coupling inequality. We now take $\varepsilon=\frac{v}{2}$
will apply Theorem~\ref{thm:local-conductance-lower-bound} with
$C=\mathsf{E}$. Recall that $m=1/\sigma_{0}^{2}$ from Example~\ref{exa:gaussian-mL},
and since $\frac{\log2}{4}v\varsigma d^{-1/2}\le1$ for any $d\in\mathbb{N}$
and $\varsigma>0$, we deduce that 
\begin{align*}
\kappa(0) & =\inf_{A\in\mathscr{E}}\frac{\mu\otimes P(A\times A^{\complement})}{\mu\otimes\mu(A\times A^{\complement})}\\
 & \geq\frac{\varepsilon}{4}\min\left\{ 1,\frac{\log2}{4}v\sigma d^{-1/2}\sqrt{m}\right\} \\
 & =\frac{v}{8}\min\left\{ 1,\frac{\log2}{4}v\varsigma d^{-1/2}\right\} \\
 & =\frac{v^{2}}{32}\log2\cdot\varsigma d^{-1/2}\\
 & =\exp\left\{ -\varsigma^{2}\left[1+2d^{-1/2}+2d^{-1}\right]\right\} \cdot\frac{1}{4}\cdot(1-{\rm e}^{-1})^{2}\cdot\frac{\log2}{32}\cdot\varsigma d^{-1/2}\\
 & \geq0.00216\cdot\exp\left\{ -\varsigma^{2}\left[1+2d^{-1/2}+2d^{-1}\right]\right\} \cdot\varsigma d^{-1/2}.
\end{align*}
The bound on ${\rm Gap_{R}(P)}$ follows by (\ref{eq:cheeger-ineqs}),
and we have ${\rm Gap}(P)={\rm Gap}_{{\rm R}}(P)$ by \cite[Lemma~3.1]{baxendale2005renewal},
since $Q$ is Gaussian. 
\end{proof}
\begin{rem}
The conductance lower bound is in $\Omega(d^{-1/2})$ and the spectral
gap lower bound is in $\Omega(d^{-1})$. Fixing $\varsigma$, we obtain
\[
\lim_{d\to\infty}\inf\kappa_{d}(0)d^{1/2}\geq0.00216\cdot\exp\left\{ -\varsigma^{2}\right\} \cdot\varsigma.
\]
The maximizing $\varsigma$ for the bound is obtained by $\varsigma^{2}=1/2$,
and this value of $\varsigma^{2}$ gives
\[
\lim_{d\to\infty}\inf\kappa_{d}(0)d^{1/2}\geq0.000926.
\]
This particular bound-maximizing value of $\varsigma$ is likely an
artifact of the proof technique; optimal scaling results suggest that
$\varsigma\approx2.38$ is optimal in high dimensions \cite{roberts2001optimal},
although they do not provide a bound on the conductance or spectral
gap of the associated Markov operator.
\end{rem}

To complement this result, we can show that the conductance must decrease
at least as $\mathcal{O}(d^{-1/2})$ when the proposal standard deviation
scales as $d^{-1/2}$, and that this is the slowest polynomial decay
possible. Hence, we may infer that in terms of optimizing conductance
and spectral gap, $d^{-1/2}$ is the correct polynomial scaling of
the standard deviation.
\begin{prop}
Consider the RWM with Gaussian proposal of standard deviation $\sigma_{d}=\varsigma\sigma_{0}d^{-\beta}$
for some $\beta\in\mathbb{R}$. Then the conductance is bounded as
\[
\kappa(0)\leq2\min\left\{ 2\varsigma d^{-\beta},\exp\left(-\frac{d}{16}\right)+\exp\left(-\varsigma^{2}\frac{d^{1-2\beta}}{8}\right)\right\} ,
\]
and the upper bound is maximized for large $d$ by taking $\beta=1/2$,
giving $\kappa_{d}(0)\leq4\varsigma d^{-1/2}$.
\end{prop}

\begin{proof}
First, let $A=\{x\in\mathsf{X}:x_{1}\geq0\}$, and we observe that
$\pi(A)=\frac{1}{2}$. We let $Z\sim\mathcal{N}(0,I_{d})$, and by
neglecting the acceptance probability and using the Chernoff bound
$\mathbb{P}(Z_{1}\leq-z)\leq\exp(-z^{2}/2)$ for $z>0$, we obtain
the bounds
\begin{align*}
\pi\otimes P(A\times A^{\complement}) & =\int_{A}\pi({\rm d}x)P(x,A^{\complement})\\
 & \leq\int_{A}\pi({\rm d}x)\mathbb{P}(x+\sigma_{d}Z\in A^{\complement})\\
 & =\int_{A}\pi({\rm d}x)\mathbb{P}(x_{1}+\sigma_{d}Z_{1}<0)\\
 & =\frac{1}{\sqrt{2\pi\sigma_{0}^{2}}}\int_{0}^{\infty}\exp\left\{ -\frac{x_{1}^{2}}{2\sigma_{0}^{2}}\right\} \mathbb{P}\left(Z_{1}<-\frac{x_{1}}{\sigma_{d}}\right){\rm d}x_{1}\\
 & \leq\frac{1}{\sqrt{2\pi\sigma_{0}^{2}}}\int_{0}^{\infty}\exp\left\{ -\frac{x_{1}^{2}}{2\sigma_{0}^{2}}-\frac{x_{1}^{2}}{2\sigma_{d}^{2}}\right\} {\rm d}x\\
 & =\left(\frac{\sigma_{d}^{2}}{\sigma_{d}^{2}+\sigma_{0}^{2}}\right)^{1/2}\\
 & \leq\frac{\sigma_{d}}{\sigma_{0}}=\varsigma d^{-\beta},
\end{align*}
and it follows that $\kappa_{d}(0)\leq\varsigma d^{-\beta}/\pi\otimes\pi(A\times A^{\complement})=4\varsigma d^{-\beta}$,
giving the first upper bound. 

Now let $B:=\{x:\left|x\right|\leq\delta_{d}\}$, where $\delta_{d}:=\frac{\sigma_{d}\sqrt{d}}{4\sqrt{2}}\wedge c_{d}=\frac{\sigma d^{-\beta+1/2}}{4\sqrt{2}}\wedge c_{d}$,
where $c_{d}$ is chosen so that $\pi(\left|x\right|\leq c_{d})=1/2$.
Hence, $\pi(B)\leq\frac{1}{2}$ and $\pi(B^{\complement})\geq\frac{1}{2}$.
We observe that for $x\in B$,
\begin{align}
\left|x+\sigma_{d}z\right|^{2}-\left|x\right|^{2} & =\left|x\right|^{2}+2\sigma_{d}\left\langle x,z\right\rangle +\sigma_{d}^{2}\left|z\right|^{2}-\left|x\right|^{2}\nonumber \\
 & \geq-2\sigma_{d}\left|x\right|\left|z\right|+\sigma_{d}^{2}\left|z\right|^{2}\nonumber \\
 & \geq-2\sigma_{d}\delta_{d}\left|z\right|+\sigma_{d}^{2}\left|z\right|^{2}\nonumber \\
 & =\sigma_{d}\left|z\right|\left(\sigma_{d}\left|z\right|-2\delta_{d}\right).\label{eq:xz_RWM_pf}
\end{align}
By Lemma~\ref{lem:laurent-massart},
\[
\mathbb{P}(\left|Z\right|^{2}\leq d-2\sqrt{du})\leq\exp(-u),
\]
and taking $u=d/16$ and $C_{d}=d\sigma_{d}^{2}/2$ we obtain
\[
\mathbb{P}(\sigma_{d}^{2}\left|Z\right|^{2}\leq C_{d})=\mathbb{P}(\left|Z\right|^{2}\leq d/2)\leq\exp(-d/16).
\]
Since $2\delta_{d}\sqrt{C_{d}}\leq C_{d}/2$, on the event $\sigma_{d}^{2}\left|Z\right|^{2}\geq C_{d}$
we have from (\ref{eq:xz_RWM_pf}) that
\[
\left|x+\sigma_{d}Z\right|^{2}-\left|x\right|^{2}\geq\sqrt{C_{d}}\left(\sqrt{C_{d}}-2\delta_{d}\right)=C_{d}-2\delta_{d}\sqrt{C_{d}}\geq C_{d}/2.
\]
It follows that for $x\in B$, the acceptance probability satisfies
\begin{align*}
\mathbb{E}\left[1\wedge r(x,\sigma_{d}Z)\right] & =\mathbb{E}\left[1\wedge\exp\left\{ -\frac{1}{2\sigma_{0}^{2}}\left(\left|x+\sigma_{d}Z\right|^{2}-\left|x\right|^{2}\right)\right\} \right]\\
 & \leq1\cdot\mathbb{P}\left(\sigma_{d}^{2}\left|Z\right|^{2}\leq C_{d}\right)+\exp\left\{ -\frac{1}{2\sigma_{0}^{2}}\left(C_{d}-2\delta_{d}\sqrt{C_{d}}\right)\right\} \cdot\mathbb{P}\left(\sigma_{d}^{2}\left|Z\right|^{2}>C_{d}\right)\\
 & \leq\exp\left(-\frac{d}{16}\right)+\exp\left\{ -\frac{1}{2\sigma_{0}^{2}}\left(C_{d}-2\delta_{d}\sqrt{C_{d}}\right)\right\} \\
 & \leq\exp\left(-\frac{d}{16}\right)+\exp\left(-\frac{C_{d}}{4\sigma_{0}^{2}}\right)\\
 & =\exp\left(-\frac{d}{16}\right)+\exp\left(-\frac{d\sigma_{d}^{2}}{8\sigma_{0}^{2}}\right).
\end{align*}
Therefore,
\begin{align*}
\frac{\pi\otimes P(B\times B^{\complement})}{\pi\otimes\pi(B\times B^{\complement})} & =\frac{\int\pi_{B}({\rm d}x)P(x,B^{\complement})}{\pi(B^{\complement})}\\
 & \leq2\int\pi_{B}({\rm d}x)P(x,\{x\}^{\complement})\\
 & \leq2\left\{ \exp\left(-\frac{d}{16}\right)+\exp\left(-\varsigma^{2}\frac{d^{1-2\beta}}{8}\right)\right\} ,
\end{align*}
and we conclude.
\end{proof}
A natural question is whether the lower bound for the spectral gap
is of the correct order when the proposal standard deviation scales
as $d^{-1/2}$, i.e. whether indeed ${\rm Gap}(P)$ scales as $d^{-1}$.
In this case, we can verify directly that this is the case.
\begin{prop}
Let $\pi$ be such that $\mathbb{E}_{\pi}[X_{1}]=0$ and $\mathbb{E}_{\pi}[X_{1}^{2}]=\sigma_{0}^{2}$,
and the proposal satisfy $Q_{x}(A)=\int_{A}\mathcal{N}(y;x,\sigma_{d}^{2}I_{d}){\rm d}y$
for $A\in\mathscr{X}$. Then
\[
{\rm Gap}(P)\leq\frac{\sigma_{d}^{2}}{2\sigma_{0}^{2}}.
\]
\end{prop}

\begin{proof}
We use the fact that ${\rm Gap}_{{\rm R}}(P)=\inf_{f\in\ELL_{0}(\pi)}\mathcal{E}(P,f)/\left\Vert f\right\Vert _{2}^{2}$.
Let $f(x)=x_{1}$. Then we compute
\begin{align*}
\mathcal{E}(P,f) & =\frac{1}{2}\int\pi({\rm d}x)P(x,{\rm d}y)(y_{1}-x_{1})^{2}\\
 & \leq\frac{1}{2}\int\pi({\rm d}x)Q_{x}({\rm d}y)(y_{1}-x_{1})^{2}\\
 & =\frac{1}{2}\sigma_{d}^{2},
\end{align*}
while $\left\Vert f\right\Vert _{2}^{2}=\sigma_{0}^{2}$, and we conclude
from ${\rm Gap}(P)\leq\mathcal{E}(P,f)/\left\Vert f\right\Vert _{2}^{2}$.
\end{proof}

\subsection{Central limit theorems \label{subsec:example-Central-limit-theorems}}

Obtaining a central limit theorem follows in a relatively straightforward
manner when $\left\Vert P^{n}f\right\Vert _{2}^{2}$ decays quickly
enough.
\begin{prop}
\label{prop:clt}Let $f\in\mathrm{L}_{0}^{2}(\mu)$ with $\Phi(f)<\infty$.
Let $(X_{n})$ be a Markov chain with Markov kernel $P$. Assume $\left\Vert P^{n}f\right\Vert _{2}^{2}\leq\Phi(f)\gamma(n)$
with $\gamma(n)\in\mathcal{O}(n^{-a})$ for some $a>1$. Then for
$\mu$-almost all $X_{0}$,
\[
\frac{1}{\sqrt{n}}\sum_{i=0}^{n-1}f(X_{i})\overset{L}{\to}\mathcal{N}(0,\sigma^{2}),
\]
where $\sigma^{2}=\lim_{n\to\infty}\frac{1}{n}\mathbb{E}_{\mu}\left[\left\{ \sum_{i=0}^{n-1}f(X_{i})\right\} ^{2}\right]<\infty$.
\end{prop}

\begin{proof}
We will verify the Maxwell--Woodroofe condition:
\begin{equation}
\sum_{n=1}^{\infty}n^{-3/2}\left\Vert V_{n}f\right\Vert _{2}<\infty,\label{eq:mw-cond}
\end{equation}
where $V_{n}f=\sum_{k=0}^{n-1}P^{k}f$. The central limit theorem
then follows from \cite[Corollary~1]{maxwell2000central}. Minkowski's
inequality gives
\[
\left\Vert V_{n}f\right\Vert _{2}=\left\Vert \sum_{k=0}^{n-1}P^{k}f\right\Vert _{2}\leq\sum_{k=0}^{n-1}\left\Vert P^{k}f\right\Vert _{2}\leq\Phi^{1/2}(f)\sum_{k=0}^{n-1}\gamma^{1/2}(k).
\]
For $a>1$ then we may write $\gamma^{1/2}(k)\leq C(k+1)^{-a/2}$
for some $C>0$ and note that $\gamma(0)<\infty$ . Then
\begin{align*}
\frac{1}{C}\sum_{k=1}^{n-1}\gamma^{1/2}(k) & \leq\sum_{k=1}^{n-1}(k+1)^{-a/2}\\
 & =\sum_{k=1}^{n-1}(k+1)^{-a/2}\\
 & \leq\int_{1}^{n}x^{-a/2}{\rm d}x\\
 & \leq\frac{2}{2-a}n^{1-\frac{a}{2}},
\end{align*}
from which we may deduce that if $a\in(1,2)$ then $\sum_{k=0}^{n-1}\gamma^{1/2}(k)\in\mathcal{O}(n^{1/2-\epsilon})$
for some $\epsilon>0$ and (\ref{eq:mw-cond}) holds. If $a\geq2$
then $\gamma(n)\in\mathcal{O}(n^{-b})$ for any $b\in(1,2)$ and we
can also conclude that (\ref{eq:mw-cond}) holds.
\end{proof}
\begin{rem}
\label{rem:clt-lp}One may verify that $\gamma(n)\in\mathcal{O}(n^{-a})$
by verifying a $(\Phi,\beta)$-WPI with $\beta\in\mathcal{O}(s^{-a})$;
see \cite[Lemma~15]{ALPW2021}. We observe that if $\gamma(n)\in\mathcal{O}(n^{-b})$
with $\Phi=\left\Vert \cdot\right\Vert _{{\rm osc}}^{2}$ and $b>1$,
then by Proposition~\ref{prop:cattiaux-et-al-gamma-p} we may deduce
that for $p>2$, $\left\Vert P^{n}f\right\Vert _{2}^{2}\leq\left\Vert f\right\Vert _{\mathrm{L}^{p}(\mu)}^{2}\gamma_{p}(n)$
for $f\in\mathrm{L}_{0}^{p}(\mu)$ with $\gamma_{p}(n)\in\mathcal{O}(n^{-b(1-\frac{p}{2})})$.
It then follows that a CLT holds for all $f\in\mathrm{L}_{0}^{p}(\mu)$
if $p>2b/(b-1)$. If $\gamma(n)$ decays faster than polynomially,
then a CLT holds for all $f\in\mathrm{L}_{0}^{p}(\mu)$ with $p>2$
arbitrary.
\end{rem}

\appendix

\section{Miscellaneous results and proofs \label{app:first-appendix}}
\begin{proof}[Proof of Proposition \ref{prop:cattiaux-et-al-gamma-p}]
 We follow the proof of \cite[Lemma 5.1]{Cattiaux2012}. So we choose
some $g\in\mathrm{L}_{0}^{p}(\mu)$ with $\|g\|_{p}=1$, and for $R>1$
to be chosen later, define $g_{R}:=g\wedge R\vee(-R)$, and set $m_{R}:=\int g_{R}\,\dif\mu$.
So we also obtain
\[
|m_{R}|\le\|g\|_{p}^{p}/R^{p-1}
\]
and
\[
\|g-g_{R}\|_{2}^{2}\le\|g\|_{p}^{p}/R^{p-2}.
\]
Then we bound using the fact that $P^{n}$ is a contraction on $\ELL(\mu)$,
\begin{align*}
\|P^{n}g\|_{2} & \le\|P^{n}g-P^{n}g_{R}\|_{2}+\|P^{n}(g_{R}-m_{R})\|_{2}+|m_{R}|\\
 & \le\|g-g_{R}\|_{2}+\|P^{n}(g_{R}-m_{R})\|_{2}+|m_{R}|\\
 & \le\|g\|_{p}^{p}/R^{\frac{p-2}{2}}+\gamma^{1/2}(n)\|g_{R}-m_{R}\|_{\mathrm{osc}}+\|g\|_{p}^{p}/R^{p-1}\\
 & \le1/R^{(p-2)/2}+2R\gamma^{1/2}(n)+1/R^{p-1}\\
 & \le2R\gamma^{1/2}(n)+2/R^{(p-2)/2}.
\end{align*}
Finally this can be optimized by choosing $R=2^{2/p}\gamma^{-1/p}(n)$.
The result then follows.
\end{proof}
\begin{lem}
\label{lem:phi-spectral-gap}Assume $\Phi$ defines a subspace of
$\mathrm{L}_{0}^{2}(\mu)$, $\mathcal{F}=\{f\in\mathrm{L}_{0}^{2}(\mu):\Phi(f)<\infty\}$.
Let $T$ be self-adjoint and assume that $f\in\mathcal{F}\Rightarrow Tf\in\mathcal{F}$.
Let $S$ denote the restriction of $T$ to the Hilbert space $\bar{\mathcal{F}}$,
the closure of $\mathcal{F}$. Then $\psi$ in Remark~\ref{rem:psi-alpha}
satisfies
\[
\psi(0;\Phi)={\rm Gap}_{{\rm R}}(S).
\]
 If $\Phi=\left\Vert \cdot\right\Vert _{{\rm osc}}^{2}$, then $\bar{\mathcal{F}}=\mathrm{L}_{0}^{2}(\mu)$
and $\psi(0;\Phi)$ is the $\mathrm{L}_{0}^{2}(\mu)$ spectral gap
of $T$.
\end{lem}

\begin{proof}
$\mathcal{F}$ is a normed vector space with norm $\left\Vert \cdot\right\Vert _{2}$,
and hence $\bar{\mathcal{F}}$ is a Hilbert space. We may deduce that
the restriction of $T$ to $\mathcal{F}$ is an operator from $\mathcal{F}$
to $\bar{\mathcal{F}}$, and that $S$ is its unique extension as
a bounded linear operator from $\bar{\mathcal{F}}$ to $\bar{\mathcal{F}}$.
By \cite[Theorem~22.A.19]{douc2018markov} we have $\sup_{f\in\bar{\mathcal{F}},\left\Vert f\right\Vert _{2}=1}\left\langle Sf,f\right\rangle =\sup\sigma(S)$
so that $\inf_{f\in\bar{\mathcal{F}}}\mathcal{E}(S,f)/\left\Vert f\right\Vert _{2}^{2}={\rm Gap}_{{\rm R}}(S)$. 

Now assume that $\Phi=\left\Vert \cdot\right\Vert _{{\rm osc}}^{2}$.
For any $f\in\mathrm{L}_{0}^{2}(\mu)$ we may define $f_{n}={\bf 1}_{A_{n}}\cdot f$
and $g_{n}=f_{n}-\mu(f_{n})$, where $A_{n}=\left\{ x:-n\leq f_{n}(x)\leq n\right\} $.
Then $(g_{n})$ is a sequence of bounded functions in $\mathrm{L}_{0}^{2}(\mu)$
with $g_{n}\to f$ pointwise and $\left|g_{n}\right|\leq\left|f\right|$.
We have
\[
\left|\mu(f_{n})\right|=\left|\mu(f_{n})-\mu(f)\right|\leq\mu(\left|f_{n}-f\right|)=\left\Vert f_{n}-f\right\Vert _{L^{1}(\mu)}\leq\left\Vert f_{n}-f\right\Vert _{2},
\]
from which we obtain that $\left\Vert g_{n}-f\right\Vert _{2}\leq2\left\Vert f_{n}-f\right\Vert _{2}\to0$
by dominated convergence, and hence $\bar{\mathcal{F}}=\mathrm{L}_{0}^{2}(\mu)$.
\end{proof}
\begin{rem}
\label{rem:phi-operator-norm}If $T=P^{*}P$, then $T$ is self-adjoint
and positive, and by \cite[Theorem~22.A.17 and Corollary~22.A.18]{douc2018markov}
we may further deduce that 
\[
\left\Vert S\right\Vert _{\mathcal{\bar{\mathcal{F}}\to\bar{\mathcal{F}}}}=\left\Vert R\right\Vert _{\mathcal{\bar{\mathcal{F}}\to\bar{\mathcal{F}}}}^{2}=1-\psi(0;\Phi),
\]
where $R$ is the restriction of $P$ to $\bar{\mathcal{F}}$.
\end{rem}

\begin{lem}
\label{lem:dirichlet-form-indicator}Let $P$ be a $\mu$-reversible
Markov transition kernel $P$ on $(\E,\mathscr{E})$. Then for any
$A\in\mathcal{E}$
\[
\mathcal{E}(P,\mathbf{1}_{A})=\mu\otimes P\big(A\times A^{\complement}\big)\text{ and }{\rm var}\big(\mathbf{1}_{A}\big)=\mu\otimes\mu\big(A\times A^{\complement}\big)\;.
\]
\end{lem}

\begin{proof}
Let $A\in\mathcal{E}$. By polarization, considering when $|\mathbf{1}_{A}(x)-\mathbf{1}_{A}(y)|=1\neq0$
and using the symmetry of $\mu\otimes P$ we have
\begin{align*}
\mathcal{E}(P,\mathbf{1}_{A}) & =\frac{1}{2}\int\big[\mathbf{1}_{A}(x)-\mathbf{1}_{A}(y)\big]^{2}\mu\otimes P({\rm d}x,{\rm d}y)\\
 & =\frac{1}{2}\int\big[\mathbf{1}_{A}(x)\mathbf{1}_{A^{\complement}}(y)+\mathbf{1}_{A^{\complement}}(x)\mathbf{1}_{A}(y)\big]\mu\otimes P({\rm d}x,{\rm d}y)\\
 & =\int\mathbf{1}_{A}(x)\mathbf{1}_{A^{\complement}}(y)\mu\otimes P({\rm d}x,{\rm d}y)\,.
\end{align*}
The result on the variance follows by considering $P(x,A)=\mu(A)$
for $(x,A)\in\mathsf{E}\times\mathscr{E}$ and the classical identity
${\rm var}\big(\mathbf{1}_{A}\big)=\frac{1}{2}\mathbb{E}_{\mu\otimes\mu}\left[\big(\mathbf{1}_{A}(X)-\mathbf{1}_{A}(Y)\big)^{2}\right]$.
\end{proof}
\begin{lem}[\cite{lawler1988bounds}]
\label{lem:lawler-sokal-beautiful} Let $\nu$ be a symmetric probability
measure on $(\mathsf{E}\times\mathsf{E},\mathscr{E}\otimes\mathscr{E})$.
Then for any $h\colon\mathsf{E}\times\mathsf{E}\rightarrow\mathbb{R_{+}}$
such that $h\in\mathrm{L}^{1}(\nu)$ and for any $x\in\mathsf{E}$,
$y\rightarrow h(x,y)$ is constant. Writing $h(x):=h(x,y)$ for notational
simplicity, define $A_{u}:=\{x\in\mathsf{E}\colon h(x)\leq u\}$ for
$u\geq0$ . Then we have
\begin{align}
\mathbb{E}_{\nu}\left[|h(X)-h(Y)|\right] & =2\int\nu(A_{t},A_{t}^{\complement})\,{\rm d}t\,.\label{eq:bridge-dirichlet-indicators}
\end{align}
\end{lem}

\begin{proof}
We have by symmetry of $\nu$ and Fubini,
\begin{align*}
\mathbb{E}_{\nu}\left[|h(X)-h(Y)|\right] & =2\int\int\nu({\rm d}x,{\rm d}y)\mathbf{1}\{h(x)\leq t<h(y)\}\,{\rm d}t\\
 & =2\int\nu(A_{t},A_{t}^{\complement})\,{\rm d}t\;\cdot
\end{align*}
\end{proof}

\bibliographystyle{plain}
\bibliography{bib-subgeom}

\end{document}